\documentclass[12pt,reqno]{amsart}
\usepackage{amsfonts,amsrefs,latexsym,amsmath, amssymb, mathrsfs, verbatim}
\usepackage{url,color}
\usepackage{pifont}
\usepackage{upgreek}
\usepackage{fancyhdr}
\usepackage{hyperref}
\usepackage{calligra}
\usepackage{marvosym}
\usepackage[percent]{overpic}
\usepackage{pict2e}
\usepackage[hypcap=false]{caption}
\usepackage{xcolor}
\usepackage{wasysym}
\usepackage{accents}

\topmargin - .23 in

\newtheorem{theorem}{Theorem}[section]
\newtheorem{proposition}{Proposition}[section]
\newtheorem{lemma}[proposition]{Lemma}

\theoremstyle{definition}
\newtheorem{definition}{Definition}[section]
\newtheorem{remark}{Remark}[section]

\theoremstyle{plain}

\DeclareMathAlphabet{\mathcalligra}{T1}{calligra}{m}{n}
\DeclareFontShape{T1}{calligra}{m}{n}{<->s*[2.2]callig15}{}

\newcommand{\Tboot}{T_{(Boot)}}

\newcommand{\Trandatasize}[1]{\mathring{\updelta}}

\newcommand{\TrandatasizeWithFactor}{\mathring{\updelta}_*}

\newcommand{\Flatdiv}{\mbox{\upshape div}\mkern 1mu}
\newcommand{\Flatcurl}{\mbox{\upshape curl}\mkern 1mu}

\newcommand{\D}{\mathscr{D}}

\newcommand{\Speed}{c_s}

\newcommand{\Transport}{B}

\newcommand{\Densrenormalized}{\uprho}

\newcommand{\Vortrenormalized}{\upomega}

\newcommand{\Fullset}{\mathscr{Z}}
\newcommand{\Tanset}{\mathscr{P}}

\newcommand{\Singletan}{P}

\newcommand{\Lgeo}{L_{(Geo)}}
\newcommand{\Lunit}{L}
\newcommand{\uLunit}{\underline{L}}
\newcommand{\uLgood}{\breve{\underline{L}}}
\newcommand{\Radunit}{X}
\newcommand{\Rad}{\breve{X}}

\newcommand{\Mult}{T}
\newcommand{\GeoAng}{Y}

\newcommand{\deform}[1]{{^{(#1)} \mkern-1mu \pi}}
\newcommand{\deformarg}[3]{{^{(#1)} \mkern-1mu \pi_{#2 #3}}}

\newcommand{\smoothfunction}{\mathrm{f}}

\newcommand{\enmomtensor}{Q}

\newcommand{\f}{\frac}
\newcommand{\rd}{\partial}

\newcommand{\bX}{\Rad}

\newcommand{\vt}{\vartheta}

\numberwithin{equation}{subsection}

\begin{document}
\title{The Hidden Null Structure of the Compressible Euler Equations and a Prelude
to Applications
}
\author{Jonathan Luk$^{*}$ and Jared Speck$^{** \dagger}$}

\thanks{$^{\dagger}$JS gratefully acknowledges support from NSF grant \# DMS-1162211,
from NSF CAREER grant \# DMS-1454419,
from a Sloan Research Fellowship provided by the Alfred P. Sloan foundation,
and from a Solomon Buchsbaum grant administered by the Massachusetts Institute of Technology.
}

\thanks{$^{*}$Stanford University, Palo Alto, CA, USA.
\texttt{jluk@stanford.edu}}

\thanks{$^{**}$Massachusetts Institute of Technology, Cambridge, MA, USA.
\texttt{jspeck@math.mit.edu}}

\begin{abstract}
We derive a new formulation of the compressible Euler equations
exhibiting remarkable structures, 
including surprisingly good null structures.
The new formulation comprises covariant wave equations for 
the Cartesian components of the velocity and the logarithmic density
coupled to a transport equation for the specific vorticity,
defined to be vorticity divided by density.
The equations allow one to use the full power of the geometric vectorfield method in treating the ``wave part'' of the system.

A crucial feature of the new formulation
is that all derivative-quadratic inhomogeneous terms 
verify the strong null condition.
The latter is a nonlinear condition
signifying the complete absence of nonlinear interactions 
involving more than one differentiation in a direction
transversal to the acoustic characteristics.
Moreover, the same good structures are found in the equations 
verified by the Euclidean divergence and curl of the specific vorticity.
This is important because one needs to combine estimates for the divergence and curl
with elliptic estimates to obtain sufficient regularity for
the specific vorticity, whose derivatives appears as inhomogeneous terms
in the wave equations.

The above structures collectively open the door for our forthcoming results:  
exhibiting a stable regime in which initially smooth solutions develop a shock singularity
(in particular the first Cartesian coordinate partial derivatives of the velocity and density blow up)
while, relative to a system of geometric coordinates
adapted to the acoustic characteristics,
the solution (including the vorticity) 
remains many times differentiable, all the way up to the shock.
The good null structures, which are often associated with global solutions,
are in fact key to proving that the shock singularity forms.
Our secondary goal in this article is to overview the central role that the 
structures play in the proof.

\bigskip

\noindent \textbf{Keywords}:
characteristics;
eikonal equation;
eikonal function;
null condition;
null hypersurface;
null structure;
singularity formation;
vectorfield method;
vorticity;
wave breaking
\bigskip

\noindent \textbf{Mathematics Subject Classification (2010)} 
Primary: 35L67 - Secondary: 35L05, 35Q31, 76N10

\end{abstract}
\maketitle

\centerline{\today}

\tableofcontents
\setcounter{tocdepth}{2}

\section{Introduction}
\label{S:INTRO}
In this article, we study the compressible Euler equations
for a perfect fluid in three spatial dimensions
under a barotropic \emph{equation of state}, that is, 
when the pressure $p$ is a function of the density $\rho$:
\begin{align} \label{E:BAROTROPICEOS}
	p = p(\rho).
\end{align}
In this setting, the compressible Euler equations
are evolution equations for the velocity $v:\mathbb{R}^{1+3} \rightarrow \mathbb{R}^3$
and the density $\rho:\mathbb{R}^{1+3} \rightarrow [0,\infty)$. Our main result
in this paper is a reformulation of the equations as a coupled system 
of (quasilinear) wave and transport equations with inhomogeneous 
terms exhibiting remarkable structures.
As we will show in \cites{jLjS2016b,jLjS2017}, 
this allows for a precise mathematical understanding of 
the formation of shock singularities
in the presence of vorticity,
starting from regular initial conditions.

\subsection{Basic background}
\subsubsection{Definitions}
Before stating the equations, we first provide some definitions.
We use the following notation\footnote{See Subsect.~\ref{SS:NOTATION}
regarding our conventions for indices and implied summation.} 
for the Euclidean 
divergence and curl of a $\Sigma_t-$tangent vectorfield $V$,
where $\Sigma_t$ denotes the hypersurface of constant Cartesian time $t$:
\begin{align} \label{E:FLATDIVANDCURL}
		\Flatdiv V
		& := \partial_a V^a,
			\qquad
		(\Flatcurl V)^i
		:= \epsilon_{iab} \partial_a V^b.
	\end{align}
	In \eqref{E:FLATDIVANDCURL}, $\epsilon_{ijk}$ is the fully antisymmetric
	symbol normalized by
	\begin{align} \label{E:EPSILONNORMALIZATION}
	\epsilon_{123} = 1.
\end{align}
The vorticity 
$\omega: \mathbb{R}^{1+3} \rightarrow \mathbb{R}^3$ 
is the vectorfield
\begin{align} \label{E:VORTICITYDEFINITION}
	\omega^i 
	& := (\Flatcurl v)^i.
\end{align}
Rather than formulating the equations in terms of the
density and the vorticity,
we find it convenient to use the
\emph{logarithmic density} $\Densrenormalized$
and the \emph{specific vorticity} $\Vortrenormalized$.

To define these quantities, we first fix a constant background density $\bar{\rho}$ such that
\begin{align}  \label{E:BACKGROUNDDENSITY}
\bar{\rho} > 0.
\end{align}
In applications, one may choose any
convenient value\footnote{For example, when studying solutions
that are perturbations of non-vacuum constant states,
one may choose $\bar{\rho}$ so that 
in terms of the variable $\Densrenormalized$ from \eqref{E:MODIFIEDVARIABLES},
the constant state corresponds to $\Densrenormalized \equiv 0$.} 
of $\bar{\rho}$.

\begin{definition}[\textbf{Some convenient solution variables}]
\label{D:MODIFIEDVARIABLES}
\begin{align} \label{E:MODIFIEDVARIABLES}
	\Densrenormalized
	& := \ln \left(\frac{\rho}{\bar{\rho}} \right),
		\qquad
	\Vortrenormalized
	:= \frac{\omega}{(\rho/\bar{\rho})}
	= \frac{\omega}{\exp \Densrenormalized}.
\end{align}
\end{definition}

We assume throughout that\footnote{Throughout 
this article, we avoid discussing the dynamics in regions with vanishing density. The reason is that
the compressible Euler equations become degenerate along fluid-vacuum boundaries 
and not much is known about compressible fluid flow in this context; see, for example, \cite{dCsS2012}
for more information.}
\begin{align} \label{E:DENSITYPOSITIVE}
	\rho > 0.
\end{align}
In particular, the variable $\Densrenormalized$
(see \eqref{E:MODIFIEDVARIABLES})
is finite assuming \eqref{E:BACKGROUNDDENSITY} and \eqref{E:DENSITYPOSITIVE}.

\subsubsection{A standard first-order formulation of the compressible Euler equations}
We now state a standard formulation of the compressible Euler equations;
see, for example, \cite{dCsM2014} for a discussion of the physical origin 
of the equations. Specifically, relative to Cartesian coordinates, the compressible Euler equations
can be expressed\footnote{Throughout, if $V$ is a vectorfield
and $f$ is a function, then $Vf := V^{\alpha} \partial_{\alpha} f$
denotes the derivative of $f$ in the direction $V$.
\label{F:VECTORFIELDSACTONFUNCTIONS}} 
as follows:
\begin{subequations}
\begin{align} \label{E:TRANSPORTDENSRENORMALIZEDRELATIVETORECTANGULAR}
	\Transport \Densrenormalized
	& = - \Flatdiv v,
		\\
	\Transport v^i 
	& = 
	- \Speed^2 \partial_i \Densrenormalized
	=
	- \Speed^2 \delta^{ia} \partial_a \Densrenormalized.
	\label{E:TRANSPORTVELOCITYRELATIVETORECTANGULAR}
\end{align}
\end{subequations}
Above, $\delta^{ia}$ is the standard Kronecker delta,
\begin{align} \label{E:MATERIALVECTORVIELDRELATIVETORECTANGULAR}
	\Transport 
	& := \partial_t + v^a \partial_a
\end{align}
is the \emph{material derivative vectorfield},
and
\begin{align} \label{E:SOUNDSPEED}
	\Speed 
	& := \sqrt{\frac{dp}{d \rho}}
\end{align}
is a fundamental quantity known as the \emph{speed of sound}.
From now on, we view
\begin{align} \label{E:SPEEDOFSOUNDISAFUNCTIONOFRENORMALIZEDDENSITY}
	\Speed 
	& = \Speed(\Densrenormalized),
\end{align}
and, for future use, we set
\begin{align} \label{E:DEFINITIONOFSPEEDPRIME}
	\Speed'
	= \Speed'(\Densrenormalized)
	& := \frac{d}{d \Densrenormalized} \Speed.
\end{align}

\subsection{Summary of the main results and preliminary discussion}
\label{SS:PRELIMINARYDISCUSSION}
Note that neither the vorticity $\omega$ nor the specific vorticity $\Vortrenormalized$ appear in 
the system \eqref{E:TRANSPORTDENSRENORMALIZEDRELATIVETORECTANGULAR}-\eqref{E:TRANSPORTVELOCITYRELATIVETORECTANGULAR}.
However, $\Vortrenormalized$ plays
a central role in the main results of the present article, which we now summarize. 
We refer the readers to 
Theorem~\ref{T:GEOMETRICWAVETRANSPORTSYSTEM} on pg.~\pageref{T:GEOMETRICWAVETRANSPORTSYSTEM}
and Theorem~\ref{T:STRONGNULL} on pg.~\pageref{T:STRONGNULL} for the precise statements.

\begin{changemargin}{.25in}{.25in} 
\textbf{Summary of the main results.}
	The compressible Euler equations
	can be reformulated as a system of covariant wave equations for 
	the Cartesian components $\lbrace v^i \rbrace_{i=1,2,3}$ of the velocity 
	and the logarithmic density $\Densrenormalized$
	coupled to a transport equation for the specific vorticity $\Vortrenormalized$,
	a transport equation for $\Flatcurl \Vortrenormalized$,
	and an identity for $\Flatdiv \Vortrenormalized$.
	Moreover, the inhomogeneous terms exhibit remarkable structures, 
	\emph{including good null structures that can be viewed as extensions of the standard null forms adapted to the acoustical metric $g$}.
\end{changemargin} 

It is well-known since the foundational work of Riemann \cite{bR1860} in one spatial dimension
that solutions to the compressible Euler equations can form shocks in finite time, 
even if the initial data are smooth and small. This occurs in spite of the fact that
solutions enjoy a conserved energy.
That is, the energy is supercritical, even in one spatial dimension, and does not prevent singularity formation. 
The formation of shocks is connected in part to the failure of a null condition;
see Subsect.~\ref{SS:STRONGNULLVSOTHERNULL} for further discussion about ``null conditions.''
Put differently, 
there are Riccati-type interaction terms in the 
equations satisfied by the solution's first derivatives,
and these terms can drive the formation of a singularity tied to the intersection of characteristics. 
More precisely, the Riccati-type terms drive the blowup of the first Cartesian coordinate partial derivatives of the
density and velocity, while the velocity and density themselves remain bounded; this is the crudest 
picture of the formation of a shock singularity.
On the other hand, in our new formulation of the equations, 
all of the terms that violate the null condition are 
on the left-hand side of the equations,
``hidden ''
in the terms  $\square_g v^i$ and $\square_g \Densrenormalized$ (see \eqref{E:VELOCITYWAVEEQUATION}	and \eqref{E:RENORMALIZEDDENSITYWAVEEQUATION}),
where $\square_g$ is a covariant wave operator (see Def.~\ref{D:COVWAVEOP}).
We derive the new formulation by differentiating the system 
\eqref{E:TRANSPORTDENSRENORMALIZEDRELATIVETORECTANGULAR}-\eqref{E:TRANSPORTVELOCITYRELATIVETORECTANGULAR}
with suitable operators and observing cancellations.
In more than one spatial dimension, the approach of hiding the difficult terms 
in the operator $\square_g$
turns out to be crucial for understanding the formation of the shock.
Put differently, if one writes the wave equations in divergence form,
then all explicitly written inhomogeneous terms satisfy a null condition
(distinct from the one mentioned earlier in this paragraph), 
\emph{even in the presence of vorticity}!
We devote all of Sect.~\ref{S:PROOFOFTHMSTRONGNULL} to discussing this null condition
and its relation to other null conditions that have appeared in the literature.

A similar -- but much simpler -- structure had previously been found by Christodoulou--Miao \cite{dCsM2014} 
in their proof of shock formation in irrotational
(that is, vorticity free) regions. In the irrotational case, the dynamics reduces
to a single quasilinear wave equation for the fluid potential.\footnote{The fluid potential $\Phi$ is defined such that $\rd_i\Phi=-v^i$} In fact, as we further explain in Remark~\ref{rmk.irrot},
for irrotational solutions, our wave equation \eqref{E:VELOCITYWAVEEQUATION}
for the velocity $v^i$ follows as an implicit consequence of the calculations in \cite{dCsM2014}. 
More precisely, Christodoulou--Miao showed that appropriately defined variations\footnote{Roughly, a variation of
$\Phi$ is the derivative of $\Phi$ with respect to some first-order differential operator.} 
of the fluid potential satisfy \emph{homogeneous} covariant quasilinear wave equations and thus
\emph{all} of their nonlinearities are hidden in the covariant wave operator $\square_g$; see Subsect. \ref{SSS:PRELIMINARYDICUSSION}
for further discussion. We note that the first observation and use of this kind of good structure 
was made by Christodoulou in his breakthrough work on (small-data) 
shock formation \cite{dC2007} for solutions to relativistic Euler equations in $(1+3)$ dimensions in irrotational regions.

There is a long, rich history of prior result leading up to the works \cites{dC2007,dCsM2014}
and their recent extensions. Readers may consult the
survey article \cite{gHsKjSwW2016} for more details; here we 
discuss only the works that are most relevant for the present article.
Alinhac was the first \cites{sA1999a,sA1999b,sA2001b,sA2002}
to prove shock formation results
for hyperbolic PDEs in more than one spatial dimension without symmetry assumptions.
Specifically, in two and three spatial dimensions,
he proved small-data shock formation results
for wave equations of the form
$
(g^{-1})^{\alpha \beta} (\partial \Phi) \partial_{\alpha} \partial_{\beta} \Phi = 0
$
whenever the nonlinear terms
fail to satisfy Klainerman's null condition 
(which we describe in more detail in Subsubsect.~\ref{sss.small.data}).
More precisely, Alinhac exhibited a set of small data such that 
$\partial^2 \Phi$ blows up in finite time due to the intersection of characteristics,
while $\Phi$ and $\partial \Phi$ remain bounded, all the way up to the singularity.
After appropriate renormalizations, it may be seen that
all wave equations treated by Christodoulou \cite{dC2007}
and Christodoulou--Miao \cite{dCsM2014} essentially fall under the scope
of Alinhac's work. However, the approach developed by Christodoulou in \cite{dC2007}
was a big advancement over that of Alinhac for the following main reasons.
\begin{itemize}
	\item For the wave equations of irrotational compressible fluid mechanics,
		Christodoulou and Christodoulou--Miao
		gave a fully geometric description of 
		the singularity formation that, for small data, exactly ties singularity
		formation to the intersection of characteristics. That is, 
		unlike Alinhac's framework, Christodoulou's yields that
		shocks are the only possible kinds of singularities that can in principle occur
		when the data are small.
	\item Christodoulou's framework yields sharp information about
	 	the maximal classical development\footnote{Roughly, the maximal classical development 
		is the largest possible classical solution that is uniquely determined by the data; see, for example, 
		\cites{jSb2016,wW2013} for further discussion.} 
	 	of the data, including the behavior of the solution up to the boundary.
	 	As is described in \cites{dC2007,dCaL2016}, this information is essential 
	 	even to properly set up the shock development problem, which
	 	is the problem of weakly continuing the solution past the singularity.
	 	In contrast, due to fundamental technical limitations tied to his use of a Nash-Moser 
	 	energy estimate framework, Alinhac's proof breaks down precisely 
	 	at the time of first blowup and cannot be extended to yield
	 	information about the boundary of the maximal development.
	 	For similar reasons, 
	 	Alinhac's proof relies on a non-degeneracy assumption on the initial 
	 	data that ensures that there is a unique blowup point
	 	in the constant-time hypersurface of first blowup.
	 \item Many features of Christodoulou's framework
	 	are robust\footnote{We note, however, that shock formation results
	 	seem to be significantly less stable 
	 	than small-data global existence proofs
	 	under perturbations of the \emph{equations};
	 	we explore this in detail in Subsect.~\ref{SS:STRONGNULLVSOTHERNULL}.} 
	 	and have the potential to be applied to 
	 	other equations.
\end{itemize}
In view of the above remarks, it is clear why
Christodoulou's approach to proving shock formation in the irrotational case
served as the starting point for our study of shock formation in the presence of vorticity.

Another seed idea was found in the work \cite{jS2014b}, 
in which Speck proved shock formation results similar 
to those of \cites{dC2007,dCsM2014} 
for a large class of quasilinear wave equations, 
which are not necessarily homogeneous, 
\emph{as long as the inhomogeneous terms satisfy a null condition}. 
Roughly speaking, the null condition from \cite{jS2014b}
allows one to show that the nonlinear inhomogeneous terms 
do not interfere with the shock formation mechanisms and that for small compactly supported data,
no other kinds of singularities can occur prior to shock formation. 
We clarify that the null condition from \cite{jS2014b}
is visible when the wave equations are written in covariant form; if one
expresses the wave equations relative to the standard Cartesian coordinates,
then the nonlinear terms \emph{fail} to satisfy Klainerman's classic null
condition; see Subsect.~\ref{SS:STRONGNULLVSOTHERNULL} for further discussion
of the different null conditions.
Our new formulation of the compressible Euler equations, 
for which an \underline{extended notion} of such a null condition is also verified, 
\emph{even in the presence of vorticity}, 
opens the door to our forthcoming results \cites{jLjS2016b,jLjS2017}: 
showing that there is an open (relative to an appropriate Sobolev topology) 
set of regular initial data such that the solution forms a shock in finite time. 
The main novel feature of our works \cites{dC2007,dCsM2014}   
is that the vorticity is \emph{not required vanish at the shock}. 
That is, we have to control the vorticity in a neighborhood of the first singularity caused
by compression. In particular, we must rule out the onset of ``wild instabilities''
that could in principle be caused by the interaction of vorticity flow and shocks.
We summarize our results as follows.

\begin{changemargin}{.25in}{.25in} 
\textbf{Summary of forthcoming results \cites{jLjS2016b, jLjS2017}.}
In two or three spatial dimensions,\footnote{Since we will be considering solutions on spatial domains $\mathbb R \times \mathbb T$ and $\mathbb R\times \mathbb T^2$ in two and three spatial dimensions respectively, the result in two spatial dimensions is strictly weaker than the one in three spatial dimensions. 
However, since the two-spatial-dimensional-case 
contains substantial new ideas compared to the irrotational case
but is technically simpler than the three-spatial-dimensional-case,
we have treated it separately in \cite{jLjS2016b};
we hope that the two-space-dimensional result
will serve as a useful starting point
for readers interested in the case of three spatial dimensions.}
for any physical equation of state except that 
of the Chaplygin gas,\footnote{A Chaplygin gas has the equation of state $p = p(\rho)=C_0-\f{C_1}{\rho}$,
where $C_0\in \mathbb R$ and $C_1>0$.}
there exists an \underline{open} set of regular initial data,
with elements close to the data of a subset of simple plane wave solutions, 
that leads to \underline{stable} finite-time shock formation. 
The specific vorticity, which is \underline{provably non-vanishing at the shock} for some of our solutions, 
remains uniformly bounded, all the way up to the shock.
Moreover, the dynamics are ``well-described'' by the irrotational Euler equations. 
\end{changemargin}

All prior blowup results for the compressible Euler equations that allow for
non-zero vorticity are either non-constructive in nature or 
are such that the potential formulation of the Euler equations was used near the shock because the 
vorticity was \emph{provably non-zero there}; 
see \cites{tS1985,sA1993,dC2007,dCsM2014,jS2014b,jSgHjLwW2016}. 
We refer the readers to the survey paper \cite{gHsKjSwW2016} as well as our companion paper 
\cite{jLjS2016b} for further discussions on related previous works.
Importantly, we note that our work may be relevant 
for the shock development problem,\footnote{This problem was recently solved
in the spherically symmetric relativistic case \cite{dCaL2016}
starting from the state of a spherically symmetric irrotational
solution at the end of its classical lifespan, 
which was obtained by Christodoulou in \cite{dC2007} as a special case.
Away from spherical symmetry, the shock development problem 
remains open and is expected to be quite difficult.} 
that is, 
the problem of continuing the solution as a weak solution 
after a shock has formed. The reason is that upon
weakly continuing the solution,
vorticity is typically generated across the shock
hypersurface \cite{dC2007}, even if the solution is 
irrotational up to the time of first blowup. 

In our forthcoming works \cites{jLjS2016b,jLjS2017}, 
we are able to give a complete description of 
the behavior of the solution up to the onset of the first shock, 
including that of the vorticity. In particular,
relative to a system of geometric coordinates
adapted to the acoustic characteristics
(which are null hypersurfaces corresponding to sound wave propagation), 
\emph{the solution, including the vorticity, 
remains many times differentiable, 
all the way up to the shock}. Moreover, in the case of two spatial dimensions, our methods can in principle also give a description of a portion of the boundary of the maximal classical development of the solution, at least for a subclass of solutions
verifying non-degeneracy conditions of the type assumed in \cites{dC2007,dCsM2014}.\footnote{In three spatial dimensions, it remains unclear whether our methods can be extended to yield a sharp description of the
boundary of the maximal development. The main difficulty is technical in nature and is tied to
our reliance on elliptic estimates on constant-$t$ hypersurfaces
to control the top-order derivatives of the specific vorticity.}

\subsection{Preliminary overview of the role of the present work in our forthcoming proofs of shock formation}
Our secondary goal in this paper is to overview 
our forthcoming proofs of shock formation in the presence of vorticity
and to highlight the role played by our new formulation of the equations.
For reasons to be explained, we will treat the case of two and three spatial dimensions
in separate works.
Our proofs are based in part on the framework 
developed by Christodoulou \cite{dC2007} and Christodoulou--Miao\cite{dCsM2014}
in their study of shock formation in three spatial dimensions in
the irrotational case,\footnote{Strictly speaking, the solutions
in \cite{dC2007} contained vorticity. 
However, the initial conditions studied in \cites{dC2007}
led to the vorticity being confined to a region far away from the shock. 
Hence, Christodoulou did not have to confront the difficult problem 
of having to control the vorticity at the shock itself.}
on an extended version of the notion of good null structure observed in \cite{jS2014b},
and on the framework of \cite{jSgHjLwW2016},
in which the authors extended Christodoulou's results\footnote{The results of \cite{jSgHjLwW2016}
apply to a large class of wave equations, of which the 
irrotational compressible Euler equations are a special example.} 
to a new solution regime in which
the solutions are close to simple plane symmetric waves.
To control the vorticity up to the singularity, 
we exploit all of the geo-analytic structures 
revealed by Theorems~\ref{T:GEOMETRICWAVETRANSPORTSYSTEM} and \ref{T:STRONGNULL},
structures which are compatible with an extended version of Christodoulou's framework.

We now briefly summarize our new formulation 
of the compressible Euler equations
and its implications for the study of shock formation.
We revisit these issues in extended detail in Sect.~\ref{S:IDEASBEHINDSHOCKFORMATION}.
\begin{enumerate}
\item \textbf{(New formulation of the equations)}
	 The system comprises \emph{covariant} wave equations
		for the Cartesian components  
		$\lbrace v^i \rbrace_{i=1,2,3}$
		of the velocity
		and the logarithmic density $\Densrenormalized$
		coupled to a transport equation for the specific vorticity $\Vortrenormalized$;
		see Theorem~\ref{T:GEOMETRICWAVETRANSPORTSYSTEM}.
\item \textbf{(Structure of the inhomogeneous terms)} The system features inhomogeneous terms that can be split into three distinct classes:
		\begin{enumerate}
		\item Quadratic terms in the derivatives of $v^i$, $\Densrenormalized$ and $\Vortrenormalized^i$ obeying the \emph{strong null condition} 
		relative to the acoustical metric 
		$g = g(\Densrenormalized,v)$
		(which is the Lorentzian metric corresponding to the propagation of sound wave, 
		see Def.~\ref{D:ACOUSTICALMETRIC}). We exhibit the good null structure enjoyed by these terms in Theorem~\ref{T:STRONGNULL}.
		\item Products that are linear in $\lbrace \partial v^i \rbrace$
		or $\partial \Densrenormalized$, where $\partial$ denotes the spacetime gradient with respect to the 
		Cartesian coordinates.
		\item Products that are linear in 
		$\underline{\partial} \Vortrenormalized$,
		where $\underline{\partial}$ denotes the spatial gradient with respect to the 
		Cartesian coordinates.
		\end{enumerate}
		Importantly, these good structures also hold for 
		$\Flatdiv \Vortrenormalized$ and $\Flatcurl \Vortrenormalized$, 
		where $\Flatdiv$ and $\Flatcurl$ are the usual Euclidean operators.
		That is, these terms
		satisfy equations with inhomogeneous terms enjoying the \emph{same structure}
		highlighted above, which we need in order to obtain suitable estimates for $\Flatdiv \Vortrenormalized$
		and $\Flatcurl \Vortrenormalized$.
		Moreover, even though a general Cartesian spatial derivative $\underline{\partial} \Vortrenormalized$ does 
		\emph{not} obey an equation with such a good structure, 
		the information that one can obtain for
		$\Flatdiv \Vortrenormalized$ and $\Flatcurl \Vortrenormalized$ 
		(using the good structures of their equations)
		is sufficient to close, via elliptic estimates, a top-order estimate for $\Vortrenormalized$.
		Put differently, to control $\underline{\partial} \Vortrenormalized$,
		we first control the ``good terms''
		$\Flatdiv \Vortrenormalized$ and $\Flatcurl \Vortrenormalized$ 
		and then use elliptic estimates.
	\item \textbf{(Null structure is a key ingredient in the proof of shock formation)} 
		As we have mentioned, in \cite{dC2007}, Christodoulou introduced a 
		foundational geometric framework for proving shock formation in regions with vanishing vorticity.
		The main ingredient in his approach was an \emph{eikonal function} $u$,
		which solved the eikonal equation
		$(g^{-1})^{\alpha \beta} \partial_{\alpha} u \partial_{\beta} u=0$
		(see \eqref{E:EIKONAL})
		corresponding to the acoustical metric $g$, which also appears 
		in the (quasilinear) covariant wave operator $\square_g$ mentioned above.
		Christodoulou completed $u$ and the standard Cartesian time function $t$
		to new set of geometric coordinates 
		on spacetime and constructed related geometric vectorfields,
		adapted to the acoustic characteristics (that is, the level sets of $u$, which are $g$-null hypersurfaces).
		On the one hand, these geometric vectorfields can degenerate with respect to 
		the Cartesian\footnote{Actually, in the context of \cite{dC2007}, in which
		the physical spacetime was Minkowski spacetime, 
		it would be more accurate to
		refer to these coordinates as ``Minkowski rectangular coordinates.''} 
		coordinate partial 
		vectorfields as the shock forms.
		However, the big gain is that the 
		nonlinearities in the equation exhibit good null structure 
		when decomposed relative to the geometric coordinates and/or vectorfields.
		By taking advantage of this null structure,
		Christodoulou \cite{dC2007} and Christodoulou--Miao \cite{dCsM2014}
		were able to prove that relative to the geometric coordinates,
		the solution remains many times differentiable.
		At the same time, they proved that for an open set of data, 
		in finite time, 
		the geometric coordinates \emph{degenerate} relative to the Cartesian ones,
		which causes a singularity in the Cartesian coordinate partial derivatives of the solution.
		One might say that the geometric coordinates ``hide'' the singularity,
		which allows one to prove the estimates needed to show that the singularity 
		does in fact form.

	As the above discussion has suggested, in order to extend Christodoulou's framework to allow for the presence of vorticity, 
	it is important that the inhomogeneous terms should have a certain good null structure. 
	We formulate this precisely in Sect.~\ref{S:PROOFOFTHMSTRONGNULL},
	where we refer to the good null structure as the ``strong null condition.''
	The definition of the strong null condition 
	is motivated by the result that one eventually proves: one shows
	that the singularity forms in a derivative of the solution 
	in a direction transversal to the 
	acoustic characteristics, while the solution's tangential derivatives remain uniformly bounded.
	Hence, the good structure enjoyed by an inhomogeneous term product verifying
	the strong null condition is roughly the following:
	it contains at most one differentiation
	in a direction transversal to the acoustic characteristics.
	Near the singularity, such terms are \emph{weaker} 
	than the Riccati-type interactions 
	(which are quadratic in the solution's transversal derivatives)
	that are hidden in the action of the wave operator $\square_g$.
	Hence, these good terms do not interfere with the shock formation mechanisms.
	We stress that \emph{the strong null condition is fully nonlinear
	in nature and is not based on Taylor expanding nonlinearities to quadratic order}.
	We will explain the necessity of such a structure in 
	Subsubsect.~\ref{SSS:SOMEREMARKSONTHEPROOF}.
\item \textbf{(Null structure and top-order singular estimates)} 
		In quasilinear hyperbolic PDEs 
		in one spatial dimension, such as Burgers' equation, 
		it is often easy to construct
		geometric coordinates such that with respect to 
		these coordinates, the solution remains as regular as the initial data,
		all the way up to the singularity (see, for example, Footnote~\ref{FN:BURGER}).
		However, Christodoulou's framework is fraught with a challenging technical difficulty:
		even with the help of the geometric coordinates,
		it does not seem possible to ``hide'' the singularity 
		at very high derivative levels;\footnote{The same difficulty is found in Alinhac's approach
		\cites{sA1999a,sA1999b,sA2001b,sA2002} to proving shock formation.} 
		in all known results on the formation of shocks in more
		than one spatial dimension, the best estimates available allow for the possibility that the
		high-order geometric energies might blow up as the singularity forms.
		However, it is not know with certainty whether or not the 
		high-order energy blowup does in fact occur.\footnote{Note that this is a different question
		than whether or not the shock forms.}
		A key part of the analysis is to control the \emph{possible} blowup-rate
		of these energies and showing, 
		with a careful order-by-order analysis, 
		that the lower-order energies remain uniformly bounded,
		all the way up to the shock.

		In our study of shock formation with vorticity, 
		we in particular need to accommodate the 
		singular high-order energy estimates for the 
		``wave part'' of the system, which are inherited from the irrotational case.
		Perhaps not surprisingly, the energy estimates for 
		the transport part of the system
		(that is, for the specific vorticity) 
		are also allowed to blow up at the high orders.
		A critically important part of the analysis is 
		understanding how the different blowup-rates for the wave and transport
		parts are tied to each other, in view of the fact that the wave and transport variables 
		are coupled at the level of the equations.
		Put differently, we need to simultaneously study the wave and transport parts of the system, 
		perform an order-by-order analysis of the singular high-order energy estimates, 
		and close a Gronwall-type energy estimate
		that accounts for the distinct singular behavior of each part of the system
		at distinct derivative levels.
		The fact that we can close such an argument is intimately tied
		to the good null structures found in the coupling terms.
		\item \textbf{(Difficulties related to multiple speeds)}
		To prove shock formation in regions with vorticity, 
	we encounter all of the same difficulties
	that Christodoulou encountered plus two challenging new ones.
	The first of these that in the presence of vorticity, the equations contain \emph{multiple speeds}:
	the speed of sound,
	which corresponds to the acoustic characteristics (which were present in Christodoulou's work)
	and the speed of vorticity transport,
	which corresponds to the integral curves (also known as the flow lines) of the material derivative vectorfield, 
	along which the specific vorticity 
	is transported. This new difficulty is present in both two and three spatial dimensions.
	In an effort to isolate the new ideas needed to handle it, 
	we prove shock formation for solutions with vorticity in two spatial dimensions in the separate work
	\cite{jLjS2016b}; see Subsubsect.~\ref{SSS:GEOMETRICVECINTERACTWITHTRANSPORT} for an overview of these new ideas.
	Analytically, the challenge is that 
	the material derivative vectorfield $\Transport$, the Euclidean divergence, and the Euclidean curl
	do not have any relationship to the geometric vectorfields needed to commute
	the wave equations, which makes it difficult to obtain estimates for
	the geometric derivatives of $\Vortrenormalized$.
	However, it turns out that the geometric vectorfields have just enough structure 
	such that their commutator with an 
	appropriately weighted, but otherwise arbitrary,
	first-order differential operator\footnote{Here, by first-order differential operator,
	we mean one equal to a regular function times a Cartesian coordinate partial derivative.} 
	produces controllable error terms, consistent with ``hiding the singularity'' 
	relative to the geometric coordinates
	at the lower derivative levels.
	This is important because 
	\emph{we found a procedure that avoids, in the evolution equation-type estimates
	for $\Vortrenormalized$, having to commute through a second-order operator}. 

To derive suitable estimates for the specific vorticity,
we crucially rely on the geometric fact that 
$\Transport$ is transversal to the acoustic characteristics. 
This basic fact allows us to derive energy estimates for
$\Vortrenormalized$ along the characteristics
in which \emph{the energies do not feature any degenerate weights,}
which is critically important for controlling error terms.
A related fact is that the transversality condition allows us to
avoid a potential logarithmic divergence; 
see Subsubsect.~\ref{sss.inho.vor} for further discussion.
However, we cannot rely on the energy of $\Vortrenormalized$ along the characteristics 
at the top order since, 
at the top order, we are forced to derive
elliptic estimates with a degenerate weight 
along constant-time hypersurfaces; see the next point.
	\item \textbf{(Top order elliptic estimates for the vorticity)}
 The second new difficulty compared to the work of Christodoulou 
 is that in the presence of vorticity, one needs
to use elliptic estimates on $\Sigma_t$ to control the top derivatives of 
$\Vortrenormalized$.
More precisely, this difficulty is present only in
three or more spatial dimension since in two spatial dimensions, the
``vorticity stretching'' term responsible for the difficulty
is absent (that is, $\mbox{{\upshape RHS}~\eqref{E:RENORMALIZEDVORTICTITYTRANSPORTEQUATION}} \equiv 0$
in two spatial dimensions). Because of the significant innovations needed
to close the elliptic estimates near the singularity,
we will prove shock formation in the case of three spatial dimensions in a separate work.
In particular, in three spatial dimensions, one must derive elliptic estimates 
for derivatives of the vorticity in a direction transversal to the acoustic characteristics, 
which, near the singularity, is a severe technical difficulty 
that is not present in the irrotational case.
In particular, when expressed relative to the geometric coordinates,
the specific vorticity energies along $\Sigma_t$ contain degenerate weights.
Ultimately, these degenerate weights contribute to the fact that 
the top-order $L^2$ estimates for $\Vortrenormalized$
can blow up as the shock forms,
much like the energy estimates in the irrotational case.
To close the proof, \emph{we must show that blowup-rate for the transport variable $\Vortrenormalized$
is not too severe}. 
In particular, we must show that the blowup-rate 
is compatible with the corresponding blowup-rates for the wave variables,
whose top-order energies, as it turns out, 
are no more singular than they are in 
the irrotational case; see Subsubsect.~\ref{SSS:ELLIPTICESTIMATESFORVORTICITY}
for an overview of the main new ideas behind these elliptic estimates.
\end{enumerate}

\subsection{Paper outline} 
\label{SS:OUTLINE}
In Sect.~\ref{sec.main.theorem},
we provide definitions and state the two main theorems.
In Sect.~\ref{S:PROOFOFTHMSTRONGNULL},
we discuss some basic concepts from Lorentzian geometry and
prove the second theorem, which
exhibits the good null structure
enjoyed by the inhomogeneous terms in the equations.
We also compare and contrast these good null structures 
to different null structures found in the literature.
In Sect.~\ref{S:IDEASBEHINDSHOCKFORMATION},
we provide provide a preview on how the new formulation of the equations
can be used to prove a sharp shock-formation result 
in the presence of non-zero vorticity in three spatial dimensions. 
To provide context, we overview how to prove shock formation
in the irrotational case using a version Christodoulou's framework
adapted to the initial data that are close in spirit to the data
considered in \cites{jLjS2016b,jLjS2017}.
In Sect.~\ref{E:PROOFOFMAINTHEOREM}, 
we prove our main Theorem~\ref{T:GEOMETRICWAVETRANSPORTSYSTEM}
via a series of calculations.

\section{Statement of the main theorems}
\label{sec.main.theorem}

Our main goal in this section is to give a precise statement of the two main theorems. Before stating them, however, we will first introduce appropriate notations, as well as some basic geometric constructions necessary for the statements of the theorems.

\subsection{Notation}
\label{SS:NOTATION}
Throughout $\lbrace x^{\alpha} \rbrace_{\alpha=0,1,2,3}$ denotes
the usual Cartesian coordinate system on 
$\mathbb{R} \times \mathbb{R} \times \mathbb{T}^2$.
More precisely, $x^0 \in \mathbb{R}$ is the time coordinate and $(x^1,x^2,x^3) \in \mathbb{R} \times \mathbb{T}^2$
are spatial coordinates.
$
\displaystyle
\partial_{\alpha} 
:=
\frac{\partial}{\partial x^{\alpha}}
$
denotes the corresponding coordinate partial derivative vectorfields.
We often use the alternate notation $x^0 = t$ and $\partial_0 = \partial_t$.
Lowercase Greek ``spacetime'' indices such as $\alpha$ vary over $0,1,2,3$ while
lowercase Latin ``spatial'' indices such as $a$ vary over $1,2,3$. 
In later sections, we will use the convention that uppercase
Greek indices, associated to the array of solution functions, vary over $0,1,\dots,6$
and upper case Latin indices, associated to null frames, vary over $1,2,3,4$.
We use Einstein's summation convention in that repeated indices are summed
over their respective ranges. 
$\Sigma_t$ denotes the usual flat hypersurface of constant time $t$.

\subsection{Preliminary ingredients in the new formulation of the equations}
\label{SS:PRELIMINARYINGREDIENTS}

\subsubsection{Assumptions on the equation of state}
\label{SSS:EOS}
We make the following physical assumptions,
which ensure the hyperbolicity of the system when
$\rho > 0$:
\begin{itemize}
	\item $\Speed \geq 0$ .
	\item $\Speed > 0$ when $\rho > 0$.
\end{itemize}

\subsubsection{Geometric tensorfields associated to the flow}
\label{SSS:GEOMETRICTENSORFIELDS}
Roughly, there are two kinds of motion associated to compressible Euler flow:
the transporting of vorticity and the
propagation of sound waves. We now discuss the tensorfields associated 
to these phenomena.

The material derivative vectorfield $\Transport$, defined in \eqref{E:MATERIALVECTORVIELDRELATIVETORECTANGULAR}, 
is associated to the transporting of vorticity. We now define the Lorentzian metric $g$ corresponding to the propagation of sound waves.

\begin{definition}[\textbf{The acoustical metric and its inverse}] 
\label{D:ACOUSTICALMETRIC}
We define the \emph{acoustical metric} $g$ and the 
\emph{inverse acoustical metric} $g^{-1}$ relative
to the Cartesian coordinates as follows:
\begin{subequations}
	\begin{align}
		g 
		& := 
		-  dt \otimes dt
			+ 
			\Speed^{-2} \sum_{a=1}^3(dx^a - v^a dt) \otimes (dx^a - v^a dt),
				\label{E:ACOUSTICALMETRIC} \\
		g^{-1} 
		& := 
			- \Transport \otimes \Transport
			+ \Speed^2 \sum_{a=1}^3 \partial_a \otimes \partial_a.
			\label{E:INVERSEACOUSTICALMETRIC}
	\end{align}
\end{subequations}
\end{definition}

\begin{remark}
	\label{R:GINVERSEISTHEINVERSE}
	It is straightforward to verify that
	$g^{-1}$ is the matrix inverse of $g$, that
	is, we have $(g^{-1})^{\mu \alpha} g_{\alpha \nu} = \delta_{\nu}^{\mu}$,
	where $\delta_{\nu}^{\mu}$ is the standard Kronecker delta.
\end{remark}

\begin{remark}
	Other authors have defined the acoustical metric to be $\Speed^2 g$.
	We prefer our definition because it implies that $(g^{-1})^{00} = - 1$,
	which simplifies the presentation of many formulas.
\end{remark}

The vectorfield $\Transport$ enjoys some simple but important
geometric properties, which we provide in the next lemma.
\begin{lemma}[\textbf{Basic geometric properties of} $\Transport$]
	\label{L:BASICPROPERTIESOFTRANSPORT}
	$\Transport$ is timelike, 
	future-directed,\footnote{A vectorfield $V$ is future directed if 
	$V^0 > 0$, where $V^0$ is $0$ Cartesian component.
	\label{FN:FUTUREDIRECTED}} 
	$g-$orthogonal to $\Sigma_t$,
	and unit-length:\footnote{Throughout we use the notation $g(V,W) := g_{\alpha \beta}V^{\alpha} W^{\beta}$.}
	\begin{align} \label{E:TRANSPORTISUNITLENGTH}
		g(\Transport,\Transport)
		& = - 1.
	\end{align}
\end{lemma}

\begin{proof}
	Clearly $\Transport$ is future-directed.
	The identity \eqref{E:TRANSPORTISUNITLENGTH}
	(which also implies that $\Transport$ is timelike)
	follows from a simple calculation
	based on \eqref{E:MATERIALVECTORVIELDRELATIVETORECTANGULAR} and \eqref{E:ACOUSTICALMETRIC}.
	Similarly, we compute that
	$
	\displaystyle
	g(\Transport,\partial_i)
	:= g_{\alpha i}\Transport^{\alpha}
	= 0
	$
	for $i=1,2,3$, from which it follows that
	$\Transport$
	is $g-$orthogonal to $\Sigma_t$.
\end{proof}

\subsection{Statement of the main result I: Reformulation of the equations}
\label{SS:MAINRESULT.I}

We first recall the standard definition of the covariant wave operator $\square_g$.

\begin{definition}[\textbf{Covariant wave operator}]
\label{D:COVWAVEOP}
Relative to arbitrary coordinates,
the covariant wave operator $\square_g$ 
acts on scalar-valued functions $\phi$ as follows:
\begin{align} \label{E:WAVEOPERATORARBITRARYCOORDINATES}
\square_g \phi
	= \frac{1}{\sqrt{|\mbox{\upshape det} g|}}
	\partial_{\alpha}
	\left\lbrace
			\sqrt{|\mbox{\upshape det} g|} (g^{-1})^{\alpha \beta}
			\partial_{\beta} \phi
	\right\rbrace.
\end{align}
\end{definition}

Our first main result is the following theorem, which provides the new formulation
of the equations.

\begin{theorem}[\textbf{The geometric wave-transport formulation of the compressible Euler equations}]
	\label{T:GEOMETRICWAVETRANSPORTSYSTEM}
	In three spatial dimensions under a barotropic equation of state \eqref{E:BAROTROPICEOS},
	the compressible Euler equations 
	\eqref{E:TRANSPORTDENSRENORMALIZEDRELATIVETORECTANGULAR}-\eqref{E:TRANSPORTVELOCITYRELATIVETORECTANGULAR}
	imply the following system
	(see Footnote~\ref{F:VECTORFIELDSACTONFUNCTIONS})
	in $(\Densrenormalized,v^1,v^2,v^3,\Vortrenormalized^1,\Vortrenormalized^2,\Vortrenormalized^3)$,
	where the Cartesian component functions $v^i$ are 
	\textbf{treated as scalar-valued functions
	under covariant differentiation} on LHS~\eqref{E:VELOCITYWAVEEQUATION}
	and $\Transport$ is the material derivative vectorfield
	defined in \eqref{E:MATERIALVECTORVIELDRELATIVETORECTANGULAR}:
	\begin{subequations}
	\begin{align}
		\square_g v^i
		& = - \Speed^2 \exp(\Densrenormalized)  (\Flatcurl \Vortrenormalized)^i 
				+
				2 \exp(\Densrenormalized) \epsilon_{iab} (\Transport v^a) \Vortrenormalized^b
			+ \mathscr{Q}^i,
			\label{E:VELOCITYWAVEEQUATION}	\\
	\square_g \Densrenormalized
	& = \mathscr{Q},
		\label{E:RENORMALIZEDDENSITYWAVEEQUATION} \\
	\Transport \Vortrenormalized^i
	& = \Vortrenormalized^a \partial_a v^i.
	\label{E:RENORMALIZEDVORTICTITYTRANSPORTEQUATION}
	\end{align}
	\end{subequations}
	Above, $\mathscr{Q}^i$ and $\mathscr{Q}$ are the \textbf{null forms} relative to $g$, which are defined by
	\begin{subequations}
		\begin{align}
		\mathscr{Q}^i
		& := 	-(1+ \Speed^{-1} \Speed') 
					(g^{-1})^{\alpha \beta} \partial_{\alpha} \Densrenormalized \partial_{\beta} v^i,
			\label{E:VELOCITYNULLFORM} \\
		\mathscr{Q}
		& := 
		- 
		3 \Speed^{-1} \Speed' 
		(g^{-1})^{\alpha \beta} \partial_{\alpha} \Densrenormalized \partial_{\beta} \Densrenormalized
		+ 
		2 \sum_{1 \leq a < b \leq 3}
			\left\lbrace
				\partial_a v^a \partial_b v^b
					-
				\partial_a v^b \partial_b v^a
			\right\rbrace.
			\label{E:DENSITYNULLFORM}
		\end{align}
	\end{subequations}

	In addition, $\Flatdiv \Vortrenormalized$ and the scalar-valued functions
	$(\Flatcurl \Vortrenormalized)^i$
	verify the following equations:
	\begin{subequations}
	\begin{align} \label{E:FLATDIVOFRENORMALIZEDVORTICITY}
	\Flatdiv \Vortrenormalized
	& = - \Vortrenormalized^a \partial_a \Densrenormalized,
		\\
	\Transport (\Flatcurl \Vortrenormalized)^i
	& = (\exp \Densrenormalized) \Vortrenormalized^a \partial_a \Vortrenormalized^i
			- (\exp \Densrenormalized) \Vortrenormalized^i \Flatdiv \Vortrenormalized
			+ \mathscr{P}_{(\Vortrenormalized)}^i,
			\label{E:EVOLUTIONEQUATIONFLATCURLRENORMALIZEDVORTICITY}
\end{align}
\end{subequations}
where $\mathscr{P}_{(\Vortrenormalized)}^i$ is defined by
\begin{align} \label{E:VORTICITYNULLFORM}
	\mathscr{P}_{(\Vortrenormalized)}^i
	& := \epsilon_{iab} 
			 \left\lbrace
				(\partial_a \Vortrenormalized^c) \partial_c v^b
				- 
				(\partial_a v^c) \partial_c \Vortrenormalized^b
			 \right\rbrace.
\end{align}
\end{theorem}

\begin{remark}[\textbf{The structure of the term} $\mathscr{P}_{(\Vortrenormalized)}^i$]
	\label{R:DELICATEVORTICITYQUDRATICTERM}
	As written, the term $\mathscr{P}_{(\Vortrenormalized)}^i$ from \eqref{E:VORTICITYNULLFORM}
	does not have the special null structure that is essential
	for applications to shock formation.
	However, by using equation \eqref{E:RENORMALIZEDVORTICTITYTRANSPORTEQUATION}
	for substitution, one can show that cancellations occur,
	which yields the desired null structure;
	see the proof of Theorem~\ref{T:STRONGNULL}.
\end{remark}

\begin{remark}[\textbf{Simplified equations in two spatial dimensions}]
\label{R:SIMPLIFIEDEQUATIONSIN2D}
In two spatial dimensions, 
the equations simplify considerably
due to the absence of vorticity stretching.
Specifically, 
$\mbox{{\upshape RHS}~\eqref{E:RENORMALIZEDVORTICTITYTRANSPORTEQUATION}} \equiv 0$
for solutions that are independent of $x^3$ and have $v^3 \equiv 0$.
Consequently, one does not need to use equations 
\eqref{E:FLATDIVOFRENORMALIZEDVORTICITY}-\eqref{E:EVOLUTIONEQUATIONFLATCURLRENORMALIZEDVORTICITY}
when deriving estimates in two spatial dimensions.
\end{remark}

\begin{remark}[\textbf{The irrotational case}]\label{rmk.irrot}
For data with vanishing vorticity, the solution verifies $\Vortrenormalized \equiv 0$,
as long as it remains $C^1$.
For such solutions, the system of equations from Theorem \ref{T:GEOMETRICWAVETRANSPORTSYSTEM}
becomes a system of quasilinear wave equations 
whose right-hand sides consist only of
quadratic null forms relative to the acoustical metric $g$. 
We note that in particular, 
the equations from Theorem \ref{T:GEOMETRICWAVETRANSPORTSYSTEM}
yield, in the irrotational case,
the wave equations 
derived in \cite{dCsM2014}, 
but without the need to introduce a fluid potential. 
More precisely, in \cite{dCsM2014},
Christodoulou--Miao showed
that all Cartesian coordinate partial derivatives of the fluid potential 
$\Phi$, which verifies\footnote{We also note the equation
$
\displaystyle
\partial_t \Phi- \frac{1}{2}(\partial_1 \Phi)^2=h
$,
where $h$ is the enthalpy, defined such that $dh = \Speed^2 \,d \Densrenormalized$.
} 
$\partial_i\Phi=-v^i$, satisfy homogeneous
covariant quasilinear wave equations,
where the metric is conformal to the acoustical metric $g$ of Def. \ref{D:ACOUSTICALMETRIC}. 
We note that it is easy to show that a conformal change of the metric 
changes the wave equation, but only by generating a semilinear term 
that is proportional to a $g$-null form (as defined in Def.~\ref{D:NULLFRAME}).
Thus, in the irrotational case, $-v^i$, being a Cartesian derivative of $\Phi$, 
satisfies a quasilinear wave equation whose inhomogeneous term
exhibit the desired null structure.\footnote{It can also be directly verified using \eqref{E:TRANSPORTDENSRENORMALIZEDRELATIVETORECTANGULAR}-\eqref{E:TRANSPORTVELOCITYRELATIVETORECTANGULAR} that in the absence of vorticity, \eqref{E:VELOCITYWAVEEQUATION} is equivalent to $\square_{\tilde g} v^i=0$, where 
$\tilde{g} : =\exp(\Densrenormalized) \Speed g$ is a metric conformal to $g$.} We also note that in the irrotational case,
the calculations of \cite{dCsM2014} could be extended to yield our wave equation \eqref{E:RENORMALIZEDDENSITYWAVEEQUATION}
for $\Densrenormalized$; however, the calculations would be slightly more involved since 
$\Densrenormalized$ is a nonlinear function of the 
spacetime Cartesian coordinate partial derivatives
$\partial_{\alpha} \Phi$.
\end{remark}

\begin{remark}[\textbf{The data are constrained}]
\label{R:CONSTRAINTS}
If we think of 
$(\Densrenormalized,v^1,v^2,v^3,\Vortrenormalized^1,\Vortrenormalized^2,\Vortrenormalized^3)$
as independent scalar-valued functions, then the initial data for the mixed-order system
\eqref{E:VELOCITYWAVEEQUATION}-\eqref{E:RENORMALIZEDVORTICTITYTRANSPORTEQUATION}
and 
\eqref{E:FLATDIVOFRENORMALIZEDVORTICITY}-\eqref{E:EVOLUTIONEQUATIONFLATCURLRENORMALIZEDVORTICITY}
are 
$(\Densrenormalized,v^1,v^2,v^3,\Vortrenormalized^1,\Vortrenormalized^2,\Vortrenormalized^3)|_{t=0}$
and
$(\partial_t \Densrenormalized,\partial_t v^1,\partial_t v^2,\partial_t v^3)|_{t=0}$.
However, in order to be consistent with the compressible Euler equations
\eqref{E:TRANSPORTDENSRENORMALIZEDRELATIVETORECTANGULAR}-\eqref{E:TRANSPORTVELOCITYRELATIVETORECTANGULAR},
the data must verify ``constraints.''
Specifically, 
$\lbrace \Vortrenormalized^i \rbrace_{i=1,2,3}|_{t=0}$ is determined
in terms of $\Densrenormalized$ and 
$\lbrace \partial_j v^i \rbrace_{i,j=1,2,3}|_{t=0}$
by equation \eqref{E:MODIFIEDVARIABLES},
while
$\partial_t \Densrenormalized|_{t=0}$
and
$\lbrace \partial_t v^i \rbrace_{i=1,2,3}|_{t=0}$
are determined in terms of
$\Densrenormalized|_{t=0}$,
$\lbrace v^i \rbrace_{i=1,2,3}|_{t=0}$,
$\lbrace \partial_i \Densrenormalized \rbrace|_{i=1,2,3}|_{t=0}$,
and
$\lbrace \partial_j v^i \rbrace_{i,j=1,2,3}|_{t=0}$
via the compressible Euler equations
\eqref{E:TRANSPORTDENSRENORMALIZEDRELATIVETORECTANGULAR}-\eqref{E:TRANSPORTVELOCITYRELATIVETORECTANGULAR}.

In our forthcoming work on shock formation, we consider initial data for the system
\eqref{E:VELOCITYWAVEEQUATION}-\eqref{E:RENORMALIZEDVORTICTITYTRANSPORTEQUATION}
and 
\eqref{E:FLATDIVOFRENORMALIZEDVORTICITY}-\eqref{E:EVOLUTIONEQUATIONFLATCURLRENORMALIZEDVORTICITY}
that verify tensorial smallness/largeness conditions;
see Subsect.~\ref{SS:PREVIEWONSHOCKS} for more details.
We of course must ensure that our smallness/largeness conditions 
are consistent with the constraints.
Aside from that, the fact that the data are constrained is a minor issue
since our main interest in Theorem~\ref{T:GEOMETRICWAVETRANSPORTSYSTEM} 
is that it provides equations that are useful for
deriving a priori estimates for solutions.
\end{remark}

\subsection{Statement of the main result II: Strong null condition}
\label{SS:MAINRESULT.II}

Our second main theorem sharply characterizes the null structure 
of the inhomogeneous terms in the above system.
The theorem refers to the ``strong null condition,'' which we rigorously define
in Def.~\ref{D:STRONGNULLCONDITION}. As we have mentioned, 
the strong null condition roughly states that none 
of the inhomogeneous term products contain two factors involving differentiations
transversal to the acoustic characteristics.

\begin{theorem}[\textbf{The inhomogeneous terms verify the strong null condition}]
	\label{T:STRONGNULL}
	For \emph{solutions}\footnote{Theorem~\ref{T:STRONGNULL}
	is valid only for solutions in the sense that the proof relies on using
	\eqref{E:RENORMALIZEDVORTICTITYTRANSPORTEQUATION} for algebraic substitution
	in order to exhibit the desired structure for the term
	$\mathscr{P}_{(\Vortrenormalized)}^i$ 
	from \eqref{E:VORTICITYNULLFORM}; see Remark~\ref{R:DELICATEVORTICITYQUDRATICTERM}.
	\label{FN:THEOREMISVALIDONLYFORSOLUTIONS}} 
	to \eqref{E:VELOCITYWAVEEQUATION}-\eqref{E:RENORMALIZEDVORTICTITYTRANSPORTEQUATION}
	and
	\eqref{E:FLATDIVOFRENORMALIZEDVORTICITY}-\eqref{E:EVOLUTIONEQUATIONFLATCURLRENORMALIZEDVORTICITY},
	the inhomogeneous terms 
	on the right-hand sides of the equations consist of two types: 
	\textbf{i)} terms that are manifestly linear in the first derivatives of 
	$(\Densrenormalized,v^1,v^2,v^3,\Vortrenormalized^1,\Vortrenormalized^2,\Vortrenormalized^3)$ and 
	\textbf{ii)} terms that can be expressed 
	as products that are quadratic in the first derivatives of 
	$(\Densrenormalized,v^1,v^2,v^3,\Vortrenormalized^1,\Vortrenormalized^2,\Vortrenormalized^3)$
	and that \textbf{verify the strong null condition} 
	(see Def.~\ref{D:STRONGNULLCONDITION}) relative to the acoustical metric $g$
	(see Def.~\ref{D:ACOUSTICALMETRIC}).
\end{theorem}

\begin{remark}[\textbf{Exact decompositions are important}]
	\label{R:SIGNIFICANCEOFSTRONGNULL}
	Given Theorem~\ref{T:GEOMETRICWAVETRANSPORTSYSTEM}, 
	Theorem~\ref{T:STRONGNULL} is a simple result. However,
	its significance for the study of shock formation is profound.
	Roughly speaking, terms verifying the strong null condition
	do not prevent the shock from forming.
	The special structures associated to the strong null condition
	become visible only relative to 
	an exact (as opposed to approximate) null frame adapted to $g$.
	That is, when proving shock formation,
	\emph{there seems to be no room for error when decomposing nonlinear terms}.
	This is in stark contrast to many
	problems for nonlinear wave equations in which 
	small-data global existence holds.
	In those problems, there is often
	room for error in the decompositions
	and it is often possible to prove small-data global existence
	by decomposing nonlinear terms
	relative to a null frame adapted to a background metric.
	The background geometry allows for a drastically simplified approach to deriving estimates.
	We explore these issues in more detail in 
	Remark~\ref{R:STRONGNULLFULLYNONLINEAR},
	Subsect.~\ref{SS:STRONGNULLVSOTHERNULL},
	and Subsubsect.~\ref{SSS:SOMEREMARKSONTHEPROOF}.

\end{remark}

\section{The strong null condition and proof of Theorem~\ref{T:STRONGNULL}}
\label{S:PROOFOFTHMSTRONGNULL}
In this section, we provide some basic geometric background 
and prove Theorem~\ref{T:STRONGNULL}, which shows that the 
appropriate inhomogeneous terms from Theorem~\ref{T:GEOMETRICWAVETRANSPORTSYSTEM}
verify the strong null condition. We also compare and contrast the
strong null condition to distinct null structures found in other problems.

\subsection{Null frames, null forms, and the strong null condition}
\label{SS:NULLSTUFFANDPROOFOFSECONDTHEOREM}
Our main goal in this subsection is to define the strong null condition.
We first provide some standard background material.

\begin{definition}[\textbf{Null frame}]
	\label{D:NULLFRAME}
	Let $g$ be a Lorentzian metric on\footnote{The topology of the spacetime manifold is not relevant for our discussion here.} 
	$\mathbb{R} \times \mathbb{R} \times \mathbb{T}^2$.
	A $g-$\emph{null frame} (``null frame'' for short, when the metric is clear)
	at a point $p$ is a set of vectors 
	\begin{align} \label{E:NULLFRAME}
		\mathscr{N} := \lbrace \Lunit,\uLunit,e_1,e_2 \rbrace
	\end{align}
	belonging to the tangent space of $\mathbb{R} \times \mathbb{R} \times \mathbb{T}^2$ at $p$
	with
	\begin{subequations}
	\begin{align}
	g(\Lunit,\Lunit) 
	& = g(\uLunit,\uLunit) = 0,
		&& 
		\label{E:NULLFRAMEVECTORFIELDSLENGTHZERO} \\
	g(\Lunit,\uLunit) &= -2,
		&& 
		\label{E:NULLPAIRINNERPRODUCTMINUS2} \\
	g(\Lunit,e_A) &= g(\uLunit,e_A) = 0, 
	&& (A=1,2),
		\label{E:NULLVECTORFIELDSORTHOGONALTOSPACELIKEONES} \\
	g(e_A,e_B) &= \delta_{AB},
	&& (A,B=1,2),
	\label{E:SPACELIKEVECTORFIELDSORTHONORMAL}
	\end{align}
	\end{subequations}
	where $\delta_{AB}$ is the standard 
	Kronecker delta.
\end{definition}

The following lemma is a consequence of Def.~\ref{D:NULLFRAME};
we omit the simple proof.

\begin{lemma}[\textbf{Decomposition of $g^{-1}$ relative to a null frame}]
	\label{L:DECOMPOFGINVERSERELATIVETONULLFRAME}
	Relative to an arbitrary $g-$null frame, we have
	\begin{align} \label{E:DECOMPOFGINVERSERELATIVETONULLFRAME}
		g^{-1} 
		& = - \frac{1}{2} \Lunit \otimes \uLunit
				- \frac{1}{2} \Lunit \otimes \uLunit
				+ \sum_{A=1}^2 e_A \otimes e_A.
	\end{align}
\end{lemma}

\begin{definition}[\textbf{Decomposition of a derivative-quadratic nonlinear term relative to a null frame}]
	\label{D:DECOMPOSINGNONLINEARTERMRELATIVETONULLFRAME}
	Let $\vec{V} 
	:= 
	(\Densrenormalized,
	v^1,v^2,v^3,\Vortrenormalized^1,\Vortrenormalized^2,\Vortrenormalized^3)$
	be the array of unknowns in the system
	\eqref{E:VELOCITYWAVEEQUATION}-\eqref{E:RENORMALIZEDVORTICTITYTRANSPORTEQUATION}
	and 
	\eqref{E:FLATDIVOFRENORMALIZEDVORTICITY}-\eqref{E:EVOLUTIONEQUATIONFLATCURLRENORMALIZEDVORTICITY}.
	We label the components of $\vec{V}$ by $V^0=\Densrenormalized$, $V^i=v^i$, $V^{i+3}=\Vortrenormalized^i$ for $i=1,2,3$.
	Let $\mathcal{N}(\vec{V},\partial \vec{V})$
	be a smooth nonlinear term
	that is quadratically nonlinear in $\partial \vec{V}$.
	That is, we assume that 
	$\mathcal{N}(\vec{V},\partial \vec{V})
	=\smoothfunction(\vec V)_{\Theta\Gamma}^{\alpha \beta}\partial_{\alpha} V^\Theta \partial_{\beta} V^\Gamma$, 
	where $\smoothfunction(\vec V)_{\Theta\Gamma}^{\alpha \beta}$ is symmetric in $\Theta$ and $\Gamma$ 
	and is a smooth function of $\vec V$ 	(\emph{not} necessarily vanishing at $0$) 
	for $\alpha,\beta=0,1,2,3$ and $\Theta,\Gamma=0,1,\dots,6$.
	Given a null frame $\mathscr{N}$ as defined in Def.~\ref{D:NULLFRAME}, 
	we denote
	\[
\mathscr{N} := \lbrace e_1, e_2, e_3:= \uLunit, e_4 := \Lunit \rbrace.
\]
Moreover, we let $M_{\alpha}^A$ be the scalar functions
corresponding to expanding the Cartesian coordinate partial derivative vectorfield
$\partial_{\alpha}$ at $p$ relative to the null frame, that is,
\[
\partial_{\alpha} = \sum_{A=1}^4 M_{\alpha}^A e_A. 
\]
	Then\footnote{Here and below, we use the Einstein's summation convention, where uppercase Latin indices 
	such as $A$ and $B$
	vary over $1,2,3,4$, 
	lowercase Latin ``spatial'' indices such as $a$ and $b$ vary over $1,2,3$,
	uppercase Greek indices such as $\Theta$ and $\Gamma$
	vary over $0,1,\dots,6$, and lowercase Greek ``spacetime'' such as $\alpha$ and $\beta$
	indices vary over $0,1,2,3$. \label{FN:INDEXCONVENTIONS}}
	\begin{align} \label{E:NONLINEARTERMDECOMPOSEDRELATIVETONULLFRAME}
		\mathcal{N}_{\mathscr{N}} :=\smoothfunction(\vec V)_{\Theta\Gamma}^{\alpha \beta} 
			M_{\alpha}^A M_{\beta}^B (e_A V^\Theta)(e_B V^\Gamma)
	\end{align}
	denotes the nonlinear term
	obtained by expressing $\mathcal{N}(\vec{V},\partial \vec{V})$ in terms of the derivatives
	of $\vec{V}$ with respect to the elements of $\mathscr{N}$, that is, by
	expanding $\partial \vec{V}$ as a linear combination of 
	the derivatives of $\vec{V}$
	with respect to the elements
	of $\mathscr{N}$ and substituting the expression for the factor $\partial \vec{V}$
	in $\mathcal{N}(\vec{V},\partial \vec{V})$.
\end{definition}

We are now ready to state our main definition.

\begin{definition}[\textbf{Strong null condition}]
	\label{D:STRONGNULLCONDITION}
	Let $\mathcal{N}(\vec{V},\partial \vec{V})$ 
	be as in Def. \ref{D:DECOMPOSINGNONLINEARTERMRELATIVETONULLFRAME}. We say that $\mathcal{N}(\vec{V},\partial \vec{V})$ 
	verifies the \emph{strong null condition} relative to $g$ if
	the following condition holds: for \emph{every} $g-$null frame $\mathscr{N}$,
	$\mathcal{N}_{\mathscr{N}}$ can be expressed in a form that depends linearly (or not at all)
	on $\Lunit \vec{V}$ and $\uLunit \vec{V}$. That is, 
	there exists scalars 
	$\overline{\smoothfunction}_{\Theta\Gamma}^{AB}(\vec V)$ and $\underline{\smoothfunction}_{\Theta\Gamma}^{AB}(\vec V)$ 
	such that
	$$\overline{\smoothfunction}_{\Theta\Gamma}^{33}(\vec V)=\overline{\smoothfunction}_{\Theta\Gamma}^{44}(\vec V)=0,\quad 
	\underline{\smoothfunction}_{\Theta\Gamma}^{33}(\vec V)=\underline{\smoothfunction}_{\Theta\Gamma}^{44}(\vec V)=0$$
	and such that the following hold:
	\begin{equation}\label{strong.null.def}
	\begin{split}
	\smoothfunction(\vec V)_{\Theta\Gamma}^{\alpha \beta} M_{\alpha}^3 M_{\beta}^3 
	(e_3 V^\Theta)(e_3 V^\Gamma)=&\overline{\smoothfunction}_{\Theta\Gamma}^{AB}(\vec V)(e_A V^\Theta)(e_B V^\Gamma),\\
	\smoothfunction(\vec V)_{\Theta\Gamma}^{\alpha \beta} M_{\alpha}^4 M_{\beta}^4 (e_4 V^\Theta)(e_4 V^\Gamma)=&\underline{\smoothfunction}_{\Theta\Gamma}^{AB}(\vec V)(e_A V^\Theta)(e_B V^\Gamma).  
	\end{split}
	\end{equation}
	\end{definition}

\begin{remark}[\textbf{The strong null condition is not based on truncations}]
	\label{R:STRONGNULLFULLYNONLINEAR}
	Since the equations of Theorem~\ref{T:GEOMETRICWAVETRANSPORTSYSTEM} are such that
	the inhomogeneous terms are at most quadratic in the derivatives of the unknowns,
	we have given a definition of the strong null condition (Def.~\ref{D:STRONGNULLCONDITION}) 
	only for such nonlinearities. 
	If one were trying to study shock formation in a larger class of systems, then
	one could extend the definition of the strong null condition 
	to higher-order nonlinear terms.
	However, 
	if the definition were to be relevant for the proof of shock formation,
	then it would have to account for the exact structure of the higher-order terms.
	The reason is that one generally expects that in systems featuring
	quadratic \emph{and} cubic-or-higher-order (in the solution's derivatives) 
	terms, \emph{all terms} need to have special structure
	in order for a proof of shock formation to go through. 
	This is in contrast to Klainerman's original formulation \cite{sK1984}
	of a null condition in the context of small-data global-existence problems in $(1+3)$ dimensions, 
	which is based on truncated Taylor expansions 
	of the nonlinearities in which the cubic 
	and higher-order terms do not matter; see discussion in Section 
	\ref{SS:STRONGNULLVSOTHERNULL}. 
	That is, the structures needed to close a proof of shock formation 
	are less stable and are close in spirit to the ones that seem to be needed
	in low-regularity problems
	(see Subsubsect.~\ref{sss.low.regularity} for further discussion).
	One should perhaps not be too surprised by this, since, 
	even in the simple case of the Riccati ODE $\dot{y}= y^2$,
	the nature of solutions can be drastically altered
	by the addition of terms proportional to $y^3$, $y^4$, $y^5$, etc.

\end{remark}

\begin{remark}[\textbf{The strong null condition is adapted to the acoustical metric}]
By definition, the strong null condition depends on the acoustical metric $g$. 
Prior works on quasilinear wave equations indicate 
that such a structure is useful, 
and often indispensable, 
for handling the wave part of the system.
However, 
it is not a priori obvious that a null structure \emph{adapted to $g$} 
is also crucial for controlling
the inhomogeneous nonlinear terms in the transport equation for $\Flatcurl \Vortrenormalized$;
the principal part of the transport equation has no obvious connection
to the covariant wave operator $\square_g$. 
Nonetheless, as we will show in our forthcoming works 
on shock formation, the strong null condition is indeed the right condition, 
since the singularity formation is driven 
by the wave part of the system and not the transport part.
\end{remark}

It is well-known that there is a class of nonlinearities, associated to the standard null forms, which obey the strong null condition. 
We now recall the definition of the standard null forms.

\begin{definition}[\textbf{Standard null forms}]
	\label{D:NULLFORMS}
	The standard null forms $\mathscr{Q}^g(\cdot,\cdot)$
	(relative to $g$)
	and
	$\mathscr{Q}_{(\alpha \beta)}(\cdot,\cdot)$
	act on pairs $(\phi,\widetilde{\phi})$
	of scalar-valued functions as follows:
	\begin{subequations}
		\begin{align}
		\mathscr{Q}^g(\partial \phi, \partial \widetilde{\phi})
		&:= (g^{-1})^{\alpha \beta} \partial_{\alpha} \phi \partial_{\beta} \widetilde{\phi},
			\label{E:Q0NULLFORM} \\
		\mathscr{Q}_{(\alpha \beta)}(\partial \phi, \partial \widetilde{\phi})
		& = \partial_{\alpha} \phi \partial_{\beta} \widetilde{\phi}
			-
			 \partial_{\alpha} \widetilde{\phi} \partial_{\beta} \phi.
			\label{E:QALPHABETANULLFORM}
		\end{align}
	\end{subequations}
\end{definition}

It is well-known that the standard null forms obey the strong null condition.
For completeness, we will give a proof of this fact as part of the proof of Theorem~\ref{T:STRONGNULL},
given below in Subsect.~\ref{SS:PROOFOFTHEOREMSTRONGNULL}.

\begin{remark}[\textbf{In the new formulation, 
the derivative-quadratic nonlinear terms are 
\emph{not} all standard null forms}]\label{rmk.notallnullforms}
As we mentioned above, the standard null forms (relative to $g$) of Def.~\ref{D:NULLFORMS}
satisfy the strong null condition of Def.~\ref{D:STRONGNULLCONDITION}. 
In fact, if one requires the stronger condition that the cancellation structure occurs for all functions instead of just solutions to the system, that is, if one requires (compare with \eqref{strong.null.def})
$$\smoothfunction(\vec V)_{\Theta\Gamma}^{\alpha \beta}
M_{\alpha}^3 M_{\beta}^3
=\smoothfunction(\vec V)_{\Theta\Gamma}^{\alpha \beta} M_{\alpha}^4 M_{\beta}^4=0,\quad\mbox{for all }\Theta, \Gamma=0,1,\dots, 6,$$
then it is an easy exercise to show that the nonlinearities must be linear combinations
the standard null forms relative to $g$, with coefficients depending on $\vec{V}$. 
In our setting,
while most of the nonlinear terms in the system
\eqref{E:VELOCITYWAVEEQUATION}-\eqref{E:RENORMALIZEDVORTICTITYTRANSPORTEQUATION}
and 
\eqref{E:FLATDIVOFRENORMALIZEDVORTICITY}-\eqref{E:EVOLUTIONEQUATIONFLATCURLRENORMALIZEDVORTICITY}
have the desired good null structure because they are linear combinations of the standard null forms
relative to $g$, this is \emph{not} the case for all of them. 
In particular, the term $\mathscr{P}_{(\Vortrenormalized)}^i$ from equation
\eqref{E:EVOLUTIONEQUATIONFLATCURLRENORMALIZEDVORTICITY}
verifies the strong null condition \emph{only when $\vec V$ is a solution to the compressible Euler equations};
see the proof in Sect.~\ref{SS:PROOFOFTHEOREMSTRONGNULL} for further details.
\end{remark}

\subsection{Comparing and contrasting the strong null condition to null structures found in other contexts}
\label{SS:STRONGNULLVSOTHERNULL}

The notion of a ``null condition'' was first introduced by Klainerman \cite{sK1984} 
in his study of small-data global existence for systems of nonlinear wave equations in
three spatial dimensions. Ever since, his null condition 
and related ones have been ubiquitous in the analysis of nonlinear wave equations. 
They lie at the heart of many spectacular advances, including the stability of Minkowski spacetime, global regularity for critical geometric wave equations, and the formation of trapped surfaces, just to name a few. As we have mentioned and will further discuss, 
the shock formation results of Christodoulou \cites{dC2007,dCsM2014} and Speck \cite{jS2014b} 
also rely on a type of null condition,\footnote{Although the quasilinear terms, on the other hand, necessarily violate Klainerman's null condition in view of shock formation.} 
and obtaining a deeper understanding it 
lies at the heart of our forthcoming work on shock formation in the presence of vorticity. 
In this subsection, we briefly describe various notions of null conditions 
and compare/contrast them with the strong null condition of Def.~\ref{D:STRONGNULLCONDITION}.

\subsubsection{Small-data global existence problems}\label{sss.small.data}
As we mentioned above, a notion of a null condition was first introduced \cite{sK1984} in the context of small-data global existence problems in $(1+3)$ dimensions. His notion, which we will call the \emph{classic null condition}, is based on 
Taylor expanding the nonlinear terms up to quadratic order.
A foundational result, due independently to Christodoulou \cite{dC1986a}
and Klainerman \cite{sK1986}, states that if the nonlinearities in the equation satisfy the classic null
condition, then all sufficiently small initial data give rise to global solutions.
In the wake of \cites{sK1984,sK1986,dC1986a}, 
many extensions of these results have been proved.
Perhaps the most spectacular of these is the monumental work of Christodoulou--Klainerman \cite{dCsK1993}, 
who showed that Minkowski spacetime 
is stable as a solution to the Einstein vacuum equations. 
In particular, a key to the result is that for solutions to Einstein's equations,
the Bianchi equations, viewed as a system of evolution equation for the Weyl curvature tensor, 
exhibit a good null structure, similar to the structure introduced in \cite{sK1984}, 
but adapted to the dynamic spacetime metric.

The key insights behind Klainerman's classic null condition are 
\textbf{i)} solutions to the linear wave equation on $\mathbb{R}^{1+3}$, 
when differentiated with respect to different elements of a (canonical) null frame,\footnote{Such a null frame is adapted to 
flat Minkowski light cones with vertices at the (spatial) origin.} decay with different rates 
and \textbf{ii)} most importantly from the point of view of analysis,
the classic null condition excludes the presence of the most slowly decaying quadratic terms. 
In fact, in small-data global existence problems,
cubic and higher-order terms decay much faster and thus are easier to control.
It is for the latter reason that the classic null condition is concerned only with 
quadratic terms obtained from Taylor expanding the nonlinearities.

In addition to the classic null condition,
other kinds of null structures have been identified as being relevant
in the context of small-data global existence problems.
Moreover, like the classic null condition,
these notions of a null structure typically allow for\footnote{This is perhaps not surprising in view of the fact that small-data
global existence proofs are typically closable because there is a margin of error
in the estimates.} 
a ``margin of error.'' An important example is found in the study of
the Einstein vacuum equations in the wave coordinate gauge.
In this gauge, the equations violate the classic null condition, 
but still possess a structure that is a particular case of 
Lindblad--Rodnianski's \emph{weak null condition} \cite{hLiR2003}; 
see also \cites{sA2003,hL2008}. 
In \cite{hLiR2010}, Lindblad--Rodnianski
exploited the weak null condition to 
give a proof of the stability of Minkowski spacetime in the wave coordinate gauge. 
Remarkably, while the true dynamic metric is not the Minkowski metric 
(in fact, the true null cones provably \emph{diverge logarithmically} from their Minkowskian counterparts!), 
they nonetheless are able to control the nonlinear terms by relying on
on a weak null condition whose formulation was tied to the geometry of the background Minkowski metric. 
In addition, their weak null condition was not sensitive to the presence of most cubic terms.
Moreover, their proof of small-data global existence relied
only on vectorfields adapted to the Minkowskian characteristics 
(that is, the standard flat light cones)
-- and not the characteristics of the dynamic metric. Their approach,
which was drastically simpler than the original approach of Christodoulou-Klainerman \cite{dCsK1993},
was viable in part because even though some error terms are allowed to grow in time,
the growth is sufficiently slow and can be suitably controlled. 
This is in stark contrast to the situation encountered in the proof of shock formation,
which we describe in Subsubsect.~\ref{sec.null.euler}; near the shock singularity, 
one seems to need a null condition adapted exactly to the relevant metric 
(that is, the acoustical metric of Def.~\ref{D:ACOUSTICALMETRIC})
with no margin of error.

\subsubsection{Low-regularity problems}\label{sss.low.regularity}
Another class of problems for which standard null forms 
play an important role is low-regularity problems. 
Specifically, many remarkable global low-regularity results have been achieved
for various semilinear wave equations with standard null form nonlinearities.
Examples include wave maps, Maxwell--Klein--Gordon equations, 
and Yang--Mills equations \cites{sKmM1994,sKmM1995,tT20082009,jSdT2010,jKwS2012,jKjL2015,sjOdT2015}. 
A crucial ingredient in these results is bilinear estimates, 
for which the full structure of the nonlinearity, 
as opposed to only its quadratic part, has to be exploited. 
In fact, typical derivative-cubic terms, while completely benign in the context of Subsubsect.~\ref{sss.small.data}, 
would invalidate the proofs if the equations were modified to include them.

In a recent breakthrough, Klainerman--Rodnianski--Szeftel \cite{sKiRjS2015} extended 
the above low-regularity techniques to the Einstein vacuum equations, which, in an appropriate gauge, constitute a \emph{quasilinear} system of wave equations for which the semilinear terms are standard null forms. Their main result was a proof of the bounded $L^2$ curvature conjecture, which asserts that local existence of solutions to the Einstein vacuum equations holds true as long as the initial data have curvature in $L^2$. Their proof crucially relies on the fact that the nonlinear terms are standard null forms adapted to the dynamic metric $g$ (which occurs in the principal part of the equation), with no margin of error. 
In particular, the weak null condition of Lindblad--Rodnianski, while useful for small data global existence problems, seems irrelevant for these kinds of problems. Indeed, Ettinger--Lindblad recently showed \cite{bEhL2016}
that in the wave coordinate gauge, such a low-regularity local existence result fails.

As a final example of quasilinear wave equation for which a null condition plays a crucial role, we mention the monumental work of Christodoulou \cite{dC2009} on the Einstein vacuum equations, 
in which he showed that trapped surfaces can form dynamically
and moreover, their formation is stable.
In this work, Christodoulou introduced the \emph{short pulse method}.
More precisely, he introduced a small parameter $\delta$ such that the data are supported in a region
of $\delta$-size null affine length and obey a tensorial hierarchy of smallness-largeness estimates, 
where sizes are measured in terms of powers of $\delta^{-1}$. 
Christodoulou showed \cite{dC2009} that due to the remarkable null structure of the equations
in the double null foliation gauge, this hierarchy of large and small quantities can be propagated by the
flow of the equations long enough for a trapped surface to form. 
In his work, it was important that the good null structure 
was adapted exactly to the dynamic metric $g$ in
a manner similar to the discussions in the previous two paragraphs.
In fact, as was pointed out in \cite{jLiR2013}, 
this problem can be viewed as a low-regularity problem, 
since it is only for a very rough norm that the data are bounded independent of $\delta$.

\subsubsection{Null condition in the setting of compressible Euler equations with vorticity}\label{sec.null.euler}
As we have already mentioned in the introduction, 
in the previous works on shock formation as well as in our forthcoming work,
the special null structure of the inhomogeneous terms
is one of the key ingredients in the proofs.
In those works, although the initial data are regular, 
some low-order standard Sobolev norm 
(defined with respect to the Cartesian coordinate partial derivative vectorfields)
of the solution blows up (for example, the standard $H^1$ norms of $v^i$ and $\rho$
blow up in \cite{jLjS2016b}).
For this reason, the authors need to control the solution 
up to a time when this low-order standard Sobolev norm
blows up; it is only relative to a special \emph{low-regularity} 
norm involving directionally dependent powers of a geometric weight
(specifically, the weight $\upmu$ defined in \eqref{E:UPMUDEF}), 
designed specifically to capture the geometry of the shock,
that the solution remains bounded.
It is therefore not surprising that our strong null condition 
(see Def.~\ref{D:STRONGNULLCONDITION})
shares many similarities to the null
conditions mentioned in Subsubsect.~\ref{sss.low.regularity}
(as opposed to the null condition of Subsubsect.~\ref{sss.small.data}). 
In particular, 
it is not surprising that our
definition of the strong null condition refers to exact $g$--null frames, 
where $g$ is the acoustical metric of Def.~\ref{D:ACOUSTICALMETRIC}. 
Indeed, in the proof of shock formation, 
one must use vectorfields adapted to the true characteristics 
(as opposed to approximate ones) in order to avoid incurring uncontrollable error terms.
Moreover, our condition cannot be based on truncated Taylor expansion, since
the proof is very sensitive to cubic and higher-order terms.

We note that there are two new features of the strong null condition for the compressible Euler equations. 
First, it appears to be the first instance of a null condition that 
is \emph{not} based on truncations
and that plays a crucial role in a problem involving quasilinear wave equations coupled 
to another quasilinear equation of a \emph{different characteristic speed}. 
Second, as we already emphasized in Remark \ref{rmk.notallnullforms}, 
the strong null condition can accommodate some new nonlinearities
that are not standard null forms. However, the cancellations needed to exhibit
the good null structure of these new nonlinearities occur only for 
\emph{solutions} to the system.
This is somewhat reminiscent
of the cancellations tied to the use of
the wave coordinate gauge in the Lindblad--Rodnianski
proof \cite{hLiR2010} of the stability of Minkowski spacetime.
That is, in both cases, some of the special null structures found 
in the equations
(which are needed to close the proofs) 
occur only for solutions.
We note, however, that in our formulation of the compressible Euler equations,
the cancellations are fully nonlinear and not tied to a gauge
choice, which is different than the situation encountered in \cite{hLiR2010}.

\subsection{Proof of Theorem~\ref{T:STRONGNULL}}
\label{SS:PROOFOFTHEOREMSTRONGNULL}
It is easy to see that the terms
on the right-hand sides of equations
\eqref{E:VELOCITYWAVEEQUATION}-\eqref{E:RENORMALIZEDVORTICTITYTRANSPORTEQUATION}
and
\eqref{E:FLATDIVOFRENORMALIZEDVORTICITY}-\eqref{E:EVOLUTIONEQUATIONFLATCURLRENORMALIZEDVORTICITY}
consist of three types: type \textbf{i)} terms (as defined in the statement of the theorem),
quadratic terms consisting of linear combinations 
(with coefficients depending on $\vec{V}$)
of the standard null forms 
\eqref{E:Q0NULLFORM}-\eqref{E:QALPHABETANULLFORM} acting on the elements of $\vec{V}$,
and the terms
$\mathscr{P}_{(\Vortrenormalized)}^i$
defined in \eqref{E:VORTICITYNULLFORM}. It thus suffices to consider the standard null forms and the term $\mathscr{P}_{(\Vortrenormalized)}^i$.

\noindent{\underline{\textbf{Standard null forms satisfy the strong null condition:}}}

From the formula \eqref{E:DECOMPOFGINVERSERELATIVETONULLFRAME},
which is valid for an arbitrary $g-$null frame,
it is clear the terms of the form $\mathscr{Q}^g(\cdot,\cdot)$ verify the strong null condition.

To handle the terms of the form $\mathscr{Q}_{(\alpha \beta)}(\cdot,\cdot)$,
we denote the null frame \eqref{E:NULLFRAME} by
\[
\mathscr{N} := \lbrace e_1, e_2, e_3:= \uLunit, e_4 := \Lunit \rbrace.
\]
Since the null frame spans the tangent space at each point where it is defined, 
we can express, for $\alpha = 0,1,2,3$,
\begin{equation}\label{M.def}
\partial_{\alpha} = \sum_1^4 M_{\alpha}^A e_A, 
\end{equation}
where
the $M_{\alpha}^A$ are scalar-valued functions.
Then 
\[
\mathscr{Q}_{(\alpha \beta)}(\partial \phi,\partial \widetilde{\phi})
= \sum_{A,B=1}^4
	\left\lbrace 
		M_{\alpha}^A M_{\beta}^B 
		-
		M_{\alpha}^B M_{\beta}^A
	\right\rbrace
	(e_A \phi) 
	e_B \widetilde{\phi}.
\]
The term in braces is antisymmetric in $A$ and $B$ and thus there
are no diagonal terms $(e_A \phi) e_A \widetilde{\phi}$
present in the sum. In particular, terms proportional to
$(\uLunit \phi) \uLunit \widetilde{\phi}$
and
$(\Lunit \phi) \Lunit \widetilde{\phi}$ are not present.
It follows that the terms
$\mathscr{Q}_{(\alpha \beta)}(\cdot,\cdot)$ verify the strong null condition.

\noindent{\underline{\textbf{$\mathscr{P}_{(\Vortrenormalized)}^i$ satisfies the strong null condition:}}}

It remains for us to analyze the terms
\begin{align} \label{E:VORTICITYQUADRATICTERMNULLFRAMEEXPANDED}
 \mathscr{P}_{(\Vortrenormalized)}^i 
& =  	\sum_{A,B=1}^4
			 \epsilon_{iab}
				M_a^A
				M_c^B
			\left\lbrace
				(e_A \Vortrenormalized^c) e_B v^b
				- 
				(e_A v^c) e_B \Vortrenormalized^b
			 \right\rbrace
\end{align}
from \eqref{E:VORTICITYNULLFORM}, where $M_a^A$ is as in \eqref{M.def}.
Note that all terms on RHS~\eqref{E:VORTICITYQUADRATICTERMNULLFRAMEEXPANDED}
are allowable under the strong null condition except for
\begin{align} \label{E:EXCEPTIONALTERMS}
	& \epsilon_{iab} M_a^3 M_c^3 (e_3 \Vortrenormalized^c) e_3 v^b
		-
		\epsilon_{iab} M_a^3 M_c^3 (e_3 v^c) e_3 \Vortrenormalized^b
		\\
	& \ \ 
		+ 
		\epsilon_{iab} M_a^4 M_c^4 (e_4 \Vortrenormalized^c) e_4 v^b
		- 
		\epsilon_{iab} M_a^4 M_c^4 (e_4 v^c) e_4 \Vortrenormalized^b.
	\notag
\end{align}
To handle the terms in 
\eqref{E:EXCEPTIONALTERMS},
we will use equation \eqref{E:RENORMALIZEDVORTICTITYTRANSPORTEQUATION} for substitution.
We start by expanding the material derivative vectorfield (see \eqref{E:MATERIALVECTORVIELDRELATIVETORECTANGULAR})
relative to the null frame:
$
\Transport 
= \sum_{A=1}^4 \upbeta^A e_A
$, 
where the $\upbeta^A$ are scalar-valued functions. 
From \eqref{E:TRANSPORTISUNITLENGTH} and
\eqref{E:NULLFRAMEVECTORFIELDSLENGTHZERO}-\eqref{E:SPACELIKEVECTORFIELDSORTHONORMAL},
we find that the product $\upbeta^3 \upbeta^4$ verifies
\[
\upbeta^3 \upbeta^4 \neq 0,
\]
and thus both $\upbeta^3$ and $\upbeta^4$ are non-vanishing.
Hence, 
using the expansion
$
\Transport 
= \sum_{A=1}^4 \upbeta^A e_A
$,
we can replace the factors 
$e_3 \Vortrenormalized^c$,
$e_3 \Vortrenormalized^b$,
$e_4 \Vortrenormalized^c$,
and
$e_4 \Vortrenormalized^b$
in \eqref{E:EXCEPTIONALTERMS} 
respectively with
$(1/\upbeta^3) \Transport \Vortrenormalized^c$,
$(1/\upbeta^3) \Transport \Vortrenormalized^b$,
$(1/\upbeta^4) \Transport \Vortrenormalized^c$,
and
$(1/\upbeta^4) \Transport \Vortrenormalized^b$,
up to terms that are allowable under the strong null condition.
It remains for us to analyze
\begin{align} \label{E:MOREEXCEPTIONALTERMS}
	& \epsilon_{iab} (1/\upbeta^3) M_a^3 M_c^3 (\Transport \Vortrenormalized^c) e_3 v^b
		-
		\epsilon_{iab} (1/\upbeta^3) M_a^3 M_c^3 (e_3 v^c) \Transport \Vortrenormalized^b
		\\
	& \ \ 
		+ 
		\epsilon_{iab} (1/\upbeta^4) M_a^4 M_c^4 (\Transport \Vortrenormalized^c) e_4 v^b
		- 
		\epsilon_{iab} (1/\upbeta^4) M_a^4 M_c^4 (e_4 v^c) \Transport \Vortrenormalized^b.
	\notag
\end{align}
To handle the terms on the first line of \eqref{E:MOREEXCEPTIONALTERMS},
we use equation \eqref{E:RENORMALIZEDVORTICTITYTRANSPORTEQUATION} to replace
$
\Transport \Vortrenormalized^c 
$
with
$
\upomega^d M_d^3 e_3 v^c
$
and
$
\Transport \Vortrenormalized^b
$
with
$
\upomega^d M_d^3 e_3 v^b
$,
up to terms that are allowable under the strong null condition.
After substitution, the terms on the
first line of \eqref{E:MOREEXCEPTIONALTERMS}
become, up to terms that are allowable under the strong null condition,
$
\epsilon_{iab} (1/\upbeta^3) M_a^3 M_c^3 \upomega^d M_d^3 (e_3 v^c) e_3 v^b
-
\epsilon_{iab} (1/\upbeta^3) M_a^3 M_c^3 \upomega^d M_d^3 (e_3 v^c) e_3 v^b
=
0
$;
this identity is the key cancellation in the proof.
Similarly, to handle the terms on the second line of \eqref{E:MOREEXCEPTIONALTERMS},
we can use equation \eqref{E:RENORMALIZEDVORTICTITYTRANSPORTEQUATION} to replace
$
\Transport \Vortrenormalized^c 
$
with
$
\upomega^d M_d^4 e_4 v^c
$
and
$
\Transport \Vortrenormalized^b
$
with
$
\upomega^d M_d^4 e_4 v^b
$,
up to terms that are allowable under the strong null condition, and then argue as above.
This completes the proof of Theorem~\ref{T:STRONGNULL}.

$\hfill \qed$

\section{Ideas behind the proof of shock formation and its connection to null structure}
\label{S:IDEASBEHINDSHOCKFORMATION}
In this section, we overview, without proof, how the structures
revealed by Theorems~\ref{T:GEOMETRICWAVETRANSPORTSYSTEM} and \ref{T:STRONGNULL}
are used in our forthcoming results \cites{jLjS2016b,jLjS2017}
on stable shock-formation result for the compressible Euler equations
in regions containing vorticity. That discussion is located Subsect.~\ref{SS:PREVIEWONSHOCKS}.
In the preliminary Subsect.~\ref{SS:GOODNULLINPROOFOFSHOCKFORMATION}, we recall 
Christodoulou's framework \cite{dC2007} for proving shock formation in the irrotational case;
the framework also plays an important role in our works \cites{jLjS2016b,jLjS2017}.
We do not provide any proofs in the irrotational case either; 
detailed proofs, tailored to the discussion below,
are located in \cite{jSgHjLwW2016}, which provided an extension of Christodoulou's result \cite{dC2007}
to treat a new regime of initial data (see below for more details).
We focus mainly on the work \cite{jSgHjLwW2016}
rather than \cite{dC2007} because, for reasons explained below,
some aspects of it are simpler to implement.
Readers may also consult the survey article \cite{gHsKjSwW2016} for additional discussion
on Christodoulou's result \cite{dC2007} and related ones.

\subsection{The case of the irrotational wave equations and related quasilinear wave equations}
\label{SS:GOODNULLINPROOFOFSHOCKFORMATION}
In this subsection, we describe the main ideas behind the proof of shock formation
in solutions to a general class of quasilinear wave equations
that includes, as a special case, the irrotational compressible Euler equations.

\subsubsection{Preliminary discussion concerning the equations}
\label{SSS:PRELIMINARYDICUSSION}
In \cite{dC2007}, Christodoulou provided a complete description
of the formation of shocks for perturbations
(belonging to a suitable high-order Sobolev space) 
of the non-vacuum constant state solutions
to the equations of (special) relativistic fluid mechanics in three spatial dimensions 
in regions with vanishing vorticity.
His results hold for any barotropic equation of state\footnote{There is precisely one exceptional equation of state for the irrotational special relativistic Euler equations
			for which the shock-formation results do not hold.
			The exceptional equation of state corresponds to the Lagrangian 
			$\mathscr{L} = 1 - \sqrt{1 + (m^{-1})^{\alpha \beta} \partial_{\alpha} \Phi \partial_{\beta} \Phi}$,
			where $m$ is the Minkowski metric. 
			It is exceptional because it is the only Lagrangian for relativistic fluid mechanics 
			such that Klainerman's null condition is satisfied for perturbations near the non-vacuum
			constant states. Due to the null condition, small-data global existence holds
			\cite{hL2004}.
			A similar statement holds for the non-relativistic compressible Euler equations; 
			see \cite{dCsM2014}*{Sect.~2.2} for more information.
			\label{FN:EXCEPTIONALLAGRANGIANS}} 
and were extended to the non-relativistic compressible Euler equations in \cite{dCsM2014}.
In both the relativistic and non-relativistic cases, 
under an arbitrary barotropic equation of state, 
the dynamics in the irrotational case
reduce to a quasilinear wave equation
of Euler-Lagrange type for a potential function $\Phi$.
The equation can be written relative to Cartesian coordinates 
in the following non-Euler-Lagrange form:
\begin{align} \label{E:CHRISTODOULOUWAVEEQN}
	(g^{-1})^{\alpha \beta}(\partial \Phi) \partial_{\alpha} \partial_{\beta} \Phi
	& = 0,
\end{align}
where the form of the Cartesian metric component functions 
$g_{\alpha\beta} = g_{\alpha\beta}(\partial \Phi)$ 
is determined by the equation of state.
The Lorentzian ``spacetime'' metric $g$, whose inverse appears in \eqref{E:CHRISTODOULOUWAVEEQN}, 
may be viewed as a $4 \times 4$ symmetric matrix of
signature $(-,+,+,+)$. The metric $g$ is the exact analog of the
acoustical metric from Def.~\ref{D:ACOUSTICALMETRIC}.

It turns out that to prove shock formation for solutions to equation \eqref{E:CHRISTODOULOUWAVEEQN}, 
it is convenient to differentiate the equation one time with Cartesian coordinate partial derivatives.
This motivates the following definition, $(\nu = 0,1,2,3)$:
\begin{align} \label{E:PSICOMPONENTSDEF}
	\vec{\Psi}
	& := (\Psi_0,\Psi_1,\Psi_2,\Psi_3),
	\qquad
	\Psi_{\nu}
	:= \partial_{\nu} \Phi.
\end{align}
In \cite{dC2007}, Christodoulou showed that by differentiating the irrotational wave
equations of relativistic fluid mechanics
with Cartesian coordinate partial derivatives
(see \cite{dCsM2014} for the same result in the case of the non-relativistic compressible Euler equations),
one obtains the following system\footnote{More precisely, Christodoulou had to rescale
$g$ by a conformal factor in order to bring the equation
into the form \eqref{E:DIFFCHRISTODOULOUWAVEEQN},
but that detail is not important for our discussion.}
on $\mathbb{R}^{1+3}$:
\begin{align} \label{E:DIFFCHRISTODOULOUWAVEEQN}
	\square_{g(\vec{\Psi})} \Psi_{\nu} & = 0.
\end{align}
The nonlinearities in \eqref{E:DIFFCHRISTODOULOUWAVEEQN} are
hidden in the covariant wave operator 
$\square_{g(\vec{\Psi})}$ on the LHS.
In \cite{jS2014b}, Speck showed that the vanishing of RHS~\eqref{E:DIFFCHRISTODOULOUWAVEEQN}
is tied to the Euler-Lagrange structure of the original irrotational Euler wave equation.
Moreover, he showed that \emph{all} equations of the form 
\eqref{E:CHRISTODOULOUWAVEEQN} (not necessarily of Euler-Lagrange type), upon differentiation,
yield a system of the form
\begin{align} \label{E:DIFFSPECKWAVEEQN}
	\square_{g(\vec{\Psi})} \Psi_{\nu} 
	& = \mathscr{Q}_{\nu}(\vec{\Psi},\Psi_{\nu}),
\end{align}
where $\mathscr{Q}_{\nu}(\cdot,\cdot)$
\emph{verifies the strong null condition}
of Def.~\ref{D:STRONGNULLCONDITION}.

It is important for the proof of shock formation
that $\square_{g(\vec{\Psi})}$
is the covariant wave operator of
$g(\vec{\Psi})$ 
(see Def.~\ref{D:COVWAVEOP});
the operator $\square_{g(\vec{\Psi})}$ enjoys particularly good commutation properties
with appropriately constructed vectorfields, which we describe 
in Subsubsect.~\ref{SSS:QUICKSUMMARYOFPROOFOFSHOCKFORMATION}.
As we have mentioned, in various solution regimes, 
the nonlinear terms $\mathscr{Q}_{\nu}$ on RHS~\eqref{E:DIFFSPECKWAVEEQN}
have a negligible effect\footnote{We may caricature the negligible effect by the Riccati-type ODE
$\dot{y} = y^2 + \epsilon y$,
where the $y^2$ term caricatures the shock-producing quadratic
term obtained from expanding the covariant wave operator
relative to Cartesian coordinates and $\epsilon y$ caricatures
the $\mathscr{Q}_{\nu}$. For $\epsilon$ small relative to the data $y(0) > 0$,
the $\epsilon y$ term does not interfere with the Riccati-type blowup.}
on the dynamics; we explain this in more detail at the end of 
Subsubsect.~\ref{SSS:SOMEREMARKSONTHEPROOF}.
For this reason, we ignore the $\mathscr{Q}_{\nu}$ for most of this subsection.
Moreover, as is described in \cite{jS2014b}, the proof of shock formation
for solutions to the system \eqref{E:DIFFSPECKWAVEEQN} is not
much more difficult than the proof in the case of a single scalar wave equation.
For this reason, 
until Subsect.~\ref{SS:PREVIEWONSHOCKS},
we restrict our attention to the scalar covariant wave equation
\begin{align} \label{E:SCALARMODELWAVE}
	\square_{g(\Psi)} \Psi
	& = 0.
\end{align}

In \cite{jS2014b}, Speck proved a
small-data\footnote{The ``smallness'' in \cite{jS2014b}
was stated in terms of a Sobolev norm of the data for 
$\Phi$ in equation \eqref{E:CHRISTODOULOUWAVEEQN},
for the data of $\vec{\Psi}$ in the system \eqref{E:DIFFSPECKWAVEEQN},
and for the data of $\Psi$ in equation \eqref{E:SCALARMODELWAVE}.
In contrast, in his study of equation \eqref{E:CHRISTODOULOUWAVEEQN} in \cite{dC2007}, 
Christodoulou assumed that the data of $\Phi - k t$
was small, where $k$ is a non-zero constant and $k t$ is a global background 
solution corresponding to a global non-vacuum fluid state. 
The fact that Christodoulou studied perturbations of the solution $kt$
rather than the solution $0$ is a minor detail that has no important bearing on the analysis;
see \cite{gHsKjSwW2016} for more details.
\label{FN:SMALLNESS}
} 
shock formation result, 
in the spirit of Christodoulou's work
\cite{dC2007}, 
for all equations of type
\eqref{E:CHRISTODOULOUWAVEEQN},
\eqref{E:DIFFSPECKWAVEEQN},
and
\eqref{E:SCALARMODELWAVE}
in three spatial dimensions
whenever the nonlinear terms \emph{fail} to satisfy Klainerman's null condition \cite{sK1984}.
We recall that, as we described in Subsect.~\ref{SS:PRELIMINARYDISCUSSION},
a similar but less precise result had been proved for equations of type \eqref{E:CHRISTODOULOUWAVEEQN}
by Alinhac.

\subsubsection{Preliminary remarks on solution regimes}\label{SSS:DATA}
It is by now well-understood that the geometric framework introduced in \cite{dC2007}
can be applied to show stable shock formation for 
large classes of scalar quasilinear wave equations in different solution regimes. 
Specifically, the works \cites{dC2007,dCsM2014,jS2014b} prove shock formation results for various scalar quasilinear wave equations 
in $1+3$ dimensions
for small compactly supported data, 
\cite{jSgHjLwW2016} treats the regime of nearly simple outgoing 
plane symmetric solutions, 
and \cite{sMpY2014} treats a special ``short pulse'' regime, 
which is a large-data regime. Although the fine details behind propagating the estimates are distinct in each case,
all of these works rely on a similar geometric framework that is able to accommodate and sharply
describe the formation of the shock. In our discussion here, we will focus on the solution regime of \cite{jSgHjLwW2016},
since we will study a similar regime in our forthcoming works on shock formation in the presence of non-zero vorticity.
We chose this solution regime in part because there is no dispersion,  
which simplifies some parts of the proof. That is, one does not need to keep track of decay in time or space, 
which simplifies some aspects of the analysis. However, we expect that
our forthcoming work could be generalized to other solution regimes as long 
as one makes appropriate smallness assumptions.

In \cite{jSgHjLwW2016}, the authors proved a two-space-dimensional
shock formation result for initial data posed on the Cauchy hypersurface
$\mathbb{R} \times \mathbb{T}$ (where $\mathbb{T}$ is the torus), 
which were assumed to be close to that of plane symmetric simple\footnote{Roughly, a simple wave $\Psi$
in one spatial dimension is such that $\Psi(u,v)$
is \emph{independent of $u$}, where $(u,v)$ form a coordinate system of eikonal functions
(that is, the level sets of $u$ and $v$ are null hypersurfaces).
Put differently, $u$ and $v$ are coordinate functions that
solve the eikonal equation \eqref{E:EIKONAL}. \label{FN:SIMPLEPLANE}} 
outgoing\footnote{Roughly, outgoing means right-moving. This choice was made for convenience; 
the left-moving case can be treated with the same arguments.} 
waves,
where the $\mathbb{T}$ direction corresponds to a breaking of the plane symmetry.
By plane symmetric, we mean $\Psi = \Psi(t,x^1)$, while by 
nearly plane symmetric, we mean $\Psi = \Psi(t,x^1,x^2)$ with small initial 
dependence on $x^2 \in \mathbb{T}$.
To propagate smallness in the problem, in particular the smallness of the perturbation away 
from simple plane symmetry, the authors of \cite{jSgHjLwW2016} introduced the data-size parameters
$\mathring{\upepsilon}$ and $\mathring{\updelta}$,
where $\mathring{\upepsilon}$ 
is small relative to $\mathring\updelta^{-1}$.

The geometric meanings of
$\mathring{\upepsilon}$
and
$\mathring{\updelta}$
are easy to describe:
$\mathring{\upepsilon}$ measures the size (in appropriate norms)
of $\Psi$ itself and its derivatives in directions
\emph{tangent} to the characteristics,\footnote{In \cite{jSgHjLwW2016}, the characteristics
were a family of null hyperplanes adapted to the approximate plane symmetry of the problem.
They are analogous to the acoustic characteristics that we encounter in our study of
the compressible Euler equations with vorticity.}
while $\mathring{\updelta}$ measures the size of the 
\emph{purely transversal} (that is, transversal to the characteristics) 
derivatives of $\Psi$. It was also assumed that the
mixed transversal-tangent derivatives are of small size $\mathring{\upepsilon}$. 
In the following, we will consider this solution regime, including the $\mathring{\updelta}$-$\mathring{\epsilon}$ size hierarchy, but \emph{adapted to three spatial dimensions}\footnote{As was discussed in \cite{jSgHjLwW2016}, 
the results of \cite{jSgHjLwW2016} can be generalized to the case of three spatial dimensions
using established techniques.} with the spatial manifold equal to $\mathbb R\times \mathbb T^2$. 
As we will discuss in Subsubsect.~\ref{SSS:FEATURESINCOMMONWITHIRROTATIONALCASE}, 
we will consider a similar solution regime in our forthcoming work 
on shock formation in the presence of vorticity. 



\subsubsection{Some preliminary remarks on the proof, including the significance of the 
strong null condition}
\label{SSS:SOMEREMARKSONTHEPROOF}
For the nearly simple outgoing plane symmetric solutions
described in Subsubsect.~\ref{SSS:PRELIMINARYDICUSSION}, 
the shock-producing homogeneous quasilinear wave equations of type \eqref{E:SCALARMODELWAVE}
can be caricatured\footnote{Note that in one spatial dimension, the linear wave equation
can be written as
$\Lunit_{(Flat)} \partial_1 \Psi = \Lunit_{(Flat)} \partial_t \Psi$ 
and thus $\Lunit_{(Flat)} \partial_1 \Psi = \f 12 \Lunit_{(Flat)} \Lunit_{(Flat)} \Psi$.
Hence, in \eqref{E:EQUATIONCARICATURE}, we have included
$\Lunit_{(Flat)} \Lunit_{(Flat)} \Psi$, which vanishes for simple outgoing waves, in the term $\mbox{\upshape Error}$.
} 
by the following equation 
on $\mathbb{R} \times \mathbb{R} \times \mathbb{T}^2$
(where $\Sigma_t \simeq \mathbb{R} \times \mathbb{T}^2$),
when the derivatives are decomposed with respect to the Cartesian frame $\{\Lunit_{(Flat)} : = \partial_t + \partial_1,\, \rd_1, \, \rd_2,\, \rd_3\}$:
\begin{align} \label{E:EQUATIONCARICATURE}
	\Lunit_{(Flat)} \partial_1 \Psi 
	& = (\partial_1 \Psi)^2 
	+ \mbox{\upshape Error}.
\end{align}
In \eqref{E:EQUATIONCARICATURE}, 
$\mbox{\upshape Error}$ 
consists of quasilinear and semilinear terms
depending on the up-to-second (principal!) order derivatives of $\Psi$
and, by assumption, $\mbox{\upshape Error}$ is \emph{initially} small.
Equation \eqref{E:EQUATIONCARICATURE} 
suggests that $\partial_1 \Psi$, if initially large with respect to
$\mbox{\upshape Error}$,
should experience Riccati-type blow up along the integral curves
of $\Lunit_{(Flat)}$.
However, 
the approach of writing the equation in the form \eqref{E:EQUATIONCARICATURE}
does not seem to actually allow one to prove that $\rd_1\Psi$ blows up;
it is difficult to guarantee that $\mbox{\upshape Error}$ is small throughout the evolution. 

To explain why the above scheme should fail,
we must explain some basic facts about the nature of the blowup.
Specifically, 
a fundamental aspect of Christodoulou's framework is 
that the blowup occurs for derivatives of $\Psi$ in directions
transversal to the characteristics, whose intersection is tied to the blowup.
The relevant characteristics in the nearly plane symmetric regime are \emph{perturbations} 
of the level sets of 
$u_{(Flat)} := 1 - x^1 + t$, 
to which $\Lunit_{(Flat)}$ is tangential.
In particular, 
terms such as $\partial_2^2 \Psi$,
which have been relegated to the term
$\mbox{\upshape Error}$
on RHS~\eqref{E:EQUATIONCARICATURE},
generally blow up at the shock since 
$\partial_2$ is generally transversal to the characteristics
(even though it is tangent to the characteristics $\lbrace 1 - x^1 + t = \mbox{\upshape const} \rbrace$
of the global background solution).
It is for this reason that the scheme from the previous paragraph 
seems to be insufficient for proving that blowup occurs.


A key idea behind Christodoulou's approach in \cite{dC2007} is
to ``hide'' the singularity via a dynamic change of coordinates,
adapted to the characteristics.
This can be viewed as a high-dimensional analogue of the well-known
hodograph transformation in one spatial dimension, in which
one introduces a new system of geometric coordinates in which the solution remains regular;
the singularity reveals itself only in the degeneration of the map from geometric
to Cartesian coordinates. This same phenomenon of hiding the singularity can be exhibited in a much simpler
		context via Burgers' equation
		$\partial_t \Psi + \Psi \partial_x \Psi$:
		shock-forming solutions to Burgers' equation remain smooth 
		(that is, $C^{\infty}$ if the data are)
		\emph{relative to Lagrangian coordinates}.\footnote{In Lagrangian coordinates
		$(t,u)$, where $t$ is the Cartesian time coordinate 
		and $u$ is defined to be constant along integral curves of $\partial_t + \Psi \partial_x$, 
		Burgers' equation reads
		$
		\displaystyle
		\frac{\partial}{\partial t} \Psi = 0
		$, 
		where
		$
		\displaystyle
		\frac{\partial}{\partial t}
		$
		denotes partial differentiation with respect to $t$ at fixed $u$.
		We note that
		$
		\displaystyle
		\frac{\partial}{\partial t} 
		= \partial_t + \Psi \partial_x
		$,
		where $\partial_t$ and $\partial_x$
		are the usual Cartesian coordinate partial derivative vectorfields.
		but are such that the Cartesian coordinate partial derivative 
		$\partial_x \Psi$ blows up. \label{FN:BURGER}}

More precisely, in three spatial dimensions, 
one constructs a new system of \emph{geometric coordinates}
$(t,u,\vartheta^1,\vt^2)$ 
(with corresponding coordinate partial derivative vectorfields\footnote{Let us note that 
$
\displaystyle
\f{\rd}{\rd t}
$ here is defined with respect to the $(t,u,\vartheta^1,\vt^2)$ coordinate system, and it is not equal to
the Cartesian vectorfield $\rd_t$. 
$
\displaystyle
\f{\rd}{\rd t}
$
can viewed as a dynamically constructed analog of $\Lunit_{(Flat)}$, which takes into account the quasilinear geometry.}
$
\displaystyle
\left\lbrace
\frac{\partial}{\partial t},
\frac{\partial}{\partial u},
\frac{\partial}{\partial \vartheta^1},
\frac{\partial}{\partial \vartheta^2}
\right\rbrace
$)
on the spacetime $\mathbb{R} \times \mathbb{R} \times \mathbb{T}^2$,
relative to which the solution remains 
regular, all the way up to the shock, except 
at the very high derivative levels, 
where the corresponding energies are allowed to blow up in a controlled fashion.
These degenerate energy estimates are the source of 
almost all of the technical difficulties that one faces;
see Subsubsect.~\ref{SSS:QUICKSUMMARYOFPROOFOFSHOCKFORMATION}.

One might say that relative to the geometric coordinates,
the singularity is \emph{renormalizable} except at the very high
derivatives levels. 
The singularity occurs in the Cartesian coordinate partial derivatives $\partial_{\alpha} \Psi$
because the geometric coordinates degenerate
(in a precise fashion that lies at the heart of the proof)
relative to the Cartesian ones.
The most important geometric coordinate is the eikonal function $u$,
which we later describe in great detail. The eikonal function is constructed so that
its level sets are null hypersurfaces (which we also refer to as 
``characteristics'' or, in the context of the compressible Euler equations,
as ``acoustic characteristics''
in view of their connection to sound wave propagation)
and thus
$
\displaystyle
\left\lbrace
\frac{\partial}{\partial t},
\frac{\partial}{\partial \vartheta^1},
\frac{\partial}{\partial \vartheta^2}
\right\rbrace
$
are tangent to the acoustic characteristics.
 
It turns out that upon being re-expressed relative to the geometric coordinates, 
equation \eqref{E:EQUATIONCARICATURE} can be caricatured as
\begin{align} \label{E:GEOMETRICCOORDINATEEQUATIONCARICATURE}
	\frac{\partial}{\partial t}
	\frac{\partial}{\partial u}
	\Psi 
	& 
		= 
		\mbox{\upshape Error},
\end{align}
where all terms in \eqref{E:GEOMETRICCOORDINATEEQUATIONCARICATURE}
\emph{remain highly differentiable relative to the geometric coordinates, all the way up to the shock}.\footnote{We note, 
however, that $\mbox{\upshape Error}$
contains, among other terms,
terms involving
$
\displaystyle
\frac{\partial^2 \Psi}{\partial (\vartheta^1)^2}  
$ 
and
$
\displaystyle
\frac{\partial^2 \Psi}{\partial (\vartheta^2)^2} 
$.
In other words, while 
these terms can be viewed as error terms
for heuristic considerations near the shock, they are 
``main terms'' from the point of view of the top-order energy estimates.}

As we mentioned above, 
the singularity in the first Cartesian coordinate partial derivatives of $\Psi$
is tied to the degeneration of the change of variables map between the geometric coordinates
and the Cartesian ones.
To capture the degeneration, 
one defines a geometric weight $\upmu$ 
such that $\upmu \to 0$ corresponds to the intersection of the
acoustic characteristics, the formation of a shock, 
and the blow up of the solution's Cartesian coordinate partial derivatives. 
Thus, proving shock formation is equivalent to
showing that $\upmu$ vanishes in finite time.
It turns out that $\upmu$ 
verifies an evolution equation that can be
caricatured, relative to geometric coordinates, as follows:\footnote{Throughout we use the notation
$A \sim B$ to imprecisely indicate that $A$ is well-approximated by $B$.}
\begin{align} \label{E:UPMUEVOLUTIONCARICATURE}
	\frac{\partial}{\partial t}
	\upmu
	& \sim \frac{\partial}{\partial u} \Psi
	+ 
	\mbox{\upshape Error},
\end{align}
where $\mbox{\upshape Error}$ remains small, all the way up to the shock.
The interplay between equations 
\eqref{E:GEOMETRICCOORDINATEEQUATIONCARICATURE}
and \eqref{E:UPMUEVOLUTIONCARICATURE} is the key to understanding the shock formation.

We now describe the relationship between the geometric and Cartesian coordinate
partial derivative vectorfields.
If $u$ is appropriately constructed, then, under appropriate assumptions
on the data, one can show that for nearly plane symmetric solutions, we have
\begin{align} \label{E:PARTIALUPARTIAL1CARICATURE}
	\frac{\partial}{\partial u}
	& = 
	- 
	\upmu \partial_1 + \upmu \mbox{\upshape Error},
\end{align}
where 
in \eqref{E:GEOMETRICCOORDINATEEQUATIONCARICATURE} and \eqref{E:PARTIALUPARTIAL1CARICATURE},
$\upmu \mbox{\upshape Error}$ remains small up to the shock;
see Figure~\ref{F:FRAME} below to obtain insight on the vectorfield
$
\displaystyle
\frac{\partial}{\partial u}
$,
which is well-approximated by the vectorfield $\Rad$ appearing in that figure.
Hence, to prove finite-time shock formation,
one considers data such that 
$
\displaystyle
\frac{\partial}{\partial u} \Psi
$
is sufficiently negative at some point. 
By integrating \eqref{E:GEOMETRICCOORDINATEEQUATIONCARICATURE} in time,
one can propagate this negativity for a long time,
as long as $\mbox{\upshape Error}$ remains small.
Then by integrating \eqref{E:UPMUEVOLUTIONCARICATURE} in time,
we see that $\upmu$
will vanish in finite time and, by
dividing
\eqref{E:PARTIALUPARTIAL1CARICATURE}
by $\upmu$,
that some Cartesian coordinate partial derivative of $\Psi$ will blow up
(in particular because
$
\displaystyle
\frac{\partial}{\partial u} \Psi
$
is strictly non-zero at the points where $\upmu$ vanishes).
To make this argument precise in more than one spatial
dimension, one of course needs to derive energy estimates. 
As we mentioned above, this is the difficult part of the proof;
see Subsubsect.~\ref{SSS:QUICKSUMMARYOFPROOFOFSHOCKFORMATION}.

Based on the above discussion, it is easy to explain the 
significance of the strong null condition of Def.~\ref{D:STRONGNULLCONDITION}
in the case where the nonlinear terms on the right hand side of \eqref{E:SCALARMODELWAVE}
are precisely quadratic in the solution's derivatives.
We first note that one can construct a null\footnote{Actually, strictly speaking, \eqref{E:SHOCKSNULLFRAME}
is not a null frame in the sense of Def.~\ref{D:NULLFRAME}
because the non-zero normalization conditions
for the frame \eqref{E:SHOCKSNULLFRAME} are different; this is a 
minor issue that we ignore here.
\label{E:SORTOFANULLFRAME}}
frame\footnote{In fact, while it is most convenient to explain the necessity of a null condition using a null frame, 
in deriving estimates in \cites{jLjS2016b,jLjS2017}, we use a slightly different frame which still captures the ``good'' and ``bad'' directions, but is more convenient from the point of view of commutations; see the discussion preceding Def.~\ref{frame.def.1}.}
that is expressible relative to the geometric coordinates as follows:
\begin{align} \label{E:SHOCKSNULLFRAME}
	\left\lbrace
		\Lunit 
		= \frac{\partial}{\partial t},\,
		\uLgood 
		= \upmu \frac{\partial}{\partial t}
			+ 2 \frac{\partial}{\partial u}
			- 2 \upmu \Xi,\,
			e_1,\, e_2
	\right\rbrace.
\end{align}
Above, $\Xi$ is a vectorfield in the span of
$
\displaystyle
\left\lbrace \frac{\partial}{\partial \vartheta^1}, \f{\rd}{\rd\vt^2} \right\rbrace$,
$\uLgood$ is $g$-null and $g$--orthogonal to the constant-$(t,u)$ tori,
and $\lbrace e_1, e_2 \rbrace$ is an arbitrary $g$--orthonormal frame
in the span\footnote{As a consequence, $e_1$ and $e_2$ are indeed orthogonal to $\Lunit$ and $\uLgood$.} of
$
\displaystyle
\left\lbrace \frac{\partial}{\partial \vartheta^1}, \f{\rd}{\rd\vt^2} \right\rbrace
$.
To proceed, we note that by \eqref{E:PARTIALUPARTIAL1CARICATURE},
a typical quadratic semilinear term 
$\sum_{\alpha,\beta= 0}^3 C_{\alpha \beta} (\partial_{\alpha} \Psi) \partial_{\beta} \Psi$,
with $C_{\alpha \beta}$ constants and $\partial_{\alpha}$ the Cartesian coordinate partial derivative vectorfields,
if present on RHS~\eqref{E:EQUATIONCARICATURE}, would yield 
(on RHS~\eqref{E:GEOMETRICCOORDINATEEQUATIONCARICATURE},
after multiplying \eqref{E:EQUATIONCARICATURE} by the factor $\upmu$
as described above),
relative to the geometric coordinates,
a term proportional to
$
\displaystyle
\frac{1}{\upmu}
\left(\frac{\partial}{\partial u} \Psi \right)^2
$.
The main problem with such a term is that the factor 
of $1/ \upmu$ prevents one from proving that 
$
\displaystyle
\frac{\partial}{\partial u} \Psi
$
remains bounded all the way up to the singularity
(which is caused by the vanishing of $\upmu$)
and therefore obstructs the basic philosophy of the approach:
showing that $\Psi$ remains regular relative to the geometric coordinates.
Similarly, generic cubic terms in $\partial \Psi$
would yield an even worse term
$
\displaystyle
\frac{1}{\upmu^2}
\left(\frac{\partial}{\partial u} \Psi \right)^3
$.
However, terms verifying the strong null condition do not 
suffer from these problems; 
it is easy to see from \eqref{E:SHOCKSNULLFRAME} that
quadratic terms verifying the
strong null condition can yield, for example, terms 
on RHS~\eqref{E:UPMUEVOLUTIONCARICATURE} of the form
$
\displaystyle
\frac{\partial}{\partial t} \Psi 
\cdot
\frac{\partial}{\partial u} \Psi
$,
$
\displaystyle
\f{\rd}{\rd\vt^1} \Psi 
\cdot
\frac{\partial}{\partial u} \Psi
$,
$
\displaystyle
\upmu
\left(\f{\rd}{\rd\vt^1} \Psi \right)^2
$,
or
$
\displaystyle
\upmu
\left(\f{\rd}{\rd\vt^2} \Psi \right)^2
$,
which do not incur the dangerous factor $1/\upmu$
and which involve at least one differentiation with respect to an element of
$
\displaystyle
\left\lbrace
{\f{\rd}{\rd t}},{\f{\rd}{\rd\vt^1}},{\f{\rd}{\rd\vt^2}}
\right\rbrace,
$
which are tangent to the characteristics.

Moreover, in the context of nearly simple outgoing plane symmetric waves 
(which we described in Subsubsect.~\ref{SSS:DATA}), 
not only are the quadratic terms verifying the strong null condition
regular near the shock, they 
are also \emph{small} all the way up to the shock. This is 
because each term in their decomposition
relative to the geometric coordinates
contains a ``tangential'' factor that is differentiated with respect to an element of
$
\displaystyle
\left\lbrace
{\f{\rd}{\rd t}},{\f{\rd}{\rd\vt^1}},{\f{\rd}{\rd\vt^2}}
\right\rbrace
$;
tangential factors enjoy the $\mathcal{O}(\mathring{\upepsilon})$ smallness
described in Subsubsect.~\ref{SSS:DATA}.

\subsubsection{Normalization choices and assumptions on the nonlinearities to ensure shock formation}
\label{SSS:ASSUMPTIONSONNONLINEARITIES}
We recall that we are studying shock formation in nearly simple outgoing
plane symmetric solutions to the wave equation \eqref{E:SCALARMODELWAVE}.
After appropriate normalization choices and rescaling, we may assume\footnote{Actually, 
in putting the metric into this form, we introduce 
a semilinear term proportional to $(g^{-1})^{\alpha \beta}(\Psi) \partial_{\alpha} \Psi \partial_{\beta} \Psi$
in the covariant wave equation corresponding to the rescaled metric.
However, this term verifies the strong null condition of Def.~\ref{D:STRONGNULLCONDITION}
and therefore has a negligible impact on the dynamics.
We therefore ignore it in the exposition.
\label{FN:GINVERSE00ISMINUSONE}}  
that the Cartesian components of the metric from \eqref{E:SCALARMODELWAVE}
verify
\begin{align} \label{E:LITTLEGDECOMPOSED}
	g_{\mu \nu} 
	= g_{\mu \nu}(\Psi)
	& := m_{\mu \nu} 
		+ g_{\mu \nu}^{(Small)}(\Psi),
	&& (\mu, \nu = 0,1,2,3),
\end{align}
where $m_{\mu \nu} = \mbox{diag}(-1,1,1,1)$
is the standard Minkowski metric on $\mathbb{R} \times \mathbb{R} \times \mathbb{T}^2$
and $g_{\mu \nu}^{(Small)}(\Psi)$ are given smooth functions of $\Psi$ with
\begin{align} \label{E:METRICPERTURBATIONFUNCTION}
	g_{\mu \nu}^{(Small)}(\Psi = 0)
	= 0, \qquad (g^{-1})^{00}(\Psi) \equiv -1.
\end{align}
To ensure that shocks form, we assume that
\begin{align} \label{E:NONVANISHINGNONLINEARCOEFFICIENT}
	G_{\alpha \beta}(\Psi = 0) \Lunit_{(Flat)}^{\alpha} \Lunit_{(Flat)}^{\beta} \neq 0,
\end{align}
where
\begin{align} \label{E:LFLAT}
	G_{\alpha \beta}
	= G_{\alpha \beta}(\Psi)
	& := \frac{d}{d \Psi} g_{\alpha \beta}(\Psi),\qquad \Lunit_{(Flat)} := \partial_t + \partial_1.
\end{align}
The assumptions \eqref{E:NONVANISHINGNONLINEARCOEFFICIENT} and \eqref{E:LFLAT}
are equivalent to the assumption that Klainerman's null condition fails
for equation \eqref{E:SCALARMODELWAVE} for solutions depending only on $(t,x^1)$.
Roughly, this implies that relative to Cartesian coordinates, 
there are quadratic Riccati-type semilinear terms present in the wave equation,
as we caricatured with the model term ``$(\partial_1 \Psi)^2$''
in Subsubsect.~\ref{SSS:SOMEREMARKSONTHEPROOF}.

\subsubsection{The eikonal function and related geometric constructions}
\label{SSS:GEOMETRICCONSTRUCTIONS}
As we explained above, the main idea behind the proof of shock formation 
under Christodoulou's framework \cite{dC2007} is that one
can construct a new system of \emph{geometric coordinates}
$(t,u,\vartheta^1,\vt^2)$ 
(see Def.~\ref{D:GEOMETRICCOORDS})
relative to which the solution remains 
regular, all the way up to the shock, except at the very high derivative levels.
As in the rest of the article, $t$ is the standard Cartesian time function.
The most important geometric coordinate is the eikonal function $u$, 
which solves the eikonal equation,
a nonlinear hyperbolic PDE coupled to the wave equation:
\begin{align} \label{E:EIKONAL}
	(g^{-1})^{\alpha \beta}(\Psi)
	\partial_{\alpha} u \partial_{\beta} u
	& = 0, 
	\qquad \partial_t u > 0,
\end{align}
where $g = g(\Psi)$ is the Lorentzian metric appearing in 
\eqref{E:SCALARMODELWAVE}.
We supplement \eqref{E:EIKONAL} with the initial conditions 
\begin{align} \label{E:ICEIKONAL}
	u|_{\Sigma_0} = 1 - x^1.
\end{align}
The choice \eqref{E:ICEIKONAL} is
motivated by the (assumed) approximate plane symmetry of the initial data for the wave equation.

The following regions of spacetime are determined by $t$ and $u$
and play an important role in the analysis. They are depicted in
Figure~\ref{F:SOLIDREGION}.

\begin{definition} [\textbf{Subsets of spacetime}]
\label{D:HYPERSURFACESANDCONICALREGIONS}
We define the following spacetime subsets:
\begin{subequations}
\begin{align}
	\Sigma_{t'} & := \lbrace (t,x^1,x^2,x^3) \in \mathbb{R} \times \mathbb{R} \times \mathbb{T}^2
		\ | \ t = t' \rbrace, 
		\label{E:SIGMAT} \\
	\Sigma_{t'}^{u'} & := \lbrace (t,x^1,x^2,x^3) \in \mathbb{R} \times \mathbb{R} \times \mathbb{T}^2
		 \ | \ t = t', \ 0 \leq u(t,x^1,x^2,x^3) \leq u' \rbrace, 
		\label{E:SIGMATU} 
		\\
	\mathcal{P}_{u'} 
	& := 
		\lbrace (t,x^1,x^2,x^3) \in \mathbb{R} \times \mathbb{R} \times \mathbb{T}^2
			\ | \ u(t,x^1,x^2,x^3) = u' 
		\rbrace, 
		\label{E:PU} \\
	\mathcal{P}_{u'}^{t'} 
	& := 
		\lbrace (t,x^1,x^2,x^3) \in \mathbb{R} \times \mathbb{R} \times \mathbb{T}^2
			\ | \ 0 \leq t \leq t', \ u(t,x^1,x^2,x^3) = u' 
		\rbrace, 
		\label{E:PUT} \\
	\ell_{t',u'} 
		&:= \mathcal{P}_{u'}^{t'} \cap \Sigma_{t'}^{u'},\\
	\mathcal{M}_{t',u'} & := \cup_{u \in [0,u']} \mathcal{P}_u^{t'}.
		\label{E:MTUDEF}
\end{align}
\end{subequations}
\end{definition}
The most important of these subsets are the $\mathcal{P}_u$, which 
we describe below in more detail.

\begin{center}
\begin{overpic}[scale=.2]{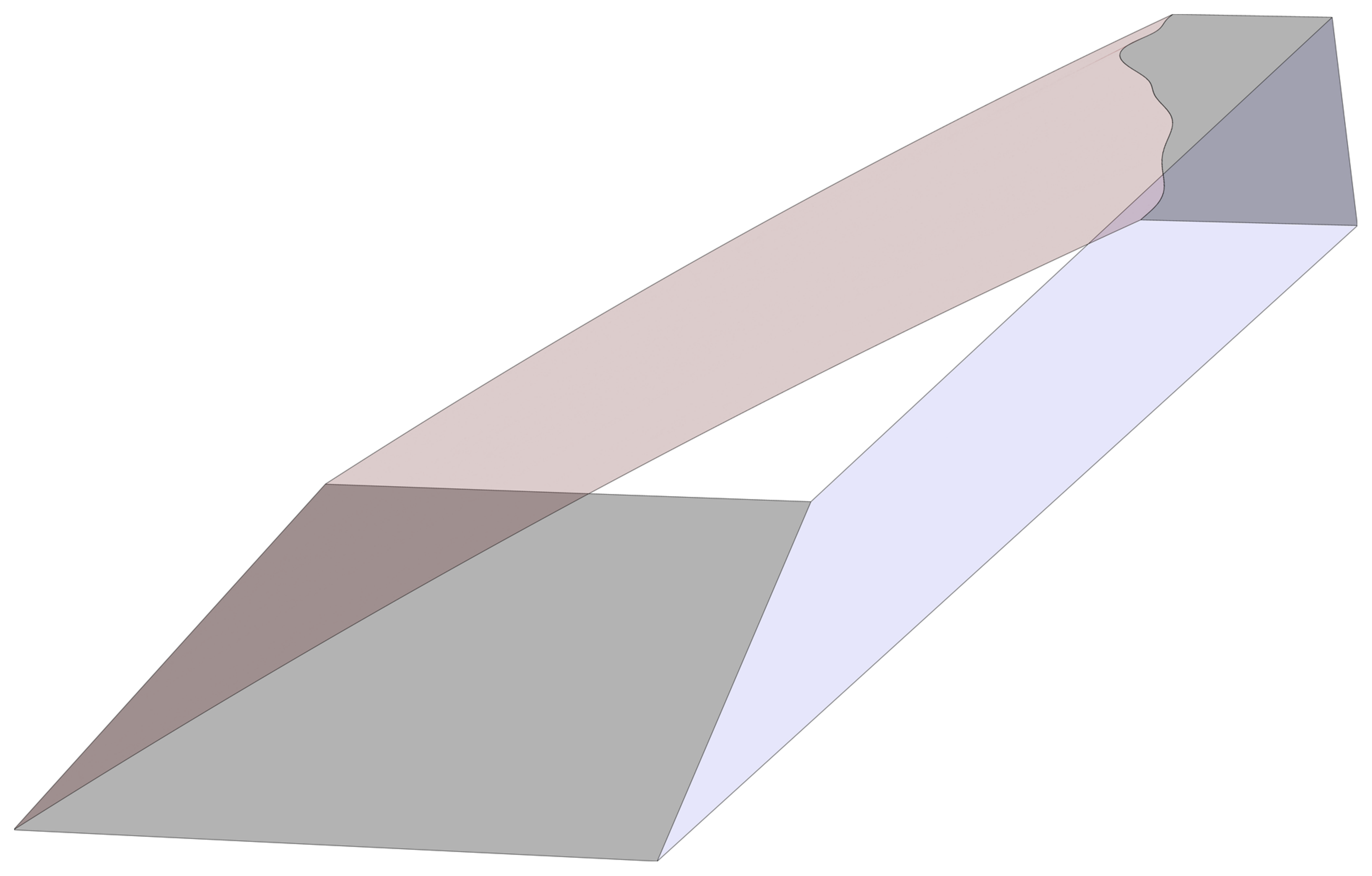} 
\put (54,31) {\large$\displaystyle \mathcal{M}_{t,u}$}
\put (42,33) {\large$\displaystyle \mathcal{P}_u^t$}
\put (74,33) {\large$\displaystyle \mathcal{P}_0^t$}
\put (75,12) {\large$\displaystyle \Psi \equiv 0$}
\put (32,17) {\large$\displaystyle \Sigma_0^u$}
\put (47,13) {\large$\displaystyle \ell_{0,0}$}
\put (12,13) {\large$\displaystyle \ell_{0,u}$}
\put (88.5,59) {\large$\displaystyle \Sigma_t^u$}
\put (93,53) {\large$\displaystyle \ell_{t,0}$}
\put (85.5,53) {\large$\displaystyle \ell_{t,u}$}
\put (-.6,16) {\large$\displaystyle x^2 \in \mathbb{T}$}
\put (22,-3) {\large$\displaystyle x^1 \in \mathbb{R}$}
\thicklines
\put (-.9,3){\vector(.9,1){22}}
\put (.7,1.8){\vector(100,-4.5){48}}
\end{overpic}
\captionof{figure}{The spacetime region and various subsets, with one spatial dimension suppressed.}
\label{F:SOLIDREGION}
\end{center}

We now explain how to use $u$ to construct a 
good vectorfield frame, given in equation \eqref{E:RESCALEDFRAME},
that is useful for studying the solution.
We note that the frame \eqref{E:RESCALEDFRAME}
is closely related to a $g$-null frame
(see Def.~\ref{D:NULLFRAME}); however, it is not literally a $g$-null frame
(and the difference is not important for the main ideas of the discussion).
To proceed, we associate the following gradient vectorfield
to the eikonal function:
\begin{align} \label{E:LGEOEQUATION}
	\Lgeo^{\nu} & := - (g^{-1})^{\nu \alpha} \partial_{\alpha} u.
\end{align}
From \eqref{E:EIKONAL}, we deduce that $\Lgeo$ is future-directed
and $g$-null:
\begin{align}  \label{E:LGEOISNULL}
g(\Lgeo,\Lgeo) 
:= g_{\alpha \beta} \Lgeo^{\alpha} \Lgeo^{\beta}
= 0.
\end{align}
Moreover, we can differentiate the eikonal equation with
$\D^{\nu} := (g^{-1})^{\nu \alpha} \D_{\alpha}$,
where $\D$ is the Levi-Civita connection of $g$,
and use the torsion-free property of $\D$ to deduce that
$
0
= 
(g^{-1})^{\alpha \beta} \D_{\alpha} u \D_{\beta} \D^{\nu} u
= - \D^{\alpha} u \D_{\alpha} \Lgeo^{\nu}
= \Lgeo^{\alpha} \D_{\alpha} \Lgeo^{\nu}
$. That is, $\Lgeo$ is geodesic:
\begin{align} \label{E:LGEOISGEODESIC}
	\D_{\Lgeo} \Lgeo & = 0.
\end{align}
In addition, since $\Lgeo$ is proportional to the $g-$dual of the one-form $d u$,
which is co-normal to the level sets $\mathcal{P}_u$ of the eikonal function,
it follows that $\Lgeo$ is $g-$orthogonal to $\mathcal{P}_u$.
Hence, the $\mathcal{P}_u$ have null normals. 
For this reason, such hypersurfaces
are known as \emph{null hypersurfaces}.
We sometimes refer to them as
``characteristics''
or, in the context of the compressible Euler equations,
as ``acoustic characteristics.''

As we mentioned earlier, the most important
quantity in connection with shock formation
is the inverse foliation density.

\begin{definition}[\textbf{Inverse foliation density}]
 \label{D:UPMUDEF}
Let $\Lgeo^0$ be the $0$ Cartesian component of the vectorfield $\Lgeo$ defined in 
\eqref{E:LGEOEQUATION}.
We define the inverse foliation density $\upmu$ as follows:
\begin{align} \label{E:UPMUDEF}
	\upmu 
	& := 
		\frac{-1}{(g^{-1})^{\alpha \beta} \partial_{\alpha} t \partial_{\beta} u} 
	= \frac{-1}{(g^{-1})^{0 \alpha} \partial_{\alpha} u} 
	= \frac{1}{\Lgeo^0}.
\end{align}
\end{definition}

$1/\upmu$ is a measure of the density of 
the characteristics $\mathcal{P}_u$
relative to the constant-time hypersurfaces $\Sigma_t$. When $\upmu$
becomes $0$, the density becomes infinite and the level sets of $u$
intersect. The idea to study this quantity in the context of shock formation
goes back at least to \cite{fJ1974}, in which John proved a blowup result 
for solutions to a large class of hyperbolic systems in one spatial dimension.

It is easy to show that under the assumptions of
Subsubsect.~\ref{SSS:ASSUMPTIONSONNONLINEARITIES}
and \eqref{E:ICEIKONAL},
we have
\begin{align} \label{E:UPMUINITIAL}
	\upmu|_{\Sigma_0} 
	& = 1 + \mathcal{O}(\Psi).
\end{align}
In particular, when $|\Psi|$ is initially small,
$\upmu$ is initially near unity.

It turns out that the Cartesian components $\Lgeo^{\nu}$ 
blowup when $\upmu$ vanishes (that is, when the shock forms).
It also turns out that the products 
$\upmu \Lgeo^{\nu}$
remain regular all the way up to the shock.
For this reason, the vectorfield $\Lunit := \upmu \Lgeo$ is useful for studying the solution.

\begin{definition}[\textbf{Rescaled null vectorfield}]
	\label{D:LUNITDEF}
	We define the rescaled null (see \eqref{E:LGEOISNULL}) vectorfield $\Lunit$ as follows:
	\begin{align} \label{E:LUNITDEF}
		\Lunit
		& := \upmu \Lgeo.
	\end{align}
\end{definition}
Note that $\Lunit t = 1$.

We now dynamically construct a geometric torus coordinates $\vartheta^1$ and $\vartheta^2$ by 
setting $\vartheta^1|_{\Sigma_0} = x^2$, $\vt^2|_{\Sigma_0} = x^3$ (with $x^2$ and $x^3$ 
being the standard Cartesian coordinates on $\mathbb{T}^2$)
and propagating the $\vartheta^A$ to the future via the transport equation
\begin{align} \label{E:TRANSPORTEQUATIONFORTHETA}
	\Lunit \vartheta^1 
	& = 
	\Lunit \vartheta^2 
	= 0.
\end{align}

\begin{definition}[\textbf{Geometric coordinates}]
		\label{D:GEOMETRICCOORDS}
		We refer to $(t,u,\vartheta^1,\vt^2)$ as the geometric coordinates.
		We denote the corresponding geometric partial derivative vectorfields by
		\begin{align} \label{E:GEOCORDSPARTIALDERIVVECTORFIELDS}
			\left\lbrace
				\frac{\partial}{\partial t},
				\frac{\partial}{\partial u},
				\frac{\partial}{\partial \vartheta^1},
				\f{\rd}{\rd\vt^2}
			\right\rbrace.
		\end{align}
\end{definition}

In addition to using the geometric coordinates, we also follow \cite{dC2007}
and introduce a set of geometric vectorfields adapted to them (and to the characteristics).
The necessity of using the geometric vectorfields is tied to a ``regularity issue''
that we suppressed in Subsubsect.~\ref{SSS:SOMEREMARKSONTHEPROOF} 
for the simplicity of the exposition:
directly commuting the coordinate vectorfields 
$
\displaystyle
\frac{\partial}{\partial u},
$
$
\displaystyle
\frac{\partial}{\partial \vartheta^1}
$
and
$
\displaystyle
\frac{\partial}{\partial \vartheta^2}
$
through the wave equation \eqref{E:SCALARMODELWAVE} leads 
error terms that lose a derivative, a difficulty that we do know how to overcome at the top order.
In contrast, in commuting with the geometric vectorfields, 
we are able to overcome\footnote{After great effort; see Subsubsect.~\ref{SSS:QUICKSUMMARYOFPROOFOFSHOCKFORMATION}.} 
the potential derivative loss
yet still capture the geometry of the shock singularity.


\begin{definition}[\textbf{The vectorfields} $\Radunit$, $\GeoAng_1$, \textbf{and} $\GeoAng_2$]\label{frame.def.1}
	We define $\Radunit$ to be the unique $\Sigma_t-$tangent vectorfield that
	is $g-$orthogonal to $\ell_{t,u}$ and normalized by
	\begin{align} \label{E:GLUNITRADUNITISMINUSONE}
		g(\Lunit,\Radunit) = -1.
	\end{align}
	Moreover we define the following $\upmu-$weighted version of $\Radunit$
	\begin{align}
		\Rad 
		& := \upmu \Radunit.
		\label{E:RADANDULGOOD}
	\end{align}
	Finally, define $\GeoAng_1$ (respectively $\GeoAng_2$) to be the $g$-orthogonal projection of $\rd_2$ (respectively $\rd_3$) onto $\ell_{t,u}$.


\end{definition}
It is convenient to use the following rescaled frame,\footnote{We refer to it as a rescaled frame 
since $\Rad$ is ``rescaled'' by a factor of $\upmu$.} 
which can be viewed as a replacement of the geometric coordinate partial derivative vectorfields
that does not suffer from the regularity problems mentioned above.
\begin{definition}[\textbf{Rescaled frame}]
\label{D:RESCALEDFRAME}
	We define the rescaled frame to be
	\begin{align} \label{E:RESCALEDFRAME}
	\left\lbrace
		\Lunit,
		\Rad,
		\GeoAng_1, \GeoAng_2
	\right\rbrace.
\end{align}
\end{definition}
The rescaled frame is depicted
in Figure~\ref{F:FRAME}.
It spans the tangent space of spacetime at each point
with $\upmu > 0$. Moreover, by construction, 
$\{L,\GeoAng_1,\GeoAng_2\}$ are \emph{tangential} 
to the characteristics $\mathcal P_u$, while $\Rad$ is \emph{transversal}. 
In addition, we note that the Cartesian components of 
$\Rad$ are proportional to $\upmu$ and thus are small
in regions where $\upmu$ is small. We now compare the rescaled frame 
to the geometric coordinate partial derivative vectorfields in \eqref{E:GEOCORDSPARTIALDERIVVECTORFIELDS}. 
Indeed, one first notes that 
\begin{equation}\label{E:LISDDT}
L=\f{\rd}{\rd t}.
\end{equation}
Next, one computes that $\Rad u=1$ and therefore, since $\Rad$ is tangent to $\Sigma_t$, we have 
$
\displaystyle
\Rad 
= 
\frac{\partial}{\partial u}
- 
\Xi
$, 
where $\Xi$ an $\ell_{t,u}-$tangent vectorfield. 
Finally, the pair $\{\GeoAng_1, \GeoAng_2\}$, like 
$
\displaystyle
\left\lbrace \f{\rd}{\rd\vt^1},\f{\rd}{\rd\vt^2} \right\rbrace
$, 
are tangential to $\ell_{t,u}$ (and in particular to the characteristics $\mathcal{P}_u$).

\begin{center}
\begin{overpic}[scale=.35]{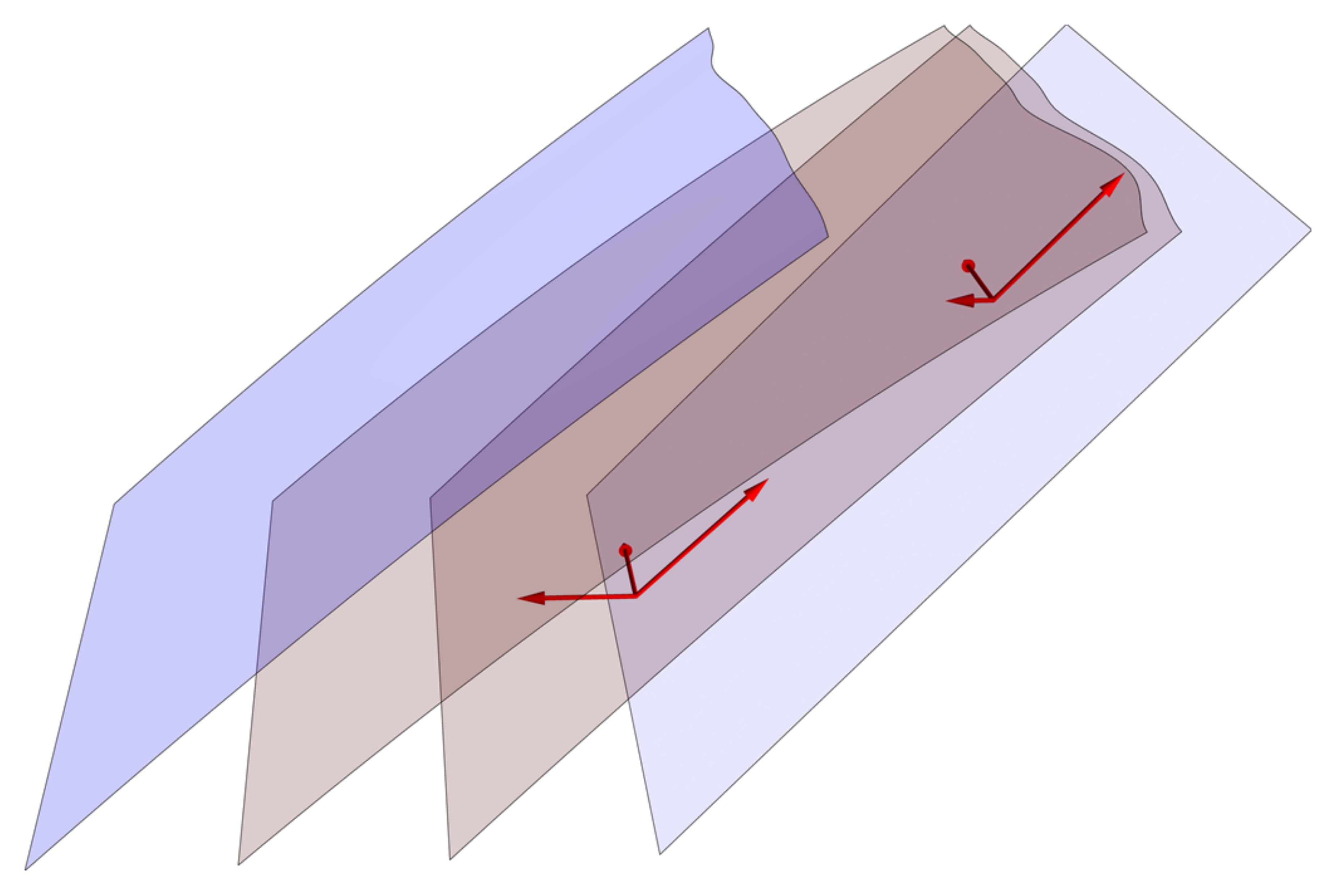} 
\put (84,55) {\large$\displaystyle \Lunit$}
\put (67,43.5) {\large$\displaystyle \Rad$}
\put (70,49) {\large$\displaystyle \GeoAng_1$}
\put (57,32) {\large$\displaystyle \Lunit$}
\put (35,21) {\large$\displaystyle \Rad$}
\put (44,28) {\large$\displaystyle \GeoAng_1$}
\put (51,13) {\large$\displaystyle \mathcal{P}_0^t$}
\put (37,13) {\large$\displaystyle \mathcal{P}_u^t$}
\put (7,13) {\large$\displaystyle \mathcal{P}_1^t$}
\put (21,10) {\large$\displaystyle \upmu \approx 1$}
\put (70,58) {\large$\displaystyle \upmu \ \mbox{\upshape small}$}
\put (75,12) {\large$\displaystyle \Psi \equiv 0$}
%
\end{overpic}
\captionof{figure}{The rescaled frame at two distinct points in 
$\mathcal{P}_u$, with one spatial dimension suppressed.}
\label{F:FRAME}
\end{center}

The reader may have noticed that while the definition 
of the strong null condition 
(see Def.~\ref{D:STRONGNULLCONDITION})
is based on null frames,
our current discussion on the geometry does not explicitly feature a null frame.\footnote{In particular, $\{L,\Rad,\GeoAng_1,\GeoAng_2\}$ is \emph{not} a $g$--null frame.} 
Nevertheless, one can construct a null frame out of $\{L,\Rad,\GeoAng_1,\GeoAng_2\}$ by defining the null vector\footnote{Note also that $\underline{L}$ as defined does \emph{not} remain regular with respect to the geometric coordinates
as the shock forms. That is, $\uLunit$ contains a component proportional to
$
\displaystyle
\frac{1}{\upmu}
\frac{\partial}{\partial u}
$.
In contrast, the rescaled vectorfield $\breve{\underline{L}}:=\upmu \underline{L}$ \emph{does} remain regular.
This also explains the form of the frame \eqref{E:SHOCKSNULLFRAME}
from Subsubsect.~\ref{SSS:SOMEREMARKSONTHEPROOF}.} $\underline{L}:=L+2\upmu^{-1}\bX$ and performing the Gram--Schmidt process on $\{\GeoAng_1,\GeoAng_2\}$ to obtain an orthonormal frame. As it turns out, 
the proof of shock formation can be carried out relative to
the rescaled frame $\left\lbrace L,\Rad,\GeoAng_1,\GeoAng_2 \right\rbrace$, which
is sufficiently adapted to the characteristics
and has all of the properties needed for capturing the good null structure in the equation.

Having introduced the geometric setup, we now provide
a more precise version of the 
evolution equation for $\upmu$ caricatured in
\eqref{E:UPMUEVOLUTIONCARICATURE}.
It plays a key role in the ensuing discussion.
The proof of the lemma is based on
decomposing the ``$0$ component'' of the
geodesic equation \eqref{E:LGEOISGEODESIC};
see \cite{jSgHjLwW2016} for the short proof.

\begin{lemma}[\textbf{Evolution equation for $\upmu$}]
\label{L:SCHEMATICEVOLUTIONEQUATIONINVERSEFOLIATIONDENSITY}
The inverse foliation density from Def.~\ref{D:UPMUDEF}
verifies the following evolution equation,
where the first product on the RHS is exactly depicted and
the second one is schematically depicted:
\begin{align} \label{E:SCHEMATICEVOLUTIONEQUATIONINVERSEFOLIATIONDENSITY}
	\Lunit \upmu
	& 
	= 
	\frac{1}{2} G_{\Lunit \Lunit} \Rad \Psi
	+
	\upmu \mathcal{O}(\Lunit \Psi),
\end{align}
where $G_{\Lunit \Lunit}  := G_{\alpha \beta} \Lunit^{\alpha} \Lunit^{\beta}$
and $G_{\alpha \beta} = G_{\alpha \beta}(\Psi)$ is defined in \eqref{E:LFLAT}.
\end{lemma}


\subsubsection{A quick summary of the proof of shock formation in the irrotational case}
\label{SSS:QUICKSUMMARYOFPROOFOFSHOCKFORMATION}
We now summarize the proof of shock formation for solutions
to equation \eqref{E:SCALARMODELWAVE} for perturbations of
simple outgoing plane symmetric outgoing waves 
(which we described in Footnote~\ref{FN:SIMPLEPLANE} and Subsubsect.~\ref{SSS:PRELIMINARYDICUSSION});
see \cite{jSgHjLwW2016} for the complete proof in the case of two spatial dimensions,
which can be extended to the current setting of three spatial dimensions
using the techniques established in
\cites{dC2007,jS2014b}. 

\begin{enumerate}
	\item (\textbf{Dynamic geometric tensors, adapted to} $u$)
		As in the proof of the stability of Minkowski spacetime \cite{dCsK1993}, 
		one constructs various geometric tensors adapted to the 
		eikonal function $u$. In particular, one constructs a set of 
		``commutation vectorfields''
		\begin{align} \label{E:COMMSET}
			\Fullset
			& := \lbrace \Lunit, \Rad, \GeoAng_1, \GeoAng_2 \rbrace,
		\end{align}
		used to differentiate the equations and obtain estimates
		for the solution's derivatives. The set 
		$\Fullset$ 
		(whose elements we constructed in Subsubsect.~\ref{SSS:GEOMETRICCONSTRUCTIONS})
		spans the tangent space of spacetime at each point with $\upmu > 0$.

		To close the proof of shock formation outlined in Subsubsect.~\ref{SSS:SOMEREMARKSONTHEPROOF},
		one needs non-degenerate $L^{\infty}$ estimates for the $\Fullset$ derivatives of $\Psi$ and various tensorfields
		up to a certain order. However, it turns out that
		due to the special structure of the equations relative to geometric coordinates,
		one can obtain sufficient energy estimates
		by commuting the wave equation only with 
		vectorfields belonging to the commutation subset
		\begin{align} \label{E:TANCOMMSET}
			\Tanset
			& := \lbrace \Lunit, \GeoAng_1, \GeoAng_2 \rbrace,
		\end{align}
		which spans the tangent space of the characteristics $\mathcal{P}_u$
		at each point. 

		One may check that for $Z \in \Fullset$,
		the Cartesian components $Z^{\alpha}$ depend on the 
		first Cartesian coordinate partial derivatives of $u$.
		We can schematically denote this by $Z \sim \partial u$.
		Hence, the regularity of the vectorfields themselves
		is tied to the regularity of $\Psi$ through the eikonal equation
		\eqref{E:EIKONAL}. It turns out that this simple fact generates
		enormous technical complications into the derivation
		of energy estimates; see Step (7).
	\item (\textbf{The geometric structure of the commuted wave equation})
	To close the proof of shock formation, 
	one needs to commute the wave equation 
	many times with the elements of $\Tanset$
	and then derive energy estimates for $\Psi$ up to top order.
	More precisely, one commutes
	the $\upmu-$weighted wave equation 
	$\upmu \square_{g(\Psi)} \Psi = 0$.
	We stress that \emph{it is important that the weight in the previous equation is precisely $\upmu$};
	the weight leads to important cancellations in commutation identities and, at the same time, 
	is compatible with various degenerate error terms 
	that one encounters energy estimates, as describe below.
	The main challenge is to bound the commutator terms,
	whose basic structure is revealed by the following commutation identity,
	written in schematic form
	except for the second product on the RHS
	(which is written exactly up to the overall sign):
	\begin{align} \label{E:COMMUTEDWAVE}
		\upmu \square_g(P \Psi) 
   & = P(\upmu \square_g \Psi)
   		+ \upmu^{-1}
				\left\lbrace
					\deformarg{P}{\Lunit}{\Rad}
					+ 
					P \upmu
				\right\rbrace
				(\upmu \square_g \Psi)
					\\
		& \ \
			+ 
   		\upmu \deform{P} \cdot \D^2 \Psi  
   		+ 
   		\upmu \D \deform{P} \cdot \D \Psi.
			\notag
	\end{align}
	In \eqref{E:COMMUTEDWAVE},
	$\D$ is the Levi-Civita connection of $g$,
	$\deformarg{P}{\alpha}{\beta}
	:= \D_{\alpha} P_{\beta} + \D_{\beta} P_{\alpha}
	$
	is the deformation tensor of $P$,
	and $\deform{P} \cdot \D^2 \Psi$ 
	and $\D \deform{P} \cdot \D \Psi$
	schematically denote tensorial contractions
	involving the derivatives of $\deform{P}$ and $\Psi$.
	There are many important cancellations in the products 
	on RHS~\eqref{E:COMMUTEDWAVE} due the special structure 
	of the elements of $\Tanset$. 
	The net effect is that 
	when one decomposes the differentiations on RHS~\eqref{E:COMMUTEDWAVE}  
	relative to the rescaled frame \eqref{E:RESCALEDFRAME},
	\emph{none of the terms involve any singular factors of} $1/\upmu$.
	For example, factors such as $(1/\upmu) \Rad \Rad \Psi$ do not appear.
	This is completely consistent with the philosophy that the solution
	should remain regular relative to the geometric coordinates
	and is closely tied our definition of the strong null condition
	(which, roughly speaking, posits that there are no terms involving
	two differentiations in the transversal $\Rad$ direction).
	We also note that from this perspective,
	the operator $\upmu \D$ on RHS~\eqref{E:COMMUTEDWAVE}
	should not be viewed as having an ``extra'' $\upmu$ weight;
	for by definition \eqref{E:RADANDULGOOD}, the factor of $\upmu$
	gets soaked up into the definition of $\Rad$ when performing 
	decompositions. In view of these considerations,
	we note that the second product on RHS~\eqref{E:COMMUTEDWAVE}
	could in principle, after more than one commutation,
	introduce a crippling factor of 
	$
	1/\upmu
	$ 
	into the commuted equations,
	which at the lower derivative levels would obstruct the
	goal of obtaining non-degenerate estimates.
	However, the vectorfields are constructed so that 
	the sum $\deformarg{P}{\Lunit}{\Rad} + P \upmu$
	completely vanishes.
	Note also that by the above remarks,
	the factor $\D \deform{P}$ on RHS~\eqref{E:COMMUTEDWAVE}
	depends on \emph{three} derivatives of $u$.
	As we will explain in Step (7), this simple fact 
	is the source of most of the difficulty in the proof.
	\item (\textbf{Size assumptions on the data})
		The main idea of \cite{jSgHjLwW2016} was to treat a regime in which the initial data have
		pure $\mathcal{P}_u-$transversal derivatives,
		such as
		$\Rad \Psi$ 
		and $\Rad \Rad \Psi$, that are  
		of size $\approx \mathring{\updelta} > 0$,
		while all other derivatives, such as 
		$\Singletan \Rad \Psi$,
		$\Singletan \Psi$,  
		and $\Psi$ itself, 
		are of \emph{small} size $\mathring{\upepsilon}$,
		where the smallness of $\mathring{\upepsilon}$
		is allowed to depend on $\mathring{\updelta}$.
		These size assumptions roughly correspond to perturbations
		of simple outgoing plane symmetric solutions.
		A key point is that, due to the special structure of the
		covariant wave operator $\square_{g(\Psi)}$
		when expressed relative to the geometric coordinates (see Def.~\ref{D:GEOMETRICCOORDS})
		and the special structure of the elements of $\Fullset$,
		it is possible to propagate this hierarchy all the way up to the shock,
		except at the very high derivative levels.
	\item ($L^{\infty}$ \textbf{bootstrap assumptions} for $\Psi$)
		One assumes that on a bootstrap region $\mathcal{M}_{\Tboot;U_0}$
		with $U_0 \approx 1$ (see \eqref{E:MTUDEF}), 
		on which the shock has not yet formed 
		(but perhaps is about to form),
		$\Psi$ and sufficiently many\footnote{Roughly speaking, if $N_{Top}$ denotes the maximum number of times
		that we need to commute the wave equation when deriving energy estimates,
		then we make the $L^{\infty}$ bootstrap assumptions for the 
		up-to-order $N_{Top}/2$ derivatives of $\Psi$.} 
		of its $\Tanset$ derivatives
		are bounded in $\| \cdot \|_{L^{\infty}}$ by $\leq \varepsilon$.
		At the end of the proof, one shows that the logic closes if
		$\varepsilon = C \mathring{\upepsilon}$ for a sufficiently large constant $C$.
	\item ($L^{\infty}$ \textbf{and pointwise estimates for many geometric objects})
		Using the bootstrap assumptions for $\Psi$ and the smallness of $\mathring{\upepsilon}$,
		one derives, on the region
		$\mathcal{M}_{\Tboot;U_0}$, non-degenerate $L^{\infty}$ estimates for the
		$\Tanset$ derivatives of the error terms appearing on RHS~\eqref{E:COMMUTEDWAVE}
		up to a certain order. The $L^{\infty}$ estimates can be derived by commuting various transport equations,
		including the evolution equation \eqref{E:SCHEMATICEVOLUTIONEQUATIONINVERSEFOLIATIONDENSITY}
		for $\upmu$, with the elements of $\Tanset$.
		The vast majority of these terms are 
		shown to be of small size $\mathcal{O}(\varepsilon)$ on $\mathcal{M}_{\Tboot;U_0}$.
		To derive energy estimates,
		one also needs to derive $L^{\infty}$ estimates for the 
		very low-level $\Rad$ derivatives of various tensorfields including $\upmu$. 
		These estimates are easy to obtain by studying various transport equations
		and using the $L^{\infty}$ estimates for the $\Tanset$ derivatives,
		but we will not focus on this issue here.
		Using these $L^{\infty}$ estimates, one then derives pointwise estimates
		for the error terms on RHS~\eqref{E:COMMUTEDWAVE}
		and the higher-order $\Tanset-$commuted analogs of RHS~\eqref{E:COMMUTEDWAVE},
		all the way up to top order. The pointwise estimates are needed in preparation
		for the energy estimates, which we describe below.
		In all of these estimates, it is important
		to decompose tensors relative to the rescaled frame \eqref{E:RESCALEDFRAME}.
	\item (\textbf{Showing that the shock forms})
		The main ideas behind showing that $\upmu$
		goes to $0$ in finite time and that a shock forms
		were already presented in Subsubsect.~\ref{SSS:SOMEREMARKSONTHEPROOF}.
		Here, we present them in more detail.
		Specifically, 
		given the 
		$L^{\infty}$ bootstrap assumptions for $\Psi$
		and the $L^{\infty}$ estimates,
		the proofs that $\upmu$
		goes to $0$ in finite time and that a shock forms
		are straightforward:
		one easily shows that the first product
		on RHS~\eqref{E:SCHEMATICEVOLUTIONEQUATIONINVERSEFOLIATIONDENSITY}
		is, relative to the geometric coordinates,
		approximately constant in time, that is, that
		$
		\displaystyle
		\Lunit
		\left(
		\frac{1}{2} G_{\Lunit \Lunit} \Rad \Psi
		\right) = \mathcal{O}(\mathring{\upepsilon})
		$.
		Moreover, one shows that the second product
		verifies
		$
		\upmu \mathcal{O}(\Lunit \Psi)
		= \mathcal{O}(\mathring{\upepsilon})
		$
		and is therefore a small error term.
		Hence, one obtains
		$
		[\Lunit \upmu](t,u,\vartheta) 
		= 
		\frac{1}{2} [G_{\Lunit \Lunit} \Rad \Psi](0,u,\vartheta)
		+
		\mathcal{O}(\mathring{\upepsilon})
		$
		and, by integrating in time 
		(see \eqref{E:LISDDT})
		and taking into account \eqref{E:UPMUINITIAL} and the initial $\mathcal{O}(\mathring{\upepsilon})$
		smallness of $\Psi$, 
		that\footnote{In the proof, one can allow the implicit
		constants in $\mathcal{O}(\mathring{\upepsilon})$ to depend on $t$. The reason is that
		it is possible to make a good guess about the 
		(data-dependent)
		time of blow up. Therefore, one needs only
		to propagate estimates for an amount of time that can be estimated to high accuracy in advance.}
		$\upmu(t,u,\vartheta)
		= 1 
		+ \frac{1}{2} t [G_{\Lunit \Lunit} \Rad \Psi](0,u,\vartheta)
		+
		\mathcal{O}(\mathring{\upepsilon})
		$.
		Hence, for data such that the term
		$
		\displaystyle
		\frac{1}{2} t [G_{\Lunit \Lunit} \Rad \Psi](0,u,\vartheta)
		$
		is sufficiently negative 
		at some value of $(u,\vartheta)$
		(in particular, large enough in magnitude to dominate the term $\mathcal{O}(\mathring{\upepsilon})$),
		one concludes that $\upmu$ vanishes in finite time. Moreover,
		one shows that in a past neighborhood of any point where $\upmu$ vanishes, we have
		$|\Rad \Psi| \gtrsim 1$, or equivalently, that $|\Radunit \Psi| := |\Radunit^a \partial_a \Psi| \gtrsim 1/\upmu$.
		Since $\Radunit$ is comparable to a Cartesian coordinate partial derivative $\partial$,
		we conclude that when $\upmu$ vanishes,
		$|\partial \Psi|$ blows up like 
		$
		\displaystyle
		\frac{1}{\upmu}
		$.
		\item (\textbf{Energy estimates up to top order})
		To improve the $L^{\infty}$ bootstrap assumptions for $\Psi$
		(which were used in particular in Step (6) to prove that $\upmu$ vanishes in finite time),
		the main task is to derive energy estimates for sufficiently many 
		$\Tanset$ derivatives of $\Psi$
		that \emph{do not degenerate} as $\upmu$ goes to $0$;
		one can then use Sobolev embedding to obtain, under a suitable smallness
		assumption on the data, the desired $L^{\infty}$ bounds.
		We stress again that one does not need to derive energy estimates for the $\Rad-$commuted wave equation.
		This is helpful because in the solution regime under study, the
		energies corresponding to the $\Tanset-$commuted equation are all initially of small
		size $\mathcal{O}(\mathring{\upepsilon}^2)$
		(while the energies for the $\Rad-$commuted wave equation are of
		the larger size $\mathcal{O}(\mathring{\updelta}^2)$).

		Note that the energy estimates for $\Psi$ are coupled to those for $u$
		in view of the fact that the vectorfield commutators $P \in \Tanset$
		for the wave equation depend on $\partial u$.
		To derive energy estimates for $\Psi$ and $u$ up to top order,
		one uses the pointwise estimates from Step (4) and
		a geometric energy method.
		More precisely, with $\Tanset^M$ denoting an arbitrary $M^{th}-$order
		string of vectorfields constructed out of the elements of $\Tanset$,
		one derives energy estimates for $\Tanset^M \Psi$ for $1 \leq M \leq N_{Top}$
		(with $N_{Top}$ sufficiently large).\footnote{
		For reasons described to below, the proofs require many derivatives.
		In \cite{jSgHjLwW2016},
		the authors showed that the proof closes if, at time $0$, 
		$\Psi \in H^{19}$
		and $\partial_t \Psi \in H^{18}$
		(with suitable tensorial smallness assumptions enforcing the 
		$\mathring{\upepsilon}-\mathring{\updelta}$ hierarchy described above).}
		One can derive the basic energy identity 
		by using a suitable version of the vectorfield multiplier method, 
		that is, by applying the divergence theorem
		on regions of the form $\mathcal{M}_{t,u}$ (see Figure~\ref{F:SOLIDREGION})
		to the energy current vectorfield\footnote{Here we raise and lower indices with
		$g^{-1}$ and $g$.}
		$J^{\alpha}[\Tanset^M \Psi] := \enmomtensor_{\ \beta}^{\alpha}[\Tanset^M \Psi] \Mult^{\beta}$.
		Here,
		$
		\enmomtensor_{\mu \nu}[\Tanset^M \Psi]
			:= (\D_{\mu} \Tanset^M \Psi) \D_{\nu} \Tanset^M \Psi
			- 
			\frac{1}{2} g_{\mu \nu} (\D^{\alpha} \Tanset^M \Psi) 
			(\D_{\alpha} \Tanset^M \Psi)
		$ 
		is the energy-momentum tensorfield of $\Tanset^M \Psi$
		and $\Mult := (1 + 2 \upmu) \Lunit + 2 \Rad$ is a \emph{multiplier vectorfield}, 
		which is timelike\footnote{That is, $g(\Mult,\Mult) = - 4 \upmu(1+\upmu) < 0$.}
		with respect to $g$. The careful placement of the $\upmu$ weights in the definition
		of $\Mult$, both the explicit one and the implicit one inherent in the relation 
		$\Rad = \upmu \Radunit$,
		are essential for generating suitable energies. More precisely,
		the divergence theorem yields an integral identity
		involving coercive ``energies'' 
		$
		\displaystyle
		\mathbb{E}[\Tanset^M \Psi](t,u)
		:= 
		\int_{\Sigma_t^u}
				J_{\alpha}[\Tanset^M \Psi] \Transport^{\alpha}
		$
		(where the material derivative vectorfield $\Transport$ is the future-directed normal to $\Sigma_t$)
		and also\footnote{In the interest of brevity, 
		we have avoided discussing the volume forms corresponding to the integrals
		$
		\int_{\Sigma_t^u}
		\cdots
		$
		and
		$
		\int_{\mathcal{P}_u^t}
		\cdots
		$.
		Let us simply note that the implicit forms are non-degenerate in the sense
		that relative to the geometric coordinates,
		they remain uniformly bounded from above and below (strictly away from zero),
		all the way up to the shock.
		\label{FN:NOTWRITINGFORMS}} 
		coercive ``null fluxes'' 
		$
		\displaystyle
		\mathbb{F}[\Tanset^M \Psi](t,u)
		:= 
		\int_{\mathcal{P}_u^t}
				J_{\alpha}[\Tanset^M \Psi] \Lunit^{\alpha}
		$,
		both of which are needed to close the estimates. 
		The integral identity, which forms the starting
		point for the $L^2-$type analysis, 
		follows from integrating the following
		divergence identity over regions of the form $\mathcal{M}_{t,u}$
		(see \eqref{E:MTUDEF}) with respect to a suitable volume form:
		\begin{align} \label{E:DIVID}
		\upmu \D_{\alpha} J^{\alpha}[\Tanset^M \Psi] 
		= 
		\Mult \Tanset^M \Psi
		\cdot 
		(\upmu \square_{g(\Psi)} \Tanset^M \Psi)
		+ 
		\frac{1}{2}
		\upmu \enmomtensor^{\alpha \beta}[\Psi] \deformarg{\Mult}{\alpha}{\beta},
		\end{align}
		where 
		$\deformarg{\Mult}{\alpha}{\beta}
		 := 
		\D_{\alpha} {\Mult}_{\beta} + \D_{\beta} {\Mult}_{\alpha}
		$
		is the deformation tensor of $\Mult$.
	The coerciveness of the energies and null fluxes
	are consequences of
	the dominant energy condition,
		which is the property 
		$\enmomtensor_{\alpha \beta}[\Tanset^M \Psi] V^{\alpha} W^{\beta} \geq 0$
		whenever $V$ and $W$ are future-directed\footnote{Recall that $V$ being future-directed
		simply means that $V^0 > 0$, where $V^0$ is the Cartesian time component of $V$.} 
		causal\footnote{$V$ being causal means that $g(V,V) \leq 0$.} 
		vectorfields.

		In \cite{jSgHjLwW2016}, readers may find a detailed
		description of how to close the energy estimates
		and why they are difficult to derive. Here,
		we highlight some of the main ideas.
		See also our companion article \cite{jLjS2016b}
		for a similar discussion in the context of 
		solutions to the compressible Euler equations 
		in two spatial dimensions with vorticity.
		The energy estimates are difficult to derive
		for two main reasons. First, a careful computation reveals that
		the energies $\mathbb{E}[\Tanset^M \Psi](t,u)$
		and the null fluxes
		$\mathbb{F}[\Tanset^M \Psi](t,u)$
		contain $\upmu$ weights, 
		inherited from the $\upmu$ weights found in the definition of $\Mult$. 
		Consequently,  
		$\mathbb{E}[\Tanset^M \Psi]$
		and
		$\mathbb{F}[\Tanset^M \Psi]$
		provide only very weak control
		over certain directional derivatives near the shock (where $\upmu$ is small).
		This is a serious difficulty because one encounters ``strong'' error terms 
		in the energy estimates arising, for example, from error terms
		on the RHS of the wave equations\footnote{As we indicated in \eqref{E:COMMUTEDWAVE},
		it is important to commute the $\upmu-$weighted wave equation.} 
		$\upmu \square_{g(\Psi)} \Tanset^M \Psi = \cdots$
		that \emph{do not have $\upmu$ weights}.
		To control such strong error terms, one needs to exploit various special structures.
		For example, one relies on the availability of a subtle spacetime integral 
		with a favorable ``friction-type'' sign,
		first identified by Christodoulou \cite{dC2007},
		that is generated by the term 
		$
		\displaystyle
		\frac{1}{2}
		\upmu \enmomtensor^{\alpha \beta}[\Psi] \deformarg{\Mult}{\alpha}{\beta}
		$
		on RHS~\eqref{E:DIVID} 
		(and which also appears in the energy identities). 
		Through a detailed analysis of $\upmu$ and $\Lunit \upmu$,
		the spacetime integral can be shown to be strong in regions where $\upmu$ is small,
		thanks to the negativity\footnote{That is, a key part of the proof involves showing that
		when $\upmu$ is small, $\Lunit \upmu$ must be quantitatively negative.} 
		of $\Lunit \upmu$.
		In fact, the spacetime integral can be used to absorb many of the ``strong'' error terms.

		The second reason that the energy estimates are difficult is that, 
		as we mentioned above, the top-order derivatives of $u$ are hard to estimate.
		To further explain this difficulty, we will count derivatives; 
		it suffices to explain the difficulty that arises after commuting the wave equation 
		one time, as we did in equation \eqref{E:COMMUTEDWAVE}.
		To proceed, we recall that to control RHS~\eqref{E:COMMUTEDWAVE}, we must bound
		\emph{three} derivatives of $u$ in $L^2$. The naive way to achieve this goal is
		to commute the eikonal equation \eqref{E:EIKONAL} 
		with three derivatives, schematically denoted by ``$\partial^3$'',
		to obtain the schematic evolution equation
		$(g^{-1})^{\alpha \beta}(\Psi) \partial_{\alpha} u \partial_{\beta} \partial^3 u 
		= \partial^3 \Psi \cdot \partial u
		+
		l.o.t
		$.
		The problem with this approach is that the RHS of this evolution equation 
		for $\partial^3 u$
		depends on \emph{three} derivatives of $\Psi$, which is inconsistent with the
		regularity of $\Psi$ obtained from
		deriving energy estimates for equation \eqref{E:COMMUTEDWAVE}
		(the energy estimates yield control over only two derivatives of $\Psi$ in $L^2$).
		Clearly this loss of a derivative cannot be overcome by further commuting the equations.
		To overcome it, one uses strategies employed in \cite{dCsK1993} and later in
		\cite{sKiR2003}, including using that the special structure of the vectorfields in
		$\Tanset$ leads to the absence of the worst imaginable top-order derivatives of $u$.
		That is, the third derivatives of $u$ appearing on RHS~\eqref{E:COMMUTEDWAVE}
		have special tensorial structures. One can bound these terms
		by exploiting the tensorial structures with the help of modified quantities and,
		by using elliptic estimates\footnote{Note that in two spatial dimensions,
		the $\ell_{t,u}$ are one-dimensional, and it turns out that the energy estimates close
		without elliptic estimates.}
		on 
		$\ell_{t,u}$.
		By ``modified quantities,'' we mean that one finds special combinations of
		the derivatives of $\Psi$ and $u$ that satisfy a good evolution equation
		allowing one, with the help of the aforementioned elliptic estimates, 
		to avoid the loss of a derivative.
		To construct the modified quantities and derive elliptic estimates, 
		one needs precise geometric decompositions 
		of the derivatives of $\Psi$ and $u$,
		adapted to the acoustic characteristics $\mathcal{P}_u$.

		In carrying out the above scheme, one encounters another serious difficulty:
		it turns out that introducing the modified quantities, 
		which are essential to avoid losing a derivative at the top order,
		leads to the presence
		of a difficult factor of $1/\upmu$ into the top order energy identities.
		We may caricature\footnote{In equation \eqref{E:CARICENERGYINEQ},
		we have ignored the presence of other difficult terms that lead to related but distinct difficulties.} 
		the effect of this factor in the basic energy inequality
		as follows, where $\mathbb{E}_{Top}$ denotes the top-order energy along $\Sigma_t$
		and $A$ is a universal positive constant, 
		\emph{independent of the number of times that the equations are commuted}:
		\begin{align} \label{E:CARICENERGYINEQ}
			\mathbb{E}_{Top}(t)
			\leq 
			\mathbb{E}_{Top}(0)
			+
			A
			\int_{s=0}^t
				\sup_{\Sigma_s}
				\left|
					\frac{\frac{\partial}{\partial t} \upmu}{\upmu}
				\right|
				\cdot
				\mathbb{E}_{Top}(s)
			\, ds
			+
			\cdots.
		\end{align}
		To derive a Gronwall estimate for $\mathbb{E}_{Top}(t)$,
		we need to bound
		$
		\displaystyle
		\sup_{\Sigma_s}
				\left|
					\frac{\frac{\partial}{\partial t} \upmu}{\upmu}
				\right|
		$.
		The true estimate is lengthy to state, but 
		it can be caricatured as follows:
		$
		\displaystyle
		\sup_{\Sigma_s}
				\left|
					\frac{\frac{\partial}{\partial t} \upmu}{\upmu}
				\right|
		\leq 
		\frac{\TrandatasizeWithFactor}{1 - \TrandatasizeWithFactor s}
		$,
		where
		$
		\TrandatasizeWithFactor > 0
		$
		is a constant depending on the data.
		Moreover, with
		\[
		\upmu_{\star}(s) := \min_{\Sigma_s} \upmu,
		\]
		one can prove an estimate that can be caricatured as
		\[
		\upmu_{\star}(s) \sim 1 - \TrandatasizeWithFactor s.
		\]
		Hence, from Gronwall's inequality\footnote{We stress that since the factor $1/\upmu$ 
		is present in the energy identities, one must derive
		detailed information about the way that $\upmu_{\star}$ 
		vanishes in order to close the energy estimates.
		The reason is that the vanishing rate of $\upmu_{\star}$ is tied to the blowup-rate
		of the high-order energies. In particular, it is crucially important that $\upmu_{\star}$
		goes to $0$ \emph{linearly}, as is captured by our caricature estimate.}
		one obtains an a priori estimate than can be caricatured as follows:
		\[
		\mathbb{E}_{Top}(t)
		\leq 
		\mathbb{E}_{Top}(0)
		\cdot
		\upmu_{\star}^{-A}(t)
		+ 
		\cdots.
		\]
		In particular, $\mathbb{E}_{Top}(t)$ can blow up as $\upmu$ vanishes.

		As we have noted, the argument sketched above relies on
		having non-degenerate $L^{\infty}$ estimates at the lower derivative levels,
		which are needed to control various error terms.
		Thus, the only hope of validating the above estimate for
		$\mathbb{E}_{Top}(t)$
		and thus closing the problem 
		is to derive \emph{less degenerate} energy
		estimates below top order. For if the above degenerate energy estimate 
		were the best one we could prove at all derivative levels,
		then we would not be able to recover the $L^{\infty}$ bootstrap assumptions
		from Step (4); such degenerate energy estimates, when combined with Sobolev
		embedding, would yield only that the $L^{\infty}$ norms of the low-order derivatives
		of $\Psi$ can also blow up as
		$\upmu_{\star}$ vanishes, which would completely obstruct our efforts 
		to justify the non-degenerate estimates at the lower derivative levels.
		To overcome this difficulty, one exploits the fact that below top order,
		it is permissible to allow the aforementioned loss of one derivative
		in the difficult wave equation error terms that depend on the eikonal function.
		In allowing the loss, one can avoid using modified quantities and
		thus avoid introducing the factor of $1/\upmu$ 
		into the below-top-order energy identities.
		The price one pays is that this approach
		couples the below-top-order energy identities to the degenerate top-order ones.
		Nonetheless, this approach allows one to derive less degenerate
		estimates below top order.
		More precisely, 
		via an energy estimate ``descent scheme,''
		based on successively reducing the strength of 
		the singularity via the estimate\footnote{This estimate is just a quasilinear version of the bound
		$\int_{s=0}^t s^{-B} \, ds \lesssim t^{1 - B}$, where $s=0$ represents the ``vanishing'' of $\upmu$.
		In the proof of the estimate, it is again critically important that
		$\upmu_{\star}$ vanishes linearly.} 
		$
		\displaystyle
		\int_{s=0}^t
			\upmu_{\star}^{-B}(s)
		\, ds
		\lesssim
		\upmu_{\star}^{1-B}(t)
		$,
		one can show that the below-top-order energies satisfy a hierarchy of 
		successively less degenerate estimates of the form
		\[
		\mathbb{E}_{Top-1}(t)
		\leq 
		\mbox{\upshape data}
		\cdot
		\upmu_{\star}^{-(A-2)}(t),
		\]
		\[
		\mathbb{E}_{Top-2}(t)
		\leq 
		\mbox{\upshape data}
		\cdot
		\upmu_{\star}^{-(A-4)}(t),
		\]
		\[
		\cdots,
		\]
		until one reaches a ``middle level,'' below which all energies are bounded:
		\[
		\mathbb{E}_{Mid}(t),
		\mathbb{E}_{Mid-1}(t),
		\mathbb{E}_{Mid-2}(t),
		\cdots,
		\mathbb{E}_1(t)
		\leq 
		\mbox{\upshape data}.
		\]
		From these non-degenerate estimates, Sobolev embedding,
		and a small data assumption,
		one can finally improve the $L^{\infty}$ bootstrap assumptions for $\Psi$, 
		which closes the whole process. The large number of derivatives needed\footnote{In \cite{jSgHjLwW2016},
		the authors commuted the wave equations $18$ times in order
		close the estimates.} 
		to close
		the proof is due to the large number times that one needs to descend below
		top order in order to reach the non-degenerate energies $\mathbb{E}_{Mid}(t)$.
		\end{enumerate}


\subsection{A preview of the proof of shock formation for solutions to the compressible Euler equations
in the presence of vorticity}
\label{SS:PREVIEWONSHOCKS}
In Subsubsect.~\ref{SSS:QUICKSUMMARYOFPROOFOFSHOCKFORMATION},
we overviewed a framework, based on techniques introduced by Christodoulou, 
for proving shock formation in solutions to 
quasilinear wave equations with suitable nonlinearities, 
a special case of which is the irrotational compressible Euler equations.
We now overview, without giving proofs, 
how the new geometric and analytic structures
provided by Theorems~\ref{T:GEOMETRICWAVETRANSPORTSYSTEM} and \ref{T:STRONGNULL}
fit in with the above framework and
allow one to prove shock formation in the presence of vorticity.
We restrict our attention to discussing
solutions that are close to simple outgoing plane symmetric outgoing solutions
(see Subsect.~\ref{SSS:PRELIMINARYDICUSSION})
in three spatial dimensions with spatial topology\footnote{The factor $\mathbb{T}^2$ corresponds to perturbations
away from plane symmetry.} 
$\Sigma_t = \mathbb{R} \times \mathbb{T}^2$.
This is the solution regime that we analyze in detail in our forthcoming
works \cites{jLjS2016b,jLjS2017}
(where the spatial topology is $\Sigma_t = \mathbb{R} \times \mathbb{T}^2$
in the case of two spatial dimensions treated in \cite{jLjS2016b}).

\subsubsection{The regime under consideration: solutions close to simple outgoing plane symmetric waves}
\label{SSS:SOLUTIONREGIME}
	We expect that our framework for proving finite-time shock formation 
	can be applied to various kinds of initial data, 
	including small compactly supported nearly spherically symmetric data.
	However, we restrict our attention here to the simplest vorticity-containing 
	solutions to which our framework applies: 
	(non-symmetric) perturbations of simple outgoing\footnote{We recall that, roughly speaking, outgoing means
	right-moving, as is depicted in Figure~\ref{F:FRAME}.} 
	plane symmetric solutions.
	Note that simple outgoing plane symmetric 
	solutions themselves are irrotational, but perturbations of them generally 
	have non-zero vorticity. We have already discussed simple outgoing plane symmetric 
	solutions in Subsubsect.~\ref{SSS:QUICKSUMMARYOFPROOFOFSHOCKFORMATION},
	in the context of the scalar wave equation \eqref{E:SCALARMODELWAVE}.
	However, in preparation for the subsequent discussion, 
	we now give a slightly different description of such solutions
	in the context of the compressible Euler equations.
	To this end, we will rely on Riemann's famous method \cite{bR1860} of Riemann invariants.
	Our discussion applies to any equation of state except for the one
	$
	\displaystyle
	p = C_0 - \frac{C_1}{\rho}  = C_0 - C_1 \exp(-\Densrenormalized)
	$
	corresponding to a Chaplygin gas,\footnote{The Chaplygin gas equation of state corresponds to the exceptional Lagrangian
	mentioned in Footnote~\ref{FN:EXCEPTIONALLAGRANGIANS}
	In plane symmetry, the Riemann invariants 
	$\mathcal{R}_{\pm}$ 
	for the Chaplygin gas
	solve a \emph{totally linearly degenerate} system
	of PDEs, which is not expected to exhibit shock formation; 
	see \cite{aM1984} for more discussion on totally linearly degenerate systems.} 
	(where $C_0$ and $C_1$ are arbitrary constants);
	see also Footnote~\ref{FN:NOSHOCKSFORCHAPLYGIN} for a description of why our arguments
	do not apply to the Chaplygin gas.
	By a ``plane symmetric'' solution,
	we mean that relative to Cartesian coordinates, we have
	$\Densrenormalized = \Densrenormalized(t,x^1)$,
	$v^1 = v^1(t,x^1)$,
	and
	$v^2 = v^3 \equiv 0$.

	In plane symmetry, the compressible
	Euler equations are equivalent to the system
	\[
		\uLunit \mathcal{R}_- = 0, \qquad \Lunit \mathcal{R}_+ = 0,
	\]
	where 
	$\mathcal{R}_{\pm} := v^1 \pm F(\Densrenormalized)$,
	$F$ is defined by $F'(\Densrenormalized) = c_s(\Densrenormalized)$
	with $F(0) = 0$,
	where the latter is a convenient normalization condition.
	Moreover, we have the explicit formulas
	\[
	\uLunit = \partial_t + (v^1 - \Speed) \partial_1,
	\qquad
	\Lunit = \partial_t + (v^1 + \Speed) \partial_1.
	\]
	We will study simple plane symmetric solutions such that
	$\Densrenormalized$ and $v^1$ (undifferentiated) are near $0$.
	As we explained below equation \eqref{E:BACKGROUNDDENSITY},
	the background solution 
	$(\Densrenormalized,v^1) \equiv (0,0)$
	corresponds to a constant state with non-zero density $\bar{\rho} > 0$,
	which is an analog of the global background solution
	$\Psi \equiv 0$ from Subsubsect.~\ref{SSS:QUICKSUMMARYOFPROOFOFSHOCKFORMATION}.
	In terms of the Riemann invariants, the background solution 
	takes the form $\mathcal{R}_- = \mathcal{R}_+ \equiv 0$.
	By a ``simple'' plane symmetric solution in the present context, 
	we mean that one of the Riemann invariants,
	say $\mathcal{R}_-$, completely vanishes.
	Roughly, ``simple'' means that there is only (say) a right-moving
	(that is, outgoing)
	wave rather than a combination of left-moving and right-moving waves.

	We now discuss the formation of shocks in simple plane symmetric solutions
	in which $\mathcal{R}_-$ is identically zero.
	Applying Riemann's methods to the evolution equation
	$\Lunit \mathcal{R}_+ = 0$,
	we easily deduce by differentiating the equation with $\partial_1$
	that for suitable smooth initial conditions,
	$\partial_1 \mathcal{R}_+$ experiences a Riccati-type blowup 
	along a characteristic\footnote{Here, in the setting of plane symmetry,
	a characteristic is simply an integral curve of $\Lunit$.} 
	while $\mathcal{R}_+$ remains bounded.
	That is, a shock forms\footnote{The shock-formation 
	argument does not work for the Chaplygin gas. 
	The reason is that for this equation of state,
	we have $\Lunit = \partial_t + (\mathcal{R}_- + C) \partial_1$
	(where $C$ is a constant), while $\mathcal{R}_- \equiv 0$ by assumption.
	That is, the evolution equation $\Lunit \mathcal{R}_+ = 0$
	is effectively semilinear in this case.
	\label{FN:NOSHOCKSFORCHAPLYGIN}} 
	through a mechanism similar to the one
	that drives singularity formation in solutions to Burgers' equation.
	In the rest of Subsect.~\ref{SS:PREVIEWONSHOCKS}, we also assume that 
	$\| \mathcal{R}_+ \|_{L^{\infty}(\Sigma_0)}$ 
	is small, a condition that is propagated by the flow of the equations.
	Then perturbations (away from plane symmetry)
	of the corresponding shock-forming solution will be 
	$L^{\infty}-$close (at least initially)
	to the constant state solution described in the previous paragraph;
	this $L^{\infty}-$closeness assumption is convenient
	but could most likely be relaxed.

		\subsubsection{Elements of the proof and the size of the data}
		\label{SSS:FEATURESINCOMMONWITHIRROTATIONALCASE}
		We now describe how to prove shock formation for
		perturbations of the simple outgoing plane symmetric solutions described in the previous subsubsection, 
		where the perturbations belong to a suitable Sobolev space (without symmetry assumptions).
		Although the method of Riemann invariants is convenient for generating a family of shock-forming
		solutions, it is not applicable to perturbations away from plane symmetry.
		Hence, to study the perturbed solutions,
		we will use the formulation of the equations provided by
		Theorem~\ref{T:GEOMETRICWAVETRANSPORTSYSTEM}.

		A large part of the proof consists of the same steps described in
		Subsubsect.~\ref{SSS:QUICKSUMMARYOFPROOFOFSHOCKFORMATION}.
		The reason is that $\lbrace v^i \rbrace_{i=1,2,3}$ and $\Densrenormalized$ solve the covariant
		wave equations \eqref{E:VELOCITYWAVEEQUATION}-\eqref{E:RENORMALIZEDDENSITYWAVEEQUATION},
		which are similar to the wave equations discussed in
		Subsubsect.~\ref{SSS:QUICKSUMMARYOFPROOFOFSHOCKFORMATION}.
		In particular, it is straightforward to show that 
		for perturbations of the simple outgoing plane symmetric solutions described above,
		the initial data for the ``wave variables''
		$v^i$ and $\Densrenormalized$ verify 
		$\mathring{\upepsilon}-\mathring{\updelta}$
		size assumptions that are similar to the ones described 
		in Step (3) of Subsubsect.~\ref{SSS:QUICKSUMMARYOFPROOFOFSHOCKFORMATION}.
		The new feature in the analysis here
		is the presence of inhomogeneous terms in the wave equations.
		That is, if not for the inhomogeneous terms on 
		RHSs~\eqref{E:VELOCITYWAVEEQUATION}-\eqref{E:RENORMALIZEDDENSITYWAVEEQUATION},
		the proof outlined in
		Subsubsect.~\ref{SSS:QUICKSUMMARYOFPROOFOFSHOCKFORMATION} 
		would go through without significant changes.
	Our main goal in Subsect.~\ref{SS:PREVIEWONSHOCKS} is to explain why, 
	under a suitable $\mathring{\upepsilon}-\mathring{\updelta}$
	smallness-largeness hierarchy similar to the one described 
	in the irrotational case in Subsubsect.~\ref{SSS:QUICKSUMMARYOFPROOFOFSHOCKFORMATION},
	a shock forms in the solution.
	Our assumptions on the data of the specific vorticity are that
	$\Vortrenormalized$ and \emph{all} of its derivatives up to top order
	are initially of small size $\mathcal{O}(\mathring{\upepsilon})$
	(as measured by appropriate norms).
	We aim to propagate the smallness of
	$\Vortrenormalized$ and to show that it 
	\emph{does not interfere with the shock formation processes}
	described in Subsubsect.~\ref{SSS:QUICKSUMMARYOFPROOFOFSHOCKFORMATION}.
	That is, the dynamics are not significantly distorted by the presence of 
	small amounts of vorticity.
	One of course expects that, 
	like the high-order energies in the irrotational case,
	the $L^2$ norms of the high-order derivatives of
	$\Vortrenormalized$ can blow up as $\upmu \to 0$,
	and that controlling their blowup-rates
	is at the heart of closing the proof.

	We also note that it is straightforward to 
	construct data such that the solution has non-vanishing vorticity
	at the shock.
	To explain this, we first note that
	it is easy to perturb the data of
	the simple plane symmetric solutions described above
	so that $\Vortrenormalized|_{\Sigma_0}$
	is everywhere non-zero.
	Then using the evolution equation \eqref{E:RENORMALIZEDVORTICTITYTRANSPORTEQUATION},
	it is straightforward to show that
	$\Vortrenormalized$ remains
	strictly non-zero,\footnote{To prove this, one 
	considers $\upmu \times \mbox{\eqref{E:RENORMALIZEDVORTICTITYTRANSPORTEQUATION}}$.
	It is easy to show that $\upmu \Transport = \frac{d}{du}$
	along the integral curves of $\upmu \Transport$ and that
	the factors $\upmu \partial_a v^i$ 
	on the RHS remain uniformly bounded all the way up to the shock.
	We can therefore view $\upmu \times \mbox{\eqref{E:RENORMALIZEDVORTICTITYTRANSPORTEQUATION}}$
	as a linear ODE in $\Vortrenormalized$ with regular coefficients, 
  and by the uniqueness of the $0$ solution,
	$\Vortrenormalized$ never vanishes.
	} 
	all the way up to the shock as desired.

	\subsubsection{Geometric vectorfields and their interaction with the transport operator}
	\label{SSS:GEOMETRICVECINTERACTWITHTRANSPORT}
		As we described in Subsubsect.~\ref{SSS:QUICKSUMMARYOFPROOFOFSHOCKFORMATION}, 
		to close the proof of shock formation, it is critically important to construct, 
		with the help of an eikonal function corresponding to 
		the acoustical metric $g$ (defined in \eqref{E:ACOUSTICALMETRIC}), 
		a set of commutation vectorfields 
		that are adapted to the acoustic characteristics;
		we need such vectorfields to control solutions
		to the wave equations \eqref{E:VELOCITYWAVEEQUATION}-\eqref{E:RENORMALIZEDDENSITYWAVEEQUATION}.
		A key observation of the present work is that 
		the elements of $\Fullset$ exhibit good commutation properties with
		$\upmu \partial_{\alpha}$, where 
		$\partial_{\alpha}$, ($\alpha = 0,1,2,3$), is any Cartesian coordinate partial derivative vectorfield.
		By this, we mean that for $Z \in \Fullset$ and scalar functions $f$,
		we can bound\footnote{Here $[P,Q]$ denotes the commutator of the differential operators $P$ and $Q$.} 
		$[Z, \upmu \partial_{\alpha}] f$
		in terms of the first-order $\Fullset$ derivatives of $f$
		\emph{without any dangerous factors of $1/\upmu$ appearing}.
		Schematically, we may express this by
		$[\Fullset, \upmu \partial_{\alpha}] \sim \Fullset$.
		To explain this in more detail, we first define
		$\gamma_{AB}:=g(\GeoAng_A, \GeoAng_B)=g_{ab} \GeoAng_A^a \GeoAng_B^b$ for $A,B=1,2$.
		We also define $\gamma^{-1}_{AB}$ to be the inverse of the $(2 \times 2)$ matrix 
		$\gamma_{AB}$.
		Then we have the following identity for vectorfields:
		\begin{align} \label{E:CARTESIANINTERMSOFGEOMETRIC}
		\upmu \partial_i
		=
		(g_{ai} \Radunit^a) \Rad
			+ 
			\upmu
			\sum_{A,B=1}^2\left(
				\gamma^{-1}_{AB} g_{ai} \GeoAng_A^a
			\right)
			\GeoAng_B.
		\end{align}

	 The reason that the commutator $[Z,\upmu \partial_i]$ is controllable
		is that the elements of $\Fullset$ are designed to have good commutators
		with each other 
		while the $\Fullset-$derivatives of the scalar functions
		$g_{ai} \Radunit^a$
		and
		$
		\displaystyle
		\upmu
			\left(
				\gamma^{-1}_{AB} g_{ai} \GeoAng_A^a
			\right)
		$
		on RHS~\eqref{E:CARTESIANINTERMSOFGEOMETRIC}
		are simple error terms;
		the $\Fullset-$derivatives of the Cartesian components $g_{ab}$ can be
		controlled in terms of the $\Fullset$-derivatives of $\Densrenormalized$ and the Cartesian components $v^a$
		(as is evident from the formula \eqref{E:ACOUSTICALMETRIC}),
		while the $\Fullset$-derivatives of $\upmu$ and the Cartesian components $\Radunit^a$
		and $\GeoAng_A^a$ can be estimated by analyzing
		solutions to transport equations
		(in the spirit of \eqref{E:SCHEMATICEVOLUTIONEQUATIONINVERSEFOLIATIONDENSITY}).
		Notice that this is in contrast to the commutators $[Z, \partial_{\alpha}]$;
		these commutators involve the $\Fullset-$derivatives of $\upmu^{-1}$, 
		which are not uniformly bounded up to the shock.

		It is because the commutators $[Z, \upmu \partial_{\alpha}]$ are good
		that we can successfully commute 
		($\upmu-$weighted versions of) the \emph{first-order} equations
		\eqref{E:RENORMALIZEDVORTICTITYTRANSPORTEQUATION}
		\eqref{E:FLATDIVOFRENORMALIZEDVORTICITY},
		and
		\eqref{E:EVOLUTIONEQUATIONFLATCURLRENORMALIZEDVORTICITY}
		to obtain control of the specific vorticity (see also Subsubsect.~\ref{sss.inho.vor})
		all the way up to the shock.
		More precisely, consistent with our above remarks,
		we view the scalar function $(\Flatcurl \Vortrenormalized)^i$ to be the unknown in equation
		\eqref{E:EVOLUTIONEQUATIONFLATCURLRENORMALIZEDVORTICITY}
		and hence we commute 
		\emph{only through the outer operator} $\upmu \Transport$
		in the $\upmu-$weighted version of equation
		\eqref{E:EVOLUTIONEQUATIONFLATCURLRENORMALIZEDVORTICITY}.

		We stress that the strategy described above applies
		\emph{only to commutations through first-order operators};
		although the vectorfields $Z \in \Fullset$ 
		also (by design) commute well through 
		$\upmu \square_g$
		(see \eqref{E:COMMUTEDWAVE} and the remarks below it),
		their commutator with a typical $\upmu-$weighted
		second-order differential operator,
		such as $\upmu \partial_{\alpha} \partial_{\beta}$,
		produces error terms of the schematic form 
		$(1/\upmu) \Fullset \Fullset$, which are large near the shock;
		such error terms would prevent us from deriving 
		non-degenerate estimates at the low derivative levels,
		which are essential for closing the problem.

\subsubsection{Inhomogeneous terms in the specific vorticity equation}\label{sss.inho.vor}
The equation for the specific vorticity \eqref{E:RENORMALIZEDVORTICTITYTRANSPORTEQUATION} has the inhomogeneous term $\Vortrenormalized^a\rd_a v^i$. 
In order to bound $\Vortrenormalized^i$ uniformly in $L^{\infty}$,
we need to control the integral of $\rd_a v^i$ along the integral curves of the transport operator $\Transport$. Even in the irrotational setting, the quantity $\rd_a v^i$ can blow up like $\upmu^{-1}$ near the shock,
as is suggested by the schematic\footnote{See \eqref{E:CARTESIANINTERMSOFGEOMETRIC} for the precise formula for $\partial_a$ in terms of the geometric vectorfields.} 
relation $\partial_a v^i \sim \upmu^{-1} \Rad v^i + \cdots$.
Nevertheless, by exploiting the \emph{transversality}\footnote{The transversality of $\Transport$ is critically important 
for this argument. For example if one integrates the same quantity $\rd_a v^i$ along the integral curve of the null vectorfield $\Lunit$
(which is \emph{tangent} to the acoustic characteristics),
then one can at best show that the integral is bounded by $\ln \upmu^{-1}$.} 
of the integral curves of $\Transport$ 
with the acoustic characteristics,
it can be shown\footnote{In proving this, one relies on the fact that 
the differential operator $\upmu \Transport$
can be viewed as
$
\displaystyle
\frac{d}{du}
$
along the integral curves of $\upmu \Transport$.} 
that the integral in question is uniformly bounded up to the shock, independent of $\upmu$. 
As a consequence, one can show that unlike an arbitrary Cartesian coordinate partial derivative of $v^i$,
$\Vortrenormalized^i$ remains uniformly bounded up to the shock. Moreover, as we discussed in the previous subsection, 
since $\upmu \Transport$ enjoys good commutation properties with the geometric vectorfields, 
the lower-order derivatives of $\Vortrenormalized^i$ with respect to the geometric vectorfields 
are also uniformly bounded in $L^{\infty}$.

In contrast, one cannot hope to obtain uniform $L^{\infty}$ bounds for
a general Cartesian coordinate partial derivatives of $\Vortrenormalized^i$ up to the shock.
This can easily be inferred from the specific vorticity equation \eqref{E:RENORMALIZEDVORTICTITYTRANSPORTEQUATION}, 
which states that 
$\Transport \Vortrenormalized^i$ is equal to 
$\Vortrenormalized^a \rd_a v^i$,
and we have already noted that $\rd_a v^i$ can blow up like $\upmu^{-1}$ near the shock.
Nevertheless, one can use equation \eqref{E:EVOLUTIONEQUATIONFLATCURLRENORMALIZEDVORTICITY} and an argument similar
to the one given in the previous paragraph to show that $\Flatcurl \Vortrenormalized$ is uniformly  
bounded in $L^{\infty}$ up to the shock!
That is, $\Flatcurl \Vortrenormalized$  behaves better than $\partial_{\alpha} \Vortrenormalized$! 
This argument crucially relies on the good null structure of the term $\mathscr{P}_{(\Vortrenormalized)}^i$
revealed by Theorem~\ref{T:STRONGNULL}, 
which in particular shows that the potentially damaging product
$(\Radunit \Vortrenormalized) \Radunit v$
(which is expected to be of size $\upmu^{-2}$)
is not present if one decomposes
RHS~\eqref{E:EVOLUTIONEQUATIONFLATCURLRENORMALIZEDVORTICITY}
relative to the geometric vectorfields.
Moreover, the geometric vectorfield derivatives of
$\Flatcurl \Vortrenormalized$
obey similar good $L^{\infty}$ and energy estimates,
except for at the top order (as we describe below).
These facts are extremely helpful in our approach, since, as we will later discuss, 
we need to combine these good below-top-order estimates for $\Vortrenormalized$
with appropriate top order estimate for $\Flatdiv \Vortrenormalized$ and $\Flatcurl \Vortrenormalized$
in order to obtain, via elliptic estimates,
control over the top order geometric derivatives of $\underline{\partial} \Vortrenormalized$,
where $\underline{\partial}$ denotes the spatial gradient with respect to the 
Cartesian coordinates.
Notice also that these estimates are relevant for controlling solutions to the wave equation
since $\Flatcurl \Vortrenormalized$ appears as a source term on RHS~\eqref{E:VELOCITYWAVEEQUATION}.
In fact,$\Flatcurl \Vortrenormalized$  
is the only top-order specific vorticity term featured in the wave equations.

\subsubsection{Easy error terms in the wave equation}
\label{SSS:EASYWAVEEQUATIONERRORTERMS}
	Most of the error terms on RHSs~\eqref{E:VELOCITYWAVEEQUATION}-\eqref{E:RENORMALIZEDDENSITYWAVEEQUATION}
		are easy to treat.
		In particular, 
		Christodoulou's framework \cite{dC2007} can easily be extended to treat
		the error terms on RHS \eqref{E:VELOCITYWAVEEQUATION}-\eqref{E:RENORMALIZEDDENSITYWAVEEQUATION}
		that verify the strong null condition;
		as we outlined in Subsubsect.~\ref{SSS:SOMEREMARKSONTHEPROOF}
		(see also the discussion in \cites{jS2014b,gHsKjSwW2016}),
		in the solution regime under consideration, 
		such error terms are essentially harmless and
		do not interfere with the shock formation processes.
		This means that the error terms 
		described by Theorem~\ref{T:STRONGNULL}
		are expected to have only a negligible effect on the dynamics,
		even near the shock. 

		The term
		$2 \exp(\Densrenormalized) \epsilon_{iab} (\Transport v^a) \Vortrenormalized^b$
		on RHS~\eqref{E:VELOCITYWAVEEQUATION} is also relatively easy to treat.
		The reason is that $\Vortrenormalized$ is a below-top-order factor
		that can be bounded by commuting
		the transport equation \eqref{E:RENORMALIZEDVORTICTITYTRANSPORTEQUATION}
		with the geometric vectorfields;
		as we mentioned at the end of
		Subsubsect.~\ref{SSS:QUICKSUMMARYOFPROOFOFSHOCKFORMATION},
		below-top-order terms exhibit less degenerate behavior 
		with respect to $\upmu$ compared to top-order factors.

	\subsubsection{The difficult specific vorticity term in the wave equation}
		\label{SSS:DIFFICULTVORTICITYTERMS}
		In contrast to the terms described in Subsubsect.~\ref{SSS:EASYWAVEEQUATIONERRORTERMS} 
		the factor $(\Flatcurl \Vortrenormalized)^i$
		on RHS~\eqref{E:VELOCITYWAVEEQUATION} is  
		a challenging top order factor that
		needs to be treated with the elliptic estimates mentioned in Subsubsect.~\ref{sss.inho.vor}.
		To explain why this is the case, let us count derivatives.
		Using only equation \eqref{E:RENORMALIZEDVORTICTITYTRANSPORTEQUATION},
		one can only conclude that $\Vortrenormalized$ has the same regularity
		as $\partial v$ (because of the factor $\partial v$ on RHS~\eqref{E:RENORMALIZEDVORTICTITYTRANSPORTEQUATION}),
		which suggests that $\Flatcurl \Vortrenormalized$ has the same regularity as
		$\partial \partial v$. The key point is that \emph{this regularity is not compatible with the factor
		$(\Flatcurl \Vortrenormalized)^i$ being 
		on the right-hand side of the wave equation \eqref{E:VELOCITYWAVEEQUATION} for $v^i$};
		the wave equation energy estimates for $v^i$ yield only that 
		$\partial v$ has the same regularity as $\Flatcurl \Vortrenormalized$.
		That is, this approach leads to the loss of a derivative.

\subsubsection{Elliptic estimates for the specific vorticity near the shock}
		\label{SSS:ELLIPTICESTIMATESFORVORTICITY}
		In this subsubsection, we will sketch the main ideas of
		how to bound the top-order derivatives of $\Flatcurl \Vortrenormalized$
		and thus overcome the loss of a derivative mentioned in the previous subsubsection.
		At the same time, we will discuss how to bound the blowup-rate
		of the $L^2$ norm of these top derivatives; understanding the
		blowup-rates lies at the heart of understanding how to close 
		the energy estimates for the full system of equations.

		It turns out that we cannot control 
		the top-order derivatives of
		$\Flatcurl \Vortrenormalized$ in isolation; 
		due to the $\underline{\partial} \Vortrenormalized$-dependent terms
		on RHS~\eqref{E:VORTICITYNULLFORM},
		we can control them only by obtaining control of the top-order derivatives of
		$\underline{\partial} \Vortrenormalized$,
		which requires elliptic estimates on $\Sigma_t$;
		we recall that here and throughout, 
		$\underline{\partial}$ denotes the spatial gradient with respect to the 
		Cartesian coordinates.
		To proceed, we let $N_{Top}$ denote the maximum number\footnote{At the end of this subsubsection,
		we will explain why $N_{Top}$ is the same as in the irrotational case.} 
		of times that we need to commute the wave equations 
		\eqref{E:VELOCITYWAVEEQUATION}-\eqref{E:RENORMALIZEDDENSITYWAVEEQUATION}
		to close the estimates
		and we let $\Fullset^{N_{Top}}$
		denote an arbitrary $N_{Top}^{th}-$order differential operator corresponding to
		repeated differentiation with respect to the elements\footnote{Recall that in 
		Subsubsect.~\ref{SSS:QUICKSUMMARYOFPROOFOFSHOCKFORMATION},
		when deriving energy estimates, 
		we were able to close the estimates by commuting the equations only with the elements $P$ of the $\mathcal{P}_u-$tangential
		set $\Tanset$. Here, to keep the discussion simple, we ignore this detail and allow for commutations
		with all elements of $\Fullset$.} 
		of $\Fullset$.
		Then to close the top-order wave equation energy estimates,
		we must control the wave equation source term 
		$
		\Flatcurl \Fullset^{N_{Top}} \Vortrenormalized
		$
		in $L^2$.
		Examining the right-hand side of the evolution equation\footnote{In reality, to close the estimates,
		one must study the $\upmu-$weighted version of this equation,
		but we ignore this detail here.}
		\eqref{E:EVOLUTIONEQUATIONFLATCURLRENORMALIZEDVORTICITY}
		for the scalar function $\Flatcurl \Vortrenormalized^i$,
		we see that when deriving $L^2$ estimates for $\Flatcurl \Fullset^{N_{Top}} \Vortrenormalized^i$,
		we must bound error terms depending on the $L^2$ norms of
		$
		\lbrace \underline{\partial} \Fullset^{N_{Top}} \Vortrenormalized^j \rbrace_{j=1,2,3}
		$.
		Hence, to close the estimates, we must use 
		equation \eqref{E:FLATDIVOFRENORMALIZEDVORTICITY}
		to obtain bounds for 
		$\Flatdiv \Fullset^{N_{Top}} \Vortrenormalized$
		and then employ elliptic estimates
		to bound the $L^2$ norms of
		$
		\lbrace \underline{\partial} \Fullset^{N_{Top}} \Vortrenormalized^j \rbrace_{j=1,2,3}
		$.

			The elliptic estimates that we need
			are similar to the standard Cartesian elliptic estimates along the constant-time hypersurfaces $\Sigma_t$
			and can be caricatured as follows:
			\begin{align} \label{E:ELLIPTICCARICATURE}
			\displaystyle
			\left\| 
				\sqrt{\upmu} \underline{\partial} \Fullset^{N_{Top}} \Vortrenormalized
			\right\|_{L^2(\Sigma_t)}
			\lesssim
			\left\| 
				\sqrt{\upmu} \Flatdiv \Fullset^{N_{Top}} \Vortrenormalized
			\right\|_{L^2(\Sigma_t)}
			+
			\left\| 
				\sqrt{\upmu} \Flatcurl \Fullset^{N_{Top}} \Vortrenormalized
			\right\|_{L^2(\Sigma_t)}
			+ \cdots.
			\end{align}
			Above, the $L^2$ norms are defined relative to a measure that,
			roughly speaking, is equal to 
			$du d \vartheta^1 d \vartheta^2$
			where $u$, $\vartheta^1$, and $\vartheta^2$ are the geometric coordinates on $\Sigma_t$;
			see also Footnote~\ref{FN:NOTWRITINGFORMS}.
			In view of the relation\footnote{The estimate \eqref{E:VOLFORMCOMPARISON}
			can be derived by analyzing the change of variables map between geometric
			and rectangular coordinates.
			That is, \eqref{E:VOLFORMCOMPARISON} follows from proving
			precise versions of the following heuristic statements,
			which are suggested by Figure~\ref{F:FRAME}
			(which corresponds to the nearly plane symmetric solutions under consideration):
			at a fixed $t$, we have
			$dx^1 \sim - \upmu du$,
			$d x^2 \sim d \vartheta^1$,
			$d x^3 \sim d \vartheta^2$.}
			\begin{align} \label{E:VOLFORMCOMPARISON}
				d^3 x \approx \upmu \, du d \vartheta^1 d \vartheta^2,
			\end{align}
			where $d^3 x$ is the standard Euclidean volume form on $\Sigma_t$,
			it follows that indeed, \eqref{E:ELLIPTICCARICATURE}
			is essentially equivalent to\footnote{In fact, it is perhaps preferable
			to derive the estimate
			\eqref{E:ELLIPTICCARICATURE} relative to the Cartesian spatial coordinates and volume
			form $d^3 x$.} 
			the standard Cartesian elliptic estimates along the constant-time hypersurfaces $\Sigma_t$.
			The need for elliptic estimates along $\Sigma_t$
			represents a new difficulty not found in\footnote{As we have mentioned, this difficulty is also absent in 
			the case of two spatial dimensions, even in the presence of vorticity. \label{FN:TWOSPACENOELLITPIC}} 
			previous works on 
			shock formation in three spatial dimensions.
			Most importantly, while the elliptic estimates 
			are crucial from the point of view of regularity, 
			they are based on treating \emph{all} spatial derivatives of $\Vortrenormalized$ 
			on equal footing. This clearly clashes with the philosophy of
			trying to obtain less singular estimates for the derivatives of $\Vortrenormalized$
			in directions tangent to the acoustic characteristics
			and calls into question the usefulness of the
			good null structure in the equations.
			The net effect is that top-order estimates for $\Vortrenormalized$
			are burdened with a factor of $1/\upmu$,
			which leads to degenerate top-order bounds for $\Vortrenormalized$.
			However, by exploiting the non-degenerate nature of
			an energy for
			$\Flatcurl \Fullset^{N_{Top}} \Vortrenormalized$
			along the acoustic characteristics,
			we are able to show that the 
			degeneracy created by this factor is not too severe;
			see the discussion below inequality \eqref{E:ATLEASTASDEGENERATE}
			concerning the small constant $c$.
			In fact, we are able to show that the degeneracy 
			corresponding to the elliptic estimates for $\Vortrenormalized$
			is much less severe than the analogous difficulty 
			that we encountered in the energy estimate \eqref{E:CARICENERGYINEQ}
			in the irrotational case,
			where the difficult factor of 
			$1/\upmu$ is tied to the difficult top-order regularity
			properties of the eikonal function.

			The basic strategy for controlling
			$\Flatcurl \Fullset^{N_{Top}} \Vortrenormalized$
			is to integrate the $\Fullset^{N_{Top}}-$commuted evolution equation
			\eqref{E:EVOLUTIONEQUATIONFLATCURLRENORMALIZEDVORTICITY}
			to bound, via a Gronwall estimate,
			$\Flatcurl \Fullset^{N_{Top}} \Vortrenormalized$
			in terms of simple error terms 
			and more difficult error terms that, by the elliptic estimates \eqref{E:ELLIPTICCARICATURE}
			and the $\Fullset^{N_{Top}}-$commuted equation \eqref{E:FLATDIVOFRENORMALIZEDVORTICITY}
			for\footnote{Unlike the estimates for $\Flatcurl \Fullset^{N_{Top}} \Vortrenormalized$,
			the estimates for $\Flatdiv \Fullset^{N_{Top}} \Vortrenormalized$ are easy to derive.} 
			$\Flatdiv \Fullset^{N_{Top}} \Vortrenormalized$,
			can be controlled back in terms of
			$\Flatcurl \Fullset^{N_{Top}} \Vortrenormalized$.
			As we mentioned above, a difficult factor of $1/\upmu$ 
			(more precisely,  $1/\upmu_{\star}$)
			is present in the inequality for $\Flatcurl \Fullset^{N_{Top}} \Vortrenormalized$
			that one needs to treat with Gronwall's inequality.
			In total, the integral inequality satisfied by $\Flatcurl \Fullset^{N_{Top}} \Vortrenormalized$ 
			can be caricatured as follows:
			\begin{align} \label{E:TOPORDERCURLCARICATURE}
			\left\| 
				\sqrt{\upmu} \Flatcurl \Fullset^{N_{Top}} \Vortrenormalized
			\right\|_{L^2(\Sigma_t)}^2
			& \leq
			\mbox{\upshape data}
			+
				\bar{A}
				\int_{s=0}^t
					\frac{1}{\upmu_{\star}(s)}
					\left\| 
						\sqrt{\upmu} \underline{\partial} \Fullset^{N_{Top}} \Vortrenormalized
					\right\|_{L^2(\Sigma_s)}^2
				\, ds
				+ \cdots
					\\
		& \leq
			\mbox{\upshape data}
			+
			\widetilde{A}
			\int_{s=0}^t
				\frac{1}{\upmu_{\star}(s)}
				\left\| 
					\sqrt{\upmu} \Flatcurl \Fullset^{N_{Top}} \Vortrenormalized
				\right\|_{L^2(\Sigma_s)}^2
			\, ds
			+ \cdots,
			\notag
		\end{align}
		where (as before)
		\[
		\upmu_{\star}(s) = \min_{\Sigma_s} \upmu,
		\]
		$\cdots$ denotes easier error terms, 
		and $\widetilde{A} > 0$ is a \emph{small}\footnote{In deriving \eqref{E:TOPORDERCURLCARICATURE},
		one can apply, to the main error integral involving the top order derivatives of $\Vortrenormalized$,
		Young's inequality in the form $ab \lesssim a^2 + b^2$
		to ensure that the constant $\widetilde{A}$ is small.
		This procedure also generates an error integral with a large constant, 
		but it can be suitably controlled with
		a non-degenerate energy for $\Flatcurl \Fullset^{N_{Top}} \Vortrenormalized$
		on the null hypersurface
		$\mathcal{P}_u^t$,
		which we have suppressed from
		the inequality \eqref{E:TOPORDERCURLCARICATURE}.} 
		constant.
		We now note that for reasons 
		similar to the ones given below inequality
		\eqref{E:CARICENERGYINEQ},
		the factor
		$
		\displaystyle
		\frac{1}{\upmu_{\star}}
		$
		on RHS~\eqref{E:TOPORDERCURLCARICATURE}
		leads to a Gronwall estimate for
		$
		\displaystyle
			\left\| 
				\sqrt{\upmu} \Flatcurl \Fullset^{N_{Top}} \Vortrenormalized
			\right\|_{L^2(\Sigma_t)}^2
	$
	that is \emph{at least as degenerate} as
	\begin{align} \label{E:ATLEASTASDEGENERATE}
			\left\| 
				\sqrt{\upmu} \Flatcurl \Fullset^{N_{Top}} \Vortrenormalized
			\right\|_{L^2(\Sigma_t)}^2
			\lesssim
			\mbox{\upshape Data}
			\cdot
			\upmu_{\star}^{- c}(t)
			+ \cdots,
	\end{align}
	where $c > 0$ is a small constant that is controlled by 
	the small constant $\widetilde{A}$.
	Since $c$ is small,
	this shows that the need to carry out elliptic estimates
	for the specific vorticity at the top order does not, in itself, 
	lead to drastically degenerate top-order
	estimates for $\Vortrenormalized$. In fact, the most degenerate
	terms are hiding in the terms $\cdots$ on RHS~\eqref{E:ATLEASTASDEGENERATE}.
	More precisely, included in $\cdots$
	are terms depending on the top-order derivatives
	of the wave variables $v^i$ and $\Densrenormalized$,
	whose energies can blow up like a large power of $\upmu_{\star}^{-1}$
	for the same reason as in the irrotational case.
	These terms in fact make the main contribution to the 
	top-order energy blowup-rate for $\Vortrenormalized$.
	A closely related fact is that the energy blowup-rates 
	for $\Vortrenormalized$ are compatible
	with the \emph{same} energy blowup-rates for the  
	wave variables that one derives
	in the irrotational case.
	Indeed, the 
	high-order energy blowup-rates 
	of the wave variables 
	are not affected by the presence of small amounts of
	vorticity in the problem.

		We also note that to close these estimates, 
		we must use the fact that $\Flatcurl \Vortrenormalized$ 
		satisfies, at all lower-order derivatives,
		better estimates than a general spatial derivative of $\Vortrenormalized$,
		as we described in Subsubsect.~\ref{SSS:DIFFICULTVORTICITYTERMS}.
		This fact is needed to control various lower-order error terms on
		RHS~\eqref{E:TOPORDERCURLCARICATURE}, which we have relegated to the terms $\cdots$.
		For this purpose, it is crucial that the term $\mathscr{P}_{(\Vortrenormalized)}^i$
		on the right-hand side of \eqref{E:EVOLUTIONEQUATIONFLATCURLRENORMALIZEDVORTICITY} 
		verifies the strong null condition (see Theorem~\ref{T:STRONGNULL}).

	Finally, we note that below the top derivative level, 
	we can avoid the elliptic estimates for $\Vortrenormalized$.
	The price one pays is that 
	the error terms in the energy estimates for
	$\Vortrenormalized$ depend on the derivatives of
	$\Densrenormalized$ and $v^i$ at one higher derivative level,
	which is permissible below top order.
	The gain is that we can prove less singular estimates (in terms of powers of $1/\upmu_{\star}$)
	below top order. That is, avoiding elliptic estimates allows us to employ
	an energy ``descent scheme'' similar to the one we discussed 
	in Subsubsect.~\ref{SSS:QUICKSUMMARYOFPROOFOFSHOCKFORMATION},
	which applies in particular to irrotational solutions of the compressible Euler equations.
	Eventually, one reaches a level below which all energies, 
	including those for $\Densrenormalized$, $v^i$, and $\Vortrenormalized$
	are bounded, much like in the irrotational case.
	This is the key technical step in closing the proof of shock formation.

\section{Proof of Theorem~\ref{T:GEOMETRICWAVETRANSPORTSYSTEM}}
\label{E:PROOFOFMAINTHEOREM}
In this section, we prove Theorem~\ref{T:GEOMETRICWAVETRANSPORTSYSTEM}.
The theorem is a conglomeration of
Lemmas~\ref{L:RENORMALIZEDVORTICITYEVOLUTIONEQUATION},
\ref{L:WAVEEQUATIONFORLOGDENSITY},
\ref{L:WAVEEQUATIONFORV},
and
\ref{L:DIVANDCURLEQUATIONS},
in which we separately derive the equations stated in the theorem.

\subsection{Proof of Theorem~\ref{T:GEOMETRICWAVETRANSPORTSYSTEM}}
\label{SS:PROOFOFMAINTHEOREM}
We start by deriving the well-known evolution equation 
\eqref{E:RENORMALIZEDVORTICTITYTRANSPORTEQUATION} for 
$\Vortrenormalized$.

\begin{lemma}[\textbf{Transport equation for} $\Vortrenormalized$]
\label{L:RENORMALIZEDVORTICITYEVOLUTIONEQUATION}
The compressible Euler equations 
\eqref{E:TRANSPORTDENSRENORMALIZEDRELATIVETORECTANGULAR}-\eqref{E:TRANSPORTVELOCITYRELATIVETORECTANGULAR}
imply the following evolution equation for the modified
vorticity vectorfield $\Vortrenormalized$ from Def.~\ref{D:MODIFIEDVARIABLES}:
\begin{align} \label{E:RESTATEDRENORMALIZEDVORTICITYEVOLUTIONEQUATION}
	\Transport \Vortrenormalized^i
	& = \Vortrenormalized^a \partial_a v^i.
\end{align}
\end{lemma}

\begin{proof}
In view of definition \eqref{E:MODIFIEDVARIABLES},
we commute equation \eqref{E:TRANSPORTVELOCITYRELATIVETORECTANGULAR}
with the operator 
$
\displaystyle
\frac{1}
{\exp \Densrenormalized}
\Flatcurl
$,
note that $\Flatcurl$ completely annihilates 
RHS~\eqref{E:TRANSPORTVELOCITYRELATIVETORECTANGULAR}
(which can be written as a perfect gradient $\partial_i (\cdots)$),
and use equation \eqref{E:TRANSPORTDENSRENORMALIZEDRELATIVETORECTANGULAR}
and the antisymmetry of $\epsilon_{\dots}$
to deduce
\begin{align} \label{E:FIRSTSTEPRENORMALIZEDVORTICITYEVOLUTIONEQUATION}
	\Transport \Vortrenormalized^i
	& = - \frac{1}{\exp \Densrenormalized} \epsilon_{iab} (\partial_a v^c) \partial_c v^b
	- \frac{1}{\exp \Densrenormalized} (\Transport \Densrenormalized) \omega^i
		\\
	& = - \frac{1}{\exp \Densrenormalized} \epsilon_{iab} (\partial_a v^c) \partial_c v^b
	+ (\Flatdiv v) \Vortrenormalized^i
	\notag \\
	& = - \frac{1}{\exp \Densrenormalized} 
	\epsilon_{iab} (\partial_a v^c) (\partial_c v^b - \partial_b v^c)
	+ (\Flatdiv v) \Vortrenormalized^i
	\notag \\
	& = - \epsilon_{iab} \epsilon_{cbd} \Vortrenormalized^d (\partial_a v^c) 
	+ (\Flatdiv v) \Vortrenormalized^i
	\notag \\
	& = \epsilon_{iab} \epsilon_{cdb} \Vortrenormalized^d (\partial_a v^c) 
	+ (\Flatdiv v) \Vortrenormalized^i
	\notag \\
	& = (\delta_{ic} \delta_{ad} - \delta_{id} \delta_{ac}) \Vortrenormalized^d (\partial_a v^c) 
	+ (\Flatdiv v) \Vortrenormalized^i.
	\notag 
\end{align}
Clearly, we have
$\mbox{{\upshape RHS}~\eqref{E:FIRSTSTEPRENORMALIZEDVORTICITYEVOLUTIONEQUATION}}
=
\mbox{{\upshape RHS}~\eqref{E:RESTATEDRENORMALIZEDVORTICITYEVOLUTIONEQUATION}}
$
as desired.
\end{proof}

Recall that the covariant wave operator $\square_g$
is defined in Def.~\ref{E:WAVEOPERATORARBITRARYCOORDINATES}.
In the next lemma, we provide an explicit expression for
$\square_g \phi$ that holds relative to the Cartesian coordinates.

\begin{lemma}[$\square_g$ \textbf{relative to the Cartesian coordinates}]
	\label{L:COVARIANTWAVEOPRELATIVETOCARTESIAN}
	The covariant wave operator $\square_g$ acts on scalar functions $\phi$
	via the following identity, where
	RHS~\eqref{E:COVARIANTWVAEOPERATORINRECTANGULARCOORDINATES}
	is expressed in Cartesian coordinates:
	\begin{align} \label{E:COVARIANTWVAEOPERATORINRECTANGULARCOORDINATES}
	\square_g \phi
	& = - \Transport \Transport \phi
		+ \Speed^2 \delta^{ab} \partial_a \partial_b \phi
		+ 2 \Speed^{-1} \Speed' (\Transport \Densrenormalized) \Transport \phi
		- (\partial_a v^a) \Transport \phi
		- \Speed^{-1} \Speed' (g^{-1})^{\alpha \beta} (\partial_{\alpha} \Densrenormalized) \partial_{\beta} \phi.
\end{align}
\end{lemma}
\begin{proof}
It is straightforward to compute using equations 
\eqref{E:ACOUSTICALMETRIC}-\eqref{E:INVERSEACOUSTICALMETRIC}
that relative to Cartesian coordinates, we have
\begin{align} \label{E:DETG}
	\mbox{\upshape det} g
	& = - \Speed^{-6}
\end{align}
and hence
\begin{align} \label{E:ROOTDETGTIMESGINVERSE}
	\sqrt{|\mbox{\upshape det} g|} g^{-1} 
	& = 
		- \Speed^{-3} \Transport \otimes \Transport
		+
		\Speed^{-1} \sum_{a=1}^3 \partial_a \otimes \partial_a.
\end{align}
Using \eqref{E:WAVEOPERATORARBITRARYCOORDINATES},
\eqref{E:DETG}, and
\eqref{E:ROOTDETGTIMESGINVERSE}, 
we compute that
\begin{align} \label{E:FIRSTFORMULACOVARIANTWVAEOPERATORINRECTANGULARCOORDINATES}
	\square_g \phi
	 & = 
	- 
	\Speed^3
	\left(
		\Transport^{\alpha} \partial_{\alpha} (\Speed^{-3})
	\right)
	\Transport^{\beta} \partial_{\beta} \phi
	- 
	(\partial_{\alpha} \Transport^{\alpha}) \Transport^{\beta} \partial_{\beta} \phi
	- 
	(\Transport^{\alpha} \partial_{\alpha} \Transport^{\beta}) \partial_{\beta} \phi
		\\
&  \ \
	-
	\Transport^{\alpha} \Transport^{\beta} \partial_{\alpha} \partial_{\beta}
	\phi
	+ 
	\Speed^2 \delta^{ab} \partial_a \partial_b \phi
	- 
	\Speed \Speed' \delta^{ab} (\partial_a \Densrenormalized) \partial_b \phi.
	\notag
\end{align}
Finally, from \eqref{E:FIRSTFORMULACOVARIANTWVAEOPERATORINRECTANGULARCOORDINATES},
the expression \eqref{E:MATERIALVECTORVIELDRELATIVETORECTANGULAR} for $\Transport$, 
the expression \eqref{E:INVERSEACOUSTICALMETRIC} for $g^{-1}$,
and simple calculations, we arrive at \eqref{E:COVARIANTWVAEOPERATORINRECTANGULARCOORDINATES}.
\end{proof}

In the next lemma, we derive equation \eqref{E:RENORMALIZEDDENSITYWAVEEQUATION}.

\begin{lemma}[\textbf{Wave equation for} $\Densrenormalized$]
\label{L:WAVEEQUATIONFORLOGDENSITY}
The compressible Euler equations 
\eqref{E:TRANSPORTDENSRENORMALIZEDRELATIVETORECTANGULAR}-\eqref{E:TRANSPORTVELOCITYRELATIVETORECTANGULAR}
imply the following covariant wave equation for the
logarithmic density variable $\Densrenormalized$
from Def.~\ref{D:MODIFIEDVARIABLES}:
\begin{align} \label{E:PROOFRENORMALIZEDDENSITYWAVEEQUATION}
\square_g \Densrenormalized 
& = 
		- 
		3 \Speed^{-1} \Speed' 
		(g^{-1})^{\alpha \beta} \partial_{\alpha} \Densrenormalized \partial_{\beta} \Densrenormalized
		+ 
		2 \sum_{1 \leq a < b \leq 3}
			\left\lbrace
				\partial_a v^a \partial_b v^b
					-
				\partial_a v^b \partial_b v^a
			\right\rbrace.
\end{align}
\end{lemma}

\begin{proof}
First, using \eqref{E:COVARIANTWVAEOPERATORINRECTANGULARCOORDINATES}
with $\phi = \Densrenormalized$
and equation \eqref{E:TRANSPORTDENSRENORMALIZEDRELATIVETORECTANGULAR},
we compute that
\begin{align} \label{E:FIRSTCOMPUTATIONRENORMALIZEDDENSITYWAVEEQUATION}
	\square_g \Densrenormalized
	& = - \Transport \Transport \Densrenormalized
		+ 
		\Speed^2 \delta^{ab} \partial_a \partial_b \Densrenormalized
		 + 2 \Speed^{-1} \Speed' \Transport \Densrenormalized \Transport \Densrenormalized
		+ (\partial_a v^a)^2
		- \Speed^{-1} \Speed' 
			(g^{-1})^{\alpha \beta} 
			\partial_{\alpha} \Densrenormalized \partial_{\beta} \Densrenormalized.
\end{align}
Next, we use
\eqref{E:MATERIALVECTORVIELDRELATIVETORECTANGULAR},
\eqref{E:TRANSPORTDENSRENORMALIZEDRELATIVETORECTANGULAR},
and \eqref{E:TRANSPORTVELOCITYRELATIVETORECTANGULAR},
to compute that
\begin{align} \label{E:TWOTRANSPORTAPPLIEDTORENORMALIZEDDENSITYEXPRESSION}
	\Transport \Transport \Densrenormalized
	& = - \partial_a (\Transport v^a)
		+ (\partial_a v^b) \partial_b v^a 
			\\
	& = 
		\Speed^2 \delta^{ab} \partial_a \partial_b \Densrenormalized
		+
		\delta^{ab} (\partial_a \Speed^2)  \partial_b \Densrenormalized
		+ 
		(\partial_a v^b) \partial_b v^a
			\notag \\
	& = 
		\Speed^2 \delta^{ab} \partial_a \partial_b \Densrenormalized
		+
		2 \Speed \Speed' \delta^{ab} \partial_a \Densrenormalized \partial_b \Densrenormalized
		+ 
		(\partial_a v^b) \partial_b v^a.
			\notag 
\end{align}
Finally, using \eqref{E:TWOTRANSPORTAPPLIEDTORENORMALIZEDDENSITYEXPRESSION} 
to substitute for the term
$- \Transport \Transport \Densrenormalized$ on RHS~\eqref{E:FIRSTCOMPUTATIONRENORMALIZEDDENSITYWAVEEQUATION}
and using the identities
\begin{align} \label{E:VELOCITYNULLFORMIDENTITY}
(\partial_a v^a)^2 
-
(\partial_a v^b) \partial_b v^a
=
2 
\sum_{1 \leq a < b \leq 3}
	\left\lbrace
		\partial_a v^a \partial_b v^b
		-
	\partial_a v^b \partial_b v^a
	\right\rbrace
\end{align}
and
$
\Transport \Densrenormalized \Transport \Densrenormalized
-
\Speed^2 \delta^{ab} \partial_a \Densrenormalized \partial_b \Densrenormalized
= - (g^{-1})^{\alpha \beta} \partial_{\alpha} \Densrenormalized \partial_{\beta} \Densrenormalized
$
(see \eqref{E:INVERSEACOUSTICALMETRIC}),
we arrive at the desired expression 
\eqref{E:PROOFRENORMALIZEDDENSITYWAVEEQUATION}.
\end{proof}

We now establish equation \eqref{E:VELOCITYWAVEEQUATION}.

\begin{lemma}[\textbf{Wave equation for} $v^i$]
\label{L:WAVEEQUATIONFORV}
The compressible Euler equations 
\eqref{E:TRANSPORTDENSRENORMALIZEDRELATIVETORECTANGULAR}-\eqref{E:TRANSPORTVELOCITYRELATIVETORECTANGULAR}
imply the following covariant wave equation for the
scalar-valued function $v^i$, $i=1,2$:
\begin{align} \label{E:PROOFVELOCITYWAVEEQUATION}
\square_g v^i
& = - \Speed^2 \exp(\Densrenormalized)  (\Flatcurl \Vortrenormalized)^i 
			+
			2 \exp(\Densrenormalized) \epsilon_{iab} (\Transport v^a) \Vortrenormalized^b
				\\
& \ \
			- (1+\Speed^{-1} \Speed') (g^{-1})^{\alpha \beta} \partial_{\alpha} \Densrenormalized \partial_{\beta} v^i.
				\notag
\end{align}
\end{lemma}

\begin{proof}
First, we use \eqref{E:COVARIANTWVAEOPERATORINRECTANGULARCOORDINATES}
with $\phi = v^i$ 
and equation \eqref{E:TRANSPORTVELOCITYRELATIVETORECTANGULAR}
to deduce
\begin{align} \label{E:FIRSTCOMPUTATIONVELOCITYWAVEEQUATION}
	\square_g v^i
	& = - \Transport \Transport v^i
		+ \Speed^2 \delta^{ab} \partial_a \partial_b v^i
		- 2 \Speed \Speed' \Transport \Densrenormalized \delta^{ia} \partial_a \Densrenormalized
		- (\partial_a v^a) \Transport v^i
		- \Speed^{-1} \Speed' (g^{-1})^{\alpha \beta} \partial_{\alpha} \Densrenormalized \partial_{\beta} v^i.
\end{align}
Next, we use
\eqref{E:MATERIALVECTORVIELDRELATIVETORECTANGULAR},
\eqref{E:TRANSPORTDENSRENORMALIZEDRELATIVETORECTANGULAR},
\eqref{E:TRANSPORTVELOCITYRELATIVETORECTANGULAR},
and
\eqref{E:VORTICITYDEFINITION}
to compute that
\begin{align} \label{E:TWOTRANSPORTAPPLIEDTOVELOCITYEXPRESSION}
\Transport \Transport v^i
	& = - \Speed^2 \delta^{ia} \Transport \partial_a \Densrenormalized
		- 2 \Speed \Speed' \Transport \Densrenormalized \delta^{ia} \partial_a \Densrenormalized
			\\
	& = - \Speed^2 \delta^{ia} \partial_a (\Transport \Densrenormalized)
		+ \Speed^2 \delta^{ia} \partial_a v^b \partial_b \Densrenormalized
		- 2 \Speed \Speed' \Transport \Densrenormalized \delta^{ia} \partial_a \Densrenormalized
		\notag \\
	& = \Speed^2 \delta^{ia} \delta_c^b \partial_a (\partial_b v^c)
		- \delta^{ia}  \partial_a v^b \Transport v^b
		- 2 \Speed \Speed' \Transport \Densrenormalized \delta^{ia} \partial_a \Densrenormalized
		\notag \\
	& = \Speed^2 \delta^{bc} \partial_b \partial_c v^i
		+ \Speed^2 \delta^{ia} \partial_c (\partial_a v^c - \partial_c v^a)
		- \delta^{ia}  \partial_a v^b \Transport v^b
		- 2 \Speed \Speed' \Transport \Densrenormalized \delta^{ia} \partial_a \Densrenormalized
			\notag \\
 & = \Speed^2 \delta^{bc} \partial_b \partial_c v^i
		+ \Speed^2 \delta^{ia} \partial_c (\partial_a v^c - \partial_c v^a)
		- (\partial_i v^b - \partial_b v^i) \Transport v^b
		- \partial_a v^i \Transport v^a
		- 2 \Speed \Speed' \Transport \Densrenormalized \delta^{ia} \partial_a \Densrenormalized.
		\notag
\end{align}
Next, we use the identity (see \eqref{E:MODIFIEDVARIABLES})
\begin{align} \label{E:VORTICIITYRENORMALIZEDVORTICITYANDDENSITYRELATION}
	\omega^i 
	& =
	\Vortrenormalized^i \exp(\Densrenormalized)
\end{align}
and equation \eqref{E:TRANSPORTVELOCITYRELATIVETORECTANGULAR}
to derive the identities
\begin{align}
\Speed^2 \delta^{ia} \partial_c (\partial_a v^c - \partial_c v^a)
	& = \Speed^2 \delta^{ia} \epsilon_{acd} \partial_c \omega^d
	= \Speed^2 \Flatcurl \omega^i
		\label{E:FIRSTIDUSEDINDERIVINGVELOCITYWAVEEQN}
		\\
	& = \Speed^2 \exp(\Densrenormalized)  \Flatcurl \Vortrenormalized^i
		+
		\Speed^2 \exp(\Densrenormalized) 
			\epsilon_{icd} \Vortrenormalized^d \partial_c \Densrenormalized
	\notag \\
	& =
		\Speed^2 \exp(\Densrenormalized)  \Flatcurl \Vortrenormalized^i 
			-
			\exp(\Densrenormalized) \epsilon_{iab} (\Transport v^a) \Vortrenormalized^b,
			\notag \\
	(\partial_i v^b - \partial_b v^i) \Transport v^b
 & 
=  \exp(\Densrenormalized) \epsilon_{ibc} \Vortrenormalized^c \Transport v^b
= \exp(\Densrenormalized) (\Transport v^a) \epsilon_{iab} \Vortrenormalized^b,
\label{E:SECONDIDUSEDINDERIVINGVELOCITYWAVEEQN}
\end{align}
Substituting the RHSs of 
\eqref{E:FIRSTIDUSEDINDERIVINGVELOCITYWAVEEQN}-\eqref{E:SECONDIDUSEDINDERIVINGVELOCITYWAVEEQN}
for the relevant terms on RHS~\eqref{E:TWOTRANSPORTAPPLIEDTOVELOCITYEXPRESSION},
we obtain
\begin{align} \label{E:TWOTRANSPORTAPPLIEDTOVELOCITYSECONDEXPRESSION}
	\Transport \Transport v^i
	& = \Speed^2 \delta^{bc} \partial_b \partial_c v^i
			+ 
			\Speed^2 \exp(\Densrenormalized)  \Flatcurl \Vortrenormalized^i 
			-
			2 \exp(\Densrenormalized) \epsilon_{iab} (\Transport v^a) \Vortrenormalized^b
			\\
	& \ \
		- (\partial_a v^i) \Transport v^a
		- 2 \Speed \Speed' \Transport \Densrenormalized \delta^{ia} \partial_a \Densrenormalized.
		\notag
\end{align}
Next, substituting $-$RHS~\eqref{E:TWOTRANSPORTAPPLIEDTOVELOCITYSECONDEXPRESSION}
for the term
$- \Transport \Transport v^i$ 
on RHS~\eqref{E:FIRSTCOMPUTATIONVELOCITYWAVEEQUATION},
we arrive at 
\begin{align}\label{first.thm.prelim.1}
\square_g v^i
& = - \Speed^2 \exp(\Densrenormalized)  (\Flatcurl \Vortrenormalized)^i 
			+
			2 \exp(\Densrenormalized) \epsilon_{iab} (\Transport v^a) \Vortrenormalized^b
				\\
& \ \
			+
			\left\lbrace
				(\Transport v^a) \partial_a v^i
					-
				(\partial_a v^a) \Transport v^i
			\right\rbrace
			- \Speed^{-1} \Speed' (g^{-1})^{\alpha \beta} \partial_{\alpha} \Densrenormalized \partial_{\beta} v^i.
				\notag
\end{align}
To handle the terms $\lbrace \cdot \rbrace$ in \eqref{first.thm.prelim.1}, we use 
\eqref{E:TRANSPORTDENSRENORMALIZEDRELATIVETORECTANGULAR}, 
\eqref{E:TRANSPORTVELOCITYRELATIVETORECTANGULAR},
and \eqref{E:INVERSEACOUSTICALMETRIC} 
to obtain
\begin{equation}\label{first.thm.prelim.2}
\begin{split}
(\Transport v^a) \partial_a v^i
					-
				(\partial_a v^a) \Transport v^i
=- \Speed^2 \delta^{ab} (\partial_b \Densrenormalized) \partial_a v^i
	+
	(\Transport \Densrenormalized) \Transport v^i
= -(g^{-1})^{\alpha \beta} (\partial_{\alpha} \Densrenormalized) \partial_{\beta} v^i.
\end{split}
\end{equation}
Finally, substituting \eqref{first.thm.prelim.2} into \eqref{first.thm.prelim.1}, we conclude the desired equation \eqref{E:PROOFVELOCITYWAVEEQUATION}.

\end{proof}

	We now establish equations 
	\eqref{E:FLATDIVOFRENORMALIZEDVORTICITY}-\eqref{E:EVOLUTIONEQUATIONFLATCURLRENORMALIZEDVORTICITY}.

\begin{lemma}[\textbf{Equations for} $\Flatdiv \Vortrenormalized$ \textbf{and}
	$\Flatcurl \Vortrenormalized$]
\label{L:DIVANDCURLEQUATIONS}
The compressible Euler equations 
\eqref{E:TRANSPORTDENSRENORMALIZEDRELATIVETORECTANGULAR}-\eqref{E:TRANSPORTVELOCITYRELATIVETORECTANGULAR}
imply the following equation for
$\Flatdiv \Vortrenormalized$
and transport equation for
the scalar-valued function $(\Flatcurl \Vortrenormalized)^i$:
\begin{subequations}
	\begin{align} \label{E:PROOFFLATDIVOFRENORMALIZEDVORTICITY}
	\Flatdiv \Vortrenormalized
	& = - \Vortrenormalized^a \partial_a \Densrenormalized,
		\\
	\Transport (\Flatcurl \Vortrenormalized)^i
	& = (\exp \Densrenormalized) \Vortrenormalized^a \partial_a \Vortrenormalized^i
			- (\exp \Densrenormalized) \Vortrenormalized^i \Flatdiv \Vortrenormalized
			+ 
			\epsilon_{iab} (\partial_a \Vortrenormalized^c) \partial_c v^b
			- \epsilon_{iab} (\partial_a v^c) \partial_c \Vortrenormalized^b.
			\label{E:PROOFEVOLUTIONEQUATIONFLATCURLRENORMALIZEDVORTICITY}
\end{align}
\end{subequations}

\end{lemma}

\begin{proof}
	Since $\omega = \Flatcurl v$, it follows that
	$\Flatdiv \omega = 0$. Equation \eqref{E:PROOFFLATDIVOFRENORMALIZEDVORTICITY}
	follows from this identity,
	\eqref{E:VORTICIITYRENORMALIZEDVORTICITYANDDENSITYRELATION},
	and simple calculations.

	We now derive \eqref{E:PROOFEVOLUTIONEQUATIONFLATCURLRENORMALIZEDVORTICITY}.
	Commuting the already established equation \eqref{E:RENORMALIZEDVORTICTITYTRANSPORTEQUATION}
	with the Euclidean curl operator
	and using equations 
	\eqref{E:MATERIALVECTORVIELDRELATIVETORECTANGULAR},
	\eqref{E:VORTICIITYRENORMALIZEDVORTICITYANDDENSITYRELATION},
	and
	\eqref{E:PROOFFLATDIVOFRENORMALIZEDVORTICITY},
	we obtain the desired equation as follows:
	\begin{align} \label{E:FIRSTSTEPEVOLUTIONEQUATIONFLATCURLRENORMALIZEDVORTICITY}
	\Transport (\Flatcurl \Vortrenormalized)^i
	& 	= \Vortrenormalized^a \partial_a \omega^i
			+ \epsilon_{iab} (\partial_a \Vortrenormalized^c) \partial_c v^b
			- \epsilon_{iab} (\partial_a v^c) \partial_c \Vortrenormalized^b
				\\
	& = (\exp \Densrenormalized) \Vortrenormalized^a \partial_a \Vortrenormalized^i
			+ (\exp \Densrenormalized) \Vortrenormalized^i \Vortrenormalized^a \partial_a \Densrenormalized
			+ \epsilon_{iab} (\partial_a \Vortrenormalized^c) \partial_c v^b
			- \epsilon_{iab} (\partial_a v^c) \partial_c \Vortrenormalized^b
				\notag \\
	& = (\exp \Densrenormalized) \Vortrenormalized^a \partial_a \Vortrenormalized^i
			- (\exp \Densrenormalized) \Vortrenormalized^i \Flatdiv \Vortrenormalized
			+ \epsilon_{iab} (\partial_a \Vortrenormalized^c) \partial_c v^b
			- \epsilon_{iab} (\partial_a v^c) \partial_c \Vortrenormalized^b.
			\notag
\end{align}

We have therefore established \eqref{E:PROOFEVOLUTIONEQUATIONFLATCURLRENORMALIZEDVORTICITY},
which completes the proof of Theorem~\ref{T:GEOMETRICWAVETRANSPORTSYSTEM}.

\end{proof}

\bibliographystyle{amsalpha}
\bibliography{JBib}

\end{document}